\documentclass[10pt,notitlepage,reqno]{amsart}
\usepackage[utf8]{inputenc}

\usepackage{hyperref}
\usepackage[margin=1in]{geometry}
\usepackage{amsmath, amssymb, amsthm}
\usepackage[dvipsnames]{xcolor}
\usepackage[font=footnotesize]{caption}
\usepackage{subcaption}
\usepackage{bbm}
\usepackage{hhline}
\usepackage{booktabs}
\usepackage{multicol}
\usepackage{arydshln}
\usepackage{enumitem}
\usepackage{youngtab}
\usepackage{setspace}
\usepackage{amsfonts}
\usepackage{mathrsfs}
\usepackage{multirow}
\usepackage{wasysym}

\newtheorem{theorem}{Theorem}[section] 
\newtheorem*{theorem*}{Theorem}
\newtheorem{lemma}[theorem]{Lemma}
\newtheorem{proposition}[theorem]{Proposition}
\newtheorem*{proposition*}{Proposition}
\newtheorem{corollary}[theorem]{Corollary}
\newtheorem*{corollary*}{Corollary}
\newtheorem{notation}[theorem]{Notation}
\newtheorem{definition}[theorem]{Definition}

\theoremstyle{definition}
\newtheorem{remark}[theorem]{Remark}
\newtheorem{example}[theorem]{Example}

\numberwithin{equation}{section}

\usepackage{tikz}
\usetikzlibrary{decorations.markings,decorations.pathreplacing}
\usetikzlibrary{patterns}

\allowdisplaybreaks 
\newcommand{\up}{\big\uparrow}
\newcommand{\down}{\big\downarrow}
\newcommand{\Char}{\operatorname{Char}}
\newcommand{\Gal}{\operatorname{Gal}}
\newcommand{\image}{\operatorname{Im}}
\newcommand{\Irr}{\operatorname{Irr}}
\newcommand{\LR}{\operatorname{LR}}
\newcommand{\Syl}{\operatorname{Syl}}
\newcommand{\triv}{\mathbbm{1}}
\newcommand{\C}{\mathbb{C}}
\newcommand{\N}{\mathbb{N}}
\newcommand{\Q}{\mathbb{Q}}
\newcommand{\Z}{\mathbb{Z}}
\newcommand{\cB}{\mathcal{B}}
\newcommand{\cF}{\mathcal{F}}
\newcommand{\cG}{\mathcal{G}}
\newcommand{\cM}{\mathcal{M}}
\newcommand{\cN}{\mathcal{N}}
\newcommand{\cP}{\mathcal{P}}
\newcommand{\cT}{\mathcal{T}}
\newcommand{\cX}{\mathcal{X}}

\newcommand{\fh}{\mathsf{h}}
\newcommand{\fH}{\mathsf{H}}

\newcommand{\ft}{\mathsf{t}}
\newcommand{\fv}{\mathsf{v}}
\newcommand{\tworow}[2]{\ft_{#1}({#2})}
\newcommand{\hook}[2]{\fh_{#1}({#2})}
\newcommand{\punc}{{}^\circ}
\newcommand{\puncb}[2]{{}^\circ\cB_{#1}(#2)}
\newcommand{\puncp}[1]{{}^\circ\cP(#1)}
\newcommand{\close}[1]{{#1}^\bullet}
\newcommand{\val}{\mathcal{V}}
\newcommand{\newqed}{\hfill$\lozenge$}
\newcommand{\itemspace}[1]{\setlength\itemsep{#1}}
\begin{document}
	
\title[Sylow branching trees]{Sylow branching trees for symmetric groups}

\author{Eugenio Giannelli}
\address[E. Giannelli]{Dipartimento di Matematica e Informatica U.~Dini, Viale Morgagni 67/a, Firenze, Italy}
\email{eugenio.giannelli@unifi.it}

\author{Stacey Law}
\address[S. Law]{School of Mathematics, Watson Building, University of Birmingham, Edgbaston, Birmingham B15 2TT, UK} 
\email{s.law@bham.ac.uk}

\begin{abstract}
	Let $p\ge 5$ be a prime and let $P$ be a Sylow $p$-subgroup of a finite symmetric group $S_n$. To every irreducible character of $P$ we associate a collection of labelled, complete $p$-ary trees. 
	The main results of this article describe the positivity of Sylow branching coefficients for all irreducible characters of $P$ in terms of combinatorial properties of these trees, extending previous work on the linear characters of $P$.
\end{abstract}

\dedicatory{Dedicated to Gabriel Navarro on his 60th birthday.}



\maketitle


\section{Introduction}\label{sec:intro}

Let $G$ be a finite group, let $p$ be a prime number and let $P$ be a Sylow $p$-subgroup of $G$. 
The \emph{Sylow branching coefficients} of $G$ are the multiplicities $[\chi\big\downarrow_P, \phi]$, where $\chi$ and $\phi$ are any irreducible characters of $G$ and $P$ respectively. 
The study of the relationship between the representations of a group and the structure of its Sylow $p$-subgroups has been a central theme of research for decades.  
Most notably, Problem 12 of Brauer's celebrated article \cite{Brauer} asks how much information about $P$ can be obtained from the character table of $G$. 
More recently, character restriction to Sylow subgroups has been studied in connection with the McKay conjecture \cite{NTV14, GKNT17}, and to shed light on fields of values of irreducible characters of degree coprime to $p$ \cite{NT21, IN24}.

In contrast with this activity, specific information on Sylow branching coefficients is scarcely available in the literature. 
For this reason, the authors began an investigation of these coefficients for symmetric groups in recent years. 
In particular, in \cite{GL1} we determined the irreducible constituents of the permutation character $\triv_P\up^{S_n}$ for any $n\in\N$ and any odd prime $p$. Here $P$ denotes a Sylow $p$-subgroup of $S_n$ and $\triv_P$ its trivial character. 
In other words, we completely classified those $\chi\in\Irr(S_n)$ for which $[\chi\down_P, \triv_P]\neq 0$.
Later in \cite{GL2,L}, we extended this by considering all linear (i.e.~degree 1) characters of $P$, not just the trivial character.

Recently in \cite{GLLV}, we were able to use these results to confirm a conjecture proposed by Malle and Navarro in \cite{MN}, concerning the relation between the irreducible characters of any finite group and the algebraic structure of its Sylow subgroups. In particular, we characterise when $P$ is normal in $G$ in terms of Sylow branching coefficients.

The present article represents a further, significant advance in this line of research. We now generalise our previous work from considering only linear characters $\theta$ of $P$ to studying Sylow branching coefficients for \emph{all} irreducible characters $\theta$ of $P$.
Given $\theta\in\Irr(P)$, we let $\Omega(\theta)$ be the subset of $\Irr(S_n)$ consisting of the irreducible constituents of $\theta\up^{S_n}$.
Since the irreducible characters of $S_n$ are naturally labelled by partitions of $n$, we can equivalently think of the set $\Omega(\theta)$ as a subset of the set $\cP(n)$ of partitions of $n$ in the following way:
\[ \Omega(\theta) := \{\lambda\in\cP(n) \mid [\chi^\lambda\down_P,\theta]\neq 0\}. \]
Here, we denoted by $\chi^\lambda$ the irreducible character of $S_n$ naturally labelled by $\lambda\in\cP(n)$. 

In this paper, we first associate to each $\theta\in\Irr(P)$ a collection of labelled, complete $p$-ary trees $\cT(\theta)$, then we describe the set $\Omega(\theta)$ in terms of some (surprisingly easy!) combinatorial properties of the trees in $\cT(\theta)$. (To be precise, the elements of $\cT(\theta)$ will be graph-theoretic trees if $n$ is a power of $p$; for general $n\in\N$, they will be forests, or more specifically, ordered tuples of trees.)

Our first main result is contained in \textbf{Theorems~\ref{thm:val-1-arbitrary} and~\ref{thm:puncture-arbitrary}}. For $p\ge 5$, we determine $\Omega(\theta)$ exactly for a certain class of irreducible characters of $P$ defined in terms of the \emph{value} $\val(\theta)$ of $\theta$ (see Definitions~\ref{def:trees-stats} and~\ref{def:theta-stats}).
In order to describe this result briefly, we introduce the following notation. Given $t,n\in\N$, we define
\begin{equation}\label{eq:box}
	\cB_n(t) := \{\lambda\in\cP(n) \mid \lambda_1, \ell(\lambda)\le t\}.
\end{equation}
This is the set of all partitions of $n$ whose first part and number of parts is at most $t$, or equivalently whose associated Young diagram fits into a $t\times t$ square box. (Later in Section~\ref{sec:prelim}, we will also define the `punctured' box $\puncb{n}{t}$ which is $\cB_n(t)$ with two specific partitions and their conjugates removed.)
In Theorems~\ref{thm:val-1-arbitrary} and~\ref{thm:puncture-arbitrary}, we show that there is an easy-to-compute statistic $\gamma_0(\theta)$ such that $\Omega(\theta)$ is precisely $\cB_n(n-\gamma_0(\theta))$ or $\puncb{n}{n-\gamma_0(\theta)}$, and determine when each occurs.

The second main result of this work deals with \emph{every} irreducible character $\theta$ of $P$ when $p\ge 5$. Let $m(\theta)$ and $M(\theta)$ be the positive integers defined as follows:
\begin{equation}\label{eq:m-and-M}
	m(\theta) := \max\{t\in\N \mid \cB_{n}(t) \subseteq \Omega(\theta) \} \quad \text{and}\quad  M(\theta) := \min\{t\in\N \mid \Omega(\theta) \subseteq \cB_{n}(t)\}.
\end{equation}
That is, for any $\theta\in\Irr(P)$ the values $m(\theta)$ and $M(\theta)$ are tight bounds such that
\[ \cB_{n}(m(\theta))\subseteq \Omega(\theta)\subseteq \cB_{n}(M(\theta)). \]
In \textbf{Theorems~\ref{thm:arbitrary-M} and~\ref{thm:final-m(theta)}} we determine both $M(\theta)$ and $m(\theta)$ for every $\theta\in\Irr(P)$, expressing these quantities in terms of the number of vertices in the associated trees in $\cT(\theta)$ satisfying a specific (and easily computable) property.
While it remains an open problem to determine in general the positivity of Sylow branching coefficients $[\chi^\lambda\down_P,\theta]$ when the partition $\lambda$ lies in $\cB_n(M(\theta))\setminus\cB_n(m(\theta))$, we show in \textbf{Corollary~\ref{cor:M-m}} that in fact there are very few such $\lambda$, since the difference $M(\theta)-m(\theta)$ is very small compared to $n$.
These results provide some interesting consequences.
For instance, if we define $\Omega(n)$ as the intersection of all the sets $\Omega(\theta)$ as $\theta$ varies across all of $\Irr(P)$, then \textbf{Corollary~\ref{cor:limit}} shows that the proportion $|\Omega(n)|/|\cP(n)|$ tends to $1$ when $n$ goes to infinity. This says that, asymptotically, the restriction to $P$ of almost all irreducible characters of $S_n$ has every irreducible character of $P$ as a constituent.
We conclude the article by remarking on the case of $p\in\{2,3\}$.

\bigskip

\section{Background, Notation and Preliminary Results}\label{sec:prelim}

First, we fix some notation adopted in the rest of the paper, as well as recall some basic facts concerning the combinatorial representation theory of symmetric groups and their Sylow subgroups. For $n$ a natural number, $S_n$ denotes the symmetric group on $n$ objects. Given a prime number $p$, we let $\Syl_p(S_n)$ denote the set of Sylow $p$-subgroups of $S_n$, and let $P_n\in\Syl_p(S_n)$.

Given $x,y\in\Z$ we let $[x,y]:=\{z\in\Z \mid  x\le z\le y\}$. Notice that if $x>y$ then $[x,y]$ is the empty set. 
For a finite group $G$ we let $\Char(G)$ denote the set of ordinary characters of $G$, and let $\Irr(G)$ denote the subset of those characters which are irreducible. 
Given $\theta,\chi\in\Char(G)$ we write $\theta\mid\chi$ to mean that $\chi=\theta+\Delta$ for some $\Delta\in\Char(G)$ or $\Delta=0$. If $\theta$ is irreducible and $\theta\mid \chi$, we say that $\theta$ is an irreducible constituent of $\chi$. Given a subgroup $H$ of $G$ and $\phi\in\Irr(H)$ we let $\Irr(G\mid\phi):=\{\chi\in\Irr(G) \mid [\chi\down_H,\phi] \ne 0\}$. 

\subsection*{Representations of $S_n$ and Combinatorics of Partitions}
Given $n\in\N$, it is well known that the irreducible characters of $S_n$ are naturally in bijection with the partitions of $n$. We denote by $\cP(n)$ the set of partitions of $n$ and, for any $\lambda\in\cP(n)$, we denote by $\chi^\lambda$ the irreducible character of $S_n$ corresponding to $\lambda$.
Standard sources for a complete account of the ordinary representation theory of $S_n$ are \cite{JK} and \cite{OlssonBook}. In order to ease the notation we will sometimes identify irreducible characters with their corresponding partitions. It will be always clear from the context whether $\lambda$ denotes a partition or its corresponding irreducible character.

Let $\lambda=(\lambda_1,\lambda_2,\ldots,\lambda_t)\in\cP(n)$. The integers $\lambda_i$ are called the \emph{parts} of $\lambda$. The integer $t$ is called \emph{length} of $\lambda$, also denoted $\ell(\lambda)$. If $\lambda$ is a partition and $m\in\N$ satisfies $m\ge\lambda_1$, then we will use the shorthand $(m,\lambda)$ to mean the partition $(m,\lambda_1,\lambda_2,\dotsc,\lambda_t)$. Given partitions $\lambda$ and $\mu$ we say that $\mu$ is a \emph{subpartition} of $\lambda$, written $\mu\subseteq\lambda$, if $\mu_i\le\lambda_i$ for all $i\in [1,\ell(\mu)]$.

The conjugate partition of $\lambda$ is denoted by $\lambda'$, and observe that $\ell(\lambda')=\lambda_1$. 
Given a subset $A\subseteq\cP(n)$ we say that $A$ is \emph{closed under conjugation} if for every $\lambda\in \cP(n)$ one has that $\lambda\in A$ if and only if $\lambda'\in A$. 
We define the closure $\close{A}$ of a set of partitions $A$ under the conjugation operator as follows: 
\[ \close{A} := A \cup \{\lambda' \mid \lambda \in A \}. \]
We have an easy fact concerning the Sylow restrictions of $\chi^\lambda$ and $\chi^{\lambda'}$ whenever $p$ is odd.

\begin{lemma}\label{lem:conj-restr}
	Let $p$ be an odd prime, $n\in\N$ and $P_n\in\Syl_p(S_n)$. Let $\lambda\in\cP(n)$. Then $\chi^\lambda\down_{P_n} = \chi^{\lambda'}\down_{P_n}$. In particular, $\Omega(\theta)=\close{\Omega(\theta)}$.
\end{lemma}

\begin{proof}
	This follows from the well-known fact that $\chi^{\lambda'}$ equals $\chi^\lambda$ multiplied by the sign character (see \cite[2.1.8]{JK}, for example), and the observation that $P_n$ is contained in the alternating subgroup of $S_n$ whenever $p$ is odd.
\end{proof}

Let $n\in\N$ and let $t\in [1,n]$. We recall from \eqref{eq:box} that $\cB_n(t)$ is the subset of $\cP(n)$ consisting of those partitions whose Young diagram fits into a $t\times t$ square grid. More precisely, 
\[ \cB_n(t) := \{\lambda\in\cP(n) \mid \lambda_1\le t, \ell(\lambda)\le t \}. \]
Notice also that $\cB_n(t)=\close{\cB_n(t)}$. 

A partition $\lambda$ with at most two parts, i.e.~$\ell(\lambda)\le 2$, is called a \emph{two-row} partition. On the other hand, we will call \emph{hook} partitions all those partitions of the form $(n-x, 1^x)$ with $x\in[0,n-1]$, where the exponential notation means that there are $x$ parts of size $1$. 
The following definition is non-standard and will be used frequently later in this article. 

\begin{definition}
	Let $n\in\N$.
	\begin{itemize}
		\item[(a)] We define the set $\operatorname{Thin}(n)$ of \emph{thin partitions} of $n$ to be the following subset of $\cP(n)$:
		\[ \operatorname{Thin}(n) := \close{(\{ \lambda\in\cP(n) \mid \lambda_2\le 1 \} \cup \{ \lambda\in\cP(n) \mid \ell(\lambda)\le 2 \})}. \]
		In other words, a partition is thin if it has a Durfee square of size at most 2. Explicitly, a thin partition is one which is a hook partition, a two-row partition, or the conjugate of a two-row partition.
		\item[(b)] For any $x\in[\lceil\tfrac{n}{2}\rceil,n]$, we denote by $\tworow{n}{x}$ the two-row partition $(x,n-x)$.
		\item[(c)] For any $y\in[1,n]$, we denote by $\hook{n}{y}$ the hook partition $(y,1^{n-y})$.
	\end{itemize}
\end{definition}

In order to study the restriction of characters of symmetric groups to their Sylow subgroups, we will often consider first restriction to intermediate subgroups. An important family of subgroups that will come into play in this setting is that of \emph{Young subgroups}. In this case, the properties of the restriction of irreducible characters are controlled by the Littlewood--Richardson rule. This gives an explicit method for calculating multiplicities of the form $[\chi^\lambda\down_{S_{m+n}}, \chi^\mu\times\chi^\nu]$, where $\lambda\in\cP(m+n)$, $\mu\in\cP(m)$ and $\nu\in\cP(n)$. We refer the reader to \cite[Chapter 16]{J} for the precise definition of this combinatorial rule. 

Here we just recall that given a partition $\nu=(n_1,n_2,\ldots, n_t)$ of $n$, meaning $n_i\in\N_0$ for all $i$ and $\sum_i n_i=n$, we let $Y_{\nu}=S_{n_1}\times S_{n_2}\times\cdots\times S_{n_t}\leq S_n$ be the Young subgroup corresponding to $\nu$. That is, setting $n_0=0$, the $t$ orbits of $Y_\nu$ in its action on $[1,n]$ are $\{\sum_{j=0}^{i-1}n_j+1, \sum_{j=0}^{i-1}n_j+2, \dotsc, \sum_{j=0}^{i-1}n_j+n_i \}$, for each $i\in[1,t]$.
For any $\lambda\in\cP(n)$ and any $(\mu^1,\mu^2,\ldots,\mu^t)\in\cP(n_1)\times\cP(n_2)\times\cdots\times\cP(n_t)$ we denote by $\LR(\lambda; \mu^1,\ldots, \mu^t)$ the multiplicity of 
$\chi^{\mu^1}\times\cdots\times\chi^{\mu^t}$ as an irreducible constituent of $\chi^\lambda\down_{Y_\nu}$. One combinatorial interpretation of this coefficient is described in detail in \cite[Section 2.4]{GL2} and is omitted here. Nevertheless, we record some easy consequences of the Littlewood--Richardson rule that we will use frequently in this article. 

\begin{lemma}\label{lem:LR-first-part}
	If $\LR(\lambda; \mu^1,\ldots, \mu^t)\neq 0$, then $\mu^i\subseteq \lambda$ for all $i\in [1,t]$ and $\lambda_1\leq (\mu^1)_1+(\mu^2)_1+\cdots+ (\mu^t)_1$.
\end{lemma}

\begin{lemma}\label{lem:iterated-LR}
	Let $a,b_1,\dotsc,b_a\in\N$. Let $\nu^1,\dotsc,\nu^a$ be partitions such that $b_i\ge |\nu^i|$ for all $i$ and let $c=|\nu^1|+\cdots+|\nu^a|$. 
	Let $\mu\in\cP(c)$ and let $\lambda=(b_1+b_2+\cdots+b_a,\mu)$. Then 
	\[ \LR(\lambda; (b_1,\nu^1),\ldots,(b_1,\nu^a))=\LR(\mu; \nu^1,\ldots,\nu^a). \]
\end{lemma}

Given two partitions $\lambda\in\cP(n)$ and $\mu\in\cP(m)$, we define $\lambda+\mu$ as the partition of $n+m$ whose $i$-th part is given by $\lambda_i+\mu_i$ for all $i\in\N$ (here we regard $\lambda_i=0$ whenever $i>\ell(\lambda)$ and similarly for $\mu$). 
The following is another immediate application of the Littlewood--Richardson rule. 

\begin{lemma}\label{lem:LR-sum}
	Let $n,k\in\N$, let $\mu^1,\ldots,\mu^k\in\cP(n)$ and let $\lambda=\mu^1+\mu^2+\cdots+\mu^k\in\cP(kn)$. Then $\LR(\lambda; \mu^1,\ldots, \mu^k)=1$.
\end{lemma}

Let $n,m\in\N$ and let $A\subseteq\cP(n)$ and $B\subseteq\cP(m)$. We define
\[ A\star B := \{\lambda\in\cP(n+m) \mid \LR(\lambda;\mu,\nu) \neq 0\ \text{for some}\ \mu\in A\ \text{and}\ \nu\in B\}. \]
The $\star$-product operator on sets of partitions was first introduced in \cite{GL2}, and it is clearly both commutative and associative. 
Let us now fix $q\in\N$ and let us consider $A^{\star q}$, the $\star$-product of $q$ copies of the subset $A$. By definition, $A^{\star q}$ is a subset of $\cP(qn)$. We now want to define a specific subset of $A^{\star q}$ that will play a crucial role in several arguments below. In order to do this we first fix the following convention: we will say that a collection of objects $\omega_1,\dotsc,\omega_r$ is \emph{mixed} if the objects are not all equal. (In this article, such objects may refer to collections of partitions, or of characters of a given finite group, or of trees as defined in Section~\ref{sec:trees} later, for example.)
With this in mind, we define $\cM(q, A)$ as the following subset of $A^{\star q}$:
\[ \cM(q,A) := \{\lambda\in A^{\star q} \mid \LR(\lambda; \mu^1,\ldots, \mu^q)\neq 0\ \text{for some mixed}\ \mu^1,\ldots, \mu^q\in A\}. \]
We give a short example of some of the concepts introduced so far. 

\begin{example}\label{ex:1}
	Let $n=4$ and consider the set $\cB_4(3)=\{(3,1), (2,2), (2,1,1)\}$. It is easy to verify that $\mathcal{B}_4(3)^{\star 2}=\mathcal{B}_8(6)$. For instance, $(6,2)\in \mathcal{B}_8(6)$, and in fact we have that $\LR((6,2); (3,1), (3,1))=1$. Setting $Y$ as the Young subgroup $Y_{(4,4)}=S_4\times S_4
	<S_8$, we have that
	\[ (6,2)\down_Y = (4)\times (4) + (4)\times (3,1) + (3,1)\times (4) + (4)\times (2,2) + (2,2)\times (4) + (3,1)\times (3,1). \]
	This shows that $(6,2)\notin \cM(2,\mathcal{B}_4(3))$ and therefore that $\cM(2,\mathcal{B}_4(3))\subsetneqq \mathcal{B}_4(3)^{\star 2}=\mathcal{B}_8(6)$. In fact, it is straightforward to calculate that $\cM(2,\mathcal{B}_4(3))=\mathcal{B}_8(6)\setminus \close{\{(6,2), (6,1,1), (4,4)\}}$.	\newqed
\end{example}

Example~\ref{ex:1}\footnote{Note that at the end Example~\ref{ex:1}, and throughout the rest of this paper, we have used the symbol $\lozenge$ to indicate the end of the example (or remark) environment to help with readability.} highlights a few properties of the $\star$ operator that we will use often in this paper. For instance, we have seen that $\mathcal{B}_4(3)^{\star 2}=\mathcal{B}_8(6)$. This is a specific occurrence of a more general behaviour which was first shown in \cite[Proposition 3.2]{GL2}.

\begin{proposition}\label{prop:SL5.7}
	Let $n,n',t,t'\in\N$ be such that $\tfrac{n}{2}<t\le n$ and $\tfrac{n'}{2}<t'\le n'$. Then
	\[ \cB_n(t)\star\cB_{n'}(t')=\cB_{n+n'}(t+t'). \]
\end{proposition}

A second observation is that given $m,q\in\N$ and $t\in[1,n]$, Example~\ref{ex:1} shows that we cannot hope to have $\cM(q,\mathcal{B}_m(t))=\mathcal{B}_{qm}(qt)$. However, when $q$ is at least $3$ and $t$ is not too small compared to $m$, then the difference between these two sets is very small, as shown in \cite[Proposition 3.7]{GL2}.

\begin{proposition}\label{prop:BinD}
	Let $m,t,q\in\N$ with $\tfrac{m}{2}+1<t\le m$ and $q\ge 3$. Then we have that 
	\[ \cB_{qm}(qt-1)\subseteq \cM(q,\cB_m(t)). \]
\end{proposition}

We are now ready to extend this combinatorial result and completely describe $\cM(q,\cB_m(t))$.

\begin{proposition}\label{prop:mixed}
	Let $m,t,q\in\N$ with $2<\tfrac{m}{2}+1<t\le m$ and $q\ge 3$. Then
	\[ \cM(q,\cB_m(t)) = \begin{cases}
		\cB_{qm}(qt)\setminus \close{\{(qt,qm-qt), (qt,1^{qm-qt})\}} & \text{if } m-t\ne 1,\\
		\cB_{qm}(qt-1) & \text{if } m-t=1.
	\end{cases} \]
\end{proposition}

\begin{proof}
	If $t=m$ then $\cB_{qm}(qt)=\cP(qm)$. In this case the statement clearly holds because $\cP(qm)\setminus\cB_{qm}(qm-1)=\{(qm),(1^{qm})\}$ and $(qm)$ is not in $\cM(q,\cB_m(t))$ as $(qm)\down_{(S_m)^{\times q}} = (m)^{\times q}$, and similarly $(1^{qm})\notin\cM(q,\cB_m(t))$. 

	We can now assume that $t<m$. By Propositions~\ref{prop:SL5.7} and~\ref{prop:BinD} we have $\cB_{qm}(qt-1)\subseteq \cM(q,\cB_m(t))\subseteq\cB_m(t)^{\star q}=\cB_{qm}(qt)$. Since all sets considered here are closed under conjugation, it remains to consider whether $\lambda:=(qt,\mu)$ belongs to $\cM(q,\cB_m(t))$ or not, for each $\mu\in\cP(qm-qt)$.
		
	Suppose first that $\mu\in\close{\{(qm-qt)\}}$. We want to show that $\lambda\notin\cM(q,\cB_m(t))$. 
	Suppose for a contradiction that $\lambda\in\cM(q,\cB_m(t))$. Then there exist mixed $\lambda^1, \ldots, \lambda^q\in \cB_m(t)$ such that $\LR(\lambda;\lambda^1, \ldots, \lambda^q)\neq 0$. Using Lemma~\ref{lem:LR-first-part} we see that $qt=\lambda_1\le (\lambda^1)_1+\cdots+(\lambda^q)_1 \le qt$. Hence we must have $(\lambda^i)_1=t$ for all $i\in [1,q]$. In particular there exist mixed $\mu^1,\ldots, \mu^q\in\cP(m-t)$ such that $\lambda^i=(t,\mu^i)$ for all $i\in [1,q]$. Using Lemma~\ref{lem:iterated-LR} we deduce that $\LR(\mu;\mu^1, \ldots, \mu^q)\neq 0$. This is impossible since $\mu\notin\cM(q,\cP(m-t))$, as we have observed before. 

	Suppose now that $\mu\in\cP(qm-qt)\setminus\close{\{(qm-qt)\}}=\cB_{qm-qt}(qm-qt-1)$. 
	If $m-t>2$ then $\frac{m-t}{2}+1<m-t$, so by Proposition~\ref{prop:BinD} we have $\cB_{qm-qt}(qm-qt-1)\subseteq\cM(q,\cP(m-t))$.
	Hence there exist mixed $\nu^1,\ldots ,\nu^q\in\cP(m-t)$ such that $\LR(\mu;\nu^1,\ldots, \nu^q)\neq 0$. It follows that 
	$\LR(\lambda; (t,\nu^1), \ldots, (t,\nu^q))\neq 0$ by Lemma~\ref{lem:iterated-LR}. Since $(t,\nu^1), \ldots, (t,\nu^q))\in \cB_m(t)$ and are mixed, we conclude that $\lambda\in\cM(q,\cB_m(t))$.
	Let us now assume that $m-t=2$. Then $\mu\in\cB_{2q}(2q-1)=\cM(q,\cP(2))$, and arguing exactly as in the previous case we see that $\lambda\in\cM(q,\cB_m(t))$. 
	Finally assume that $m-t=1$. In this case $\mu\in\cB_q(q-1)$. 
	Suppose for a contradiction that $\lambda\in\cM(q,\cB_m(t))$. Then there exist mixed $\lambda^1, \ldots, \lambda^q\in \cB_m(t)$ such that $\LR(\lambda;\lambda^1, \ldots, \lambda^q)\neq 0$. Using Lemma~\ref{lem:LR-first-part} we see that $qt=\lambda_1\le (\lambda^1)_1+\cdots+(\lambda^q)_1 \le qt$. Thus, for every $i\in [1,q]$ there exist $\mu^i\in\cP(m-t)$ such that $\lambda^i=(t,\mu^i)$. Since $m-t=1$ this implies that $\lambda^1=\lambda^2=\cdots=\lambda^q$, which is a contradiction. 
\end{proof}

\begin{remark}
	In the statement of Proposition~\ref{prop:mixed}, notice that when $m=t$ the two possible cases for the structure of $\cM(q,\cB_m(t))$ coincide. Moreover, the hypothesis $m>2$ is necessary. For $m\le 2$, we have the following possibilities: if $m=1$ then $t=m=1$ and $\cM(q,\cP(1))=\emptyset$ as $|\cP(1)|=1$. On the other hand, if $m=2$ then $t=1$ gives $\cB_m(t)=\emptyset$, while if $t=2$ then $\cM(q,\cP(2))=\cB_{2q}(2q-1)$ since $q\ge 3$. \newqed
\end{remark}

We conclude this section by introducing a new combinatorial object.

\begin{definition}\label{def:puncture}
	For $m,t\in\N$ with $\frac{m}{2}< t\le m-2$, we define the \emph{punctured box} $\puncb{m}{t}$ to be
	\[ \puncb{m}{t} := \cB_m(t)\setminus \close{\{ (t,m-t-1,1), (t,2,1^{m-t-2}) \}}. \]
	Similarly, for $n\in\N_{\ge 3}$, we define the \emph{punctured set of partitions of $n$}, $\puncp{n}$, to be
	\[ \puncp{n} := \cP(n)\setminus \close{\{(n-1,1)\}}. \]
\end{definition} 

As we will show in our main results, these apparently strange subsets of partitions are intimately connected to the sets $\Omega(\theta)$ that, as mentioned in the introduction, lie at the centre of our research interest in this article. 

\begin{remark}
	In the definition of a punctured box $\puncb{m}{t}$ we require $t>\tfrac{m}{2}$ so that $(t,m-t)$ and $(t,1^{m-t})$ are genuine partitions lying in $\cB_m(t)$; moreover, we observe that the condition $\tfrac{m}{2}<m-2$ implies that $m>4$. 
	On the other hand, $t\leq m-2$ guarantees the well-definition of the partition $(t, m-t-1, 1)$. In particular, we notice that when $t=m-2$ then in fact $\puncb{m}{t} = \cB_m(m-3)$.
	
	As for $\puncp{n}$, this is ill-defined for $n=1$ and $\puncp{2}=\emptyset$, hence $\puncp{n}$ is defined only for $n\ge 3$.
	Even though $\cB_n(n)=\cP(n)$, it is not the case that $\puncb{n}{n}=\puncp{n}$. In particular, notice that $\puncb{m}{t}$ was not defined when $t=m$. \newqed 
\end{remark}

The following propositions describe how the various subsets of partitions introduced so far interact with each other with respect to the $\star$-product operator. These statements, and more generally speaking the $\star$-product operator itself, constitute the necessary combinatorial tools we will need to study restriction of irreducible characters of symmetric groups to their Sylow $p$-subgroups. In order to fully understand the previous assertion, we recommend to the reader to see Corollary~\ref{cor:omega-mixed} and Lemma~\ref{lem:omega-general-n} below, before proceeding. 

\begin{proposition}\label{prop:smooth}
	Let $x,a,y,b\in\N$. Then:
	\begin{itemize}\itemspace{5pt}
		\item[(i)] $\cB_x(a)\star\cB_y(b)=\cB_{x+y}(a+b)$, for $\tfrac{x}{2}<a\le x$ and $\tfrac{y}{2}<b\le y$. 
		
		\item[(ii)] $\cB_x(a)\star \puncp{y} = \cB_{x+y}(a+y)$, for $\frac{x}{2}<a\le x$ and $y\ge 5$.
		
		\item[(iii)] $\puncp{x} \star \puncp{y} = \cP(x+y)$, for $x,y\ge 5$.
		
		\item[(iv)] $\cB_x(a) \star \puncb{y}{b} = \cB_{x+y}(a+b)$, for $\tfrac{x}{2}<a\le x-1$ and $\tfrac{y}{2}+1<b\le y-5$.
		
		\item[(v)] $\cP(x) \star \puncb{y}{b} = \puncb{x+y}{x+b}$, for $\tfrac{y}{2}+1<b\le y-5$.
		
		\item[(vi)] $\puncp{x} \star\puncb{y}{b} = \puncb{x+y}{x+b}$, for $x\ge 5$ and $\tfrac{y}{2}+1<b\le y-5$.
		
		\item[(vii)] $\puncb{x}{a} \star \puncb{y}{b} = \cB_{x+y}(a+b)$, for $\tfrac{x}{2}+1<a\le x-5$ and $\tfrac{y}{2}+1<b\le y-5$. 
	\end{itemize}
\end{proposition}

\begin{proof}
	In all cases, the sets involved are closed under conjugation of partitions, so when considering partitions  we may assume without loss of generality that their first part is greater than or equal to their length.
	\begin{enumerate}[label=\textup{(\roman*)}]
		\item This is Proposition~\ref{prop:SL5.7}.

		\item This is \cite[Lemma 3.3]{GL2}.

		\item By (ii) we have that $\puncp{x}\star\puncp{y}\supseteq \cB_x(x-2)\star\puncp{y}=\cB_{x+y}(x+y-2)$. 
		Moreover, for each $\lambda\in\{(x+y-1,1),(x+y)\}$, we have that $\LR(\lambda;(x),(y))\neq 0$. Hence $\lambda\in \puncp{x} \star \puncp{y} $ and therefore we conclude that $\puncp{x} \star \puncp{y} = \cP(x+y)$.
		
		\item First, observe that $\cB_x(a)\star \puncb{y}{b}\supseteq \cB_x(a) \star \cB_y(b-1) = \cB_{x+y}(a+b-1)$ by (i).
		Next, suppose $\lambda=(a+b,\mu)$ where $\mu\in\cP(x-a+y-b)$. Since $\cP(x-a+y-b)=\cP(x-a)\star\puncp{y-b}$ by (ii),
		then $\LR(\mu;\alpha,\beta)\neq 0$ for some $\alpha\in\cP(x-a)$ and $\beta\in\puncp{y-b}$. Thus, using Lemma~\ref{lem:iterated-LR} we see that $\LR(\lambda;(a,\alpha),(b,\beta))\neq 0$, and so $\lambda\in \cB_x(a)\star \puncb{y}{b}$. 
		
		\item First, observe that 
		$\cB_{x+y}(x+b-1) = \cP(x)\star\cB_y(b-1) \subseteq \cP(x) \star \puncb{y}{b} \subseteq \cP(x) \star \cB_y(b) = \cB_{x+y}(x+b)$. Next, suppose $\lambda=(x+b,\mu)$ where $\mu\in\cP(y-b)$. Then $\LR(\lambda;\alpha,\beta)\neq 0$ for $\alpha\in\cP(x)$ and $\beta\in\puncb{y}{b}$ only if $\alpha=(x)$ and $\beta=(b,\mu)$, from which we deduce $\lambda\in\cP(x) \star \puncb{y}{b}$ if and only if $\mu\in\puncp{y-b}$. Thus $\cP(x) \star \puncb{y}{b} = \puncb{x+y}{x+b}$.
		
		\item From (ii) and (v), $\cB_{x+y}(x+b-1) = \puncp{x}\star\cB_y(b-1) \subseteq \puncp{x}\star\puncb{y}{b} \subseteq \cP(x) \star \puncb{y}{b} = \puncb{x+y}{x+b}$. Moreover if $\lambda=(x+b,\mu)$ for some $\mu\in\puncp{y-b}$, then $\LR(\lambda;(x),(b,\mu))\ne0$ by Lemma~\ref{lem:iterated-LR}, and so $\lambda\in\puncp{x}\star\puncb{y}{b}$.

		\item From (i) and (iv) we have that 
		\[ \cB_{x+y}(a+b-1) = \puncb{x}{a} \star \cB_y(b-1) \subseteq \puncb{x}{a} \star \puncb{y}{b} \subseteq \cB_x(a)\star\cB_y(b) = \cB_{x+y}(a+b). \]
		Now, suppose $\lambda=(a+b,\mu)$ where $\mu\in \cP(x-a+y-b)$. Since $\cP(x-a+y-b)=\puncp{x-a}\star\puncp{y-b}$ by (iii),
		then $\LR(\mu;\alpha,\beta)\neq 0$ for some $\alpha\in\cP(x-a)$ and $\beta\in\cP(y-b)$. Then $\LR(\lambda;(a,\alpha),(b,\beta))\neq 0$ by Lemma~\ref{lem:iterated-LR}, and therefore we have that $\lambda\in\puncb{x}{a}\star\puncb{y}{b}$. 
	\end{enumerate}
\end{proof}

We remark that the various hypotheses on the variables $x,a,y,b$ in Proposition~\ref{prop:smooth} are necessary. For instance, in (i) we ask for $\frac{x}{2}<a$; if $a=\frac{x}{2}$ then $\cB_{2a}(a)$ has no hooks and therefore $\cB_{2a}(a)\star\cB_1(1)\ne\cB_{2a+1}(a+1)$, as $(a+1,1^a)\in \cB_{2a+1}(a+1)$. Since it is irrelevant for our current purposes we omit similar examples for the remaining statements. 

The different outcomes of $\star$-products of sets appearing in Proposition~\ref{prop:smooth} are summarised in Table~\ref{tab:smooth}. We note that the table is symmetric as $\star$ is commutative. 

\begin{table}[h!]
	\centering
	\begin{small}
		\[ \begin{array}{|c|cccc|}
			\hline
			\star & \cB & \punc\cB & \cP & \punc\cP\\
			\hline
			\cB & \cB & \cB & \cB & \cB \\
			\punc\cB & {\cB} & \cB & \punc\cB & \punc\cB \\
			\cP & {\cB} & {\punc\cB} & \cP & \cP \\
			\punc\cP & {\cB} & {\punc\cB} & {\cP} & \cP\\
			\hline
		\end{array} \]
		\caption{General form of Proposition~\ref{prop:smooth}.}\label{tab:smooth}
	\end{small}
\end{table}

We are now able to prove an analogue of Proposition~\ref{prop:BinD} for punctured boxes. 

 \begin{lemma}\label{lem:A'}
	Let $q,m,t\in\N$ be such that $q\ge 3$ and $\tfrac{m}{2}+2<t\le m-5$. Then
	\[ \cB_{qm}(qt-1) \subseteq\cM(q,\puncb{m}{t}). \]
\end{lemma}

\begin{proof}
	Since $\cB_{m}(t-1)\subseteq\puncb{m}{t}$ and $\tfrac{m}{2}+1<t-1$, we have that $\cB_{qm}(qt-q-1)\subseteq \cM(q,\puncb{m}{t})$ by Proposition~\ref{prop:mixed}. Now suppose $\lambda=(qt-u,\mu)\in\cP(qm)$ where $u\in\{1,2,\ldots,q\}$ and $\mu\in\cP(qm-qt+u)$.

	If $u=q$, then $\lambda_1=qt-q$. Let $N=q\cdot(m-t+1)$.
	If $\mu\in\cB_N(N-1)$ then $\LR(\mu;\mu^1,\ldots,\mu^q)\ne0$ for some mixed $\mu^i\in\cP(m-t+1)$, since $\cB_N(N-1)\subseteq\cM(q,\cP(m-t+1))$ by Proposition~\ref{prop:mixed}. Then $\LR(\lambda;(t-1,\mu^1),\ldots,(t-1,\mu^q))\ne0$ by Lemma~\ref{lem:iterated-LR}, and $(t-1,\mu^i)\in\puncb{m}{t}$ for all $i\in [1,q]$. 
	Hence $\lambda\in\cM(q,\puncb{m}{t})$.
	If $\mu\in\close{\{(N)\}}$ then $\lambda\in\{ \tworow{qm}{qt-q}, \hook{qm}{qt-q} \}$. 
	Suppose $\lambda=\tworow{qm}{qt-q}$. Then $\LR(\lambda;\tworow{m}{t-1},\ldots,\tworow{m}{t-1},\tworow{m}{t})\neq 0$.
	Hence $\lambda\in\cM(q,\puncb{m}{t})$. Similarly, $\hook{qm}{qt-q}\in \cM(q,\puncb{m}{t})$.
			
	Let us now consider the case where $u\in[1,q-1]$. Then $\mu\in\cP(q(m-t)+u) = \puncp{m-t}^{\star(q-u)}\star\cP(m-t+1)^{\star u}$ by Theorem~\ref{prop:smooth}, since $m-t\ge 5$ and $u,q-u\ge 1$. So $\LR(\mu;\nu^1,\ldots,\nu^{q-u},\omega^1,\ldots,\omega^u)\neq 0$ for some $\nu^1,\ldots,\nu^{q-u}\in\puncp{m-t}$ and some $\omega^1,\ldots,\omega^u\in\cP(m-t+1)$. Thus $\LR(\lambda;(t,\nu^1),\ldots,(t,\nu^{q-u}),(t-1,\omega^1),\ldots,(t-1,\omega^u))\neq 0$ by Lemma~\ref{lem:iterated-LR}. Since $u,q-u\geq 1$ we have that $(t,\nu^1),\ldots,(t,\nu^{q-u}),(t-1,\omega^1),\ldots,(t-1,\omega^u))$ are mixed elements of $\puncb{m}{t}$. Hence $\lambda\in\cM(q,\puncb{m}{t})$.

	Since $\puncb{m}{t}=\close{(\puncb{m}{t})}$, then $\cM(q,\puncb{m}{t})=\close{\cM(q,\puncb{m}{t})}$. Thus $\cB_{qm}(qt-1)\subseteq\cM(q,\puncb{m}{t})$.
\end{proof}

We conclude with a last technical statement that we obtain as a consequence of Proposition~\ref{prop:smooth}.

\begin{lemma}\label{lem:smoothing-no-thins}
	Let $x,a,y,b\in\N$. Suppose that $\Delta_x(a)$ is a subset of $\cP(x)$ of one of the following forms:
	\begin{itemize}
		\item[(i)] $\puncp{x}$ with $x\ge 5$, or
		\item[(ii)] $\puncb{x}{a}$ with $\tfrac{x}{2}+1<a\le x-5$, or
		\item[(iii)] $\cB_x(a)$ with $\tfrac{x}{2}<a\le x$.
	\end{itemize}
	Let $\cN_1$ be a (possibly empty) subset of $\cP(x)\setminus\Delta_x(a)$ containing no thin partitions. Suppose now that $\Delta_y(b)$ and $\cN_2$ are subsets of $\cP(y)$ satisfying analogous conditions. Then
	\[ (\Delta_x(a) \sqcup \cN_1) \star (\Delta_y(b) \sqcup \cN_2) = \Delta_{x+y}(a+b) \sqcup \cN_3, \]
	where 
	\[ \Delta_{x+y}(a+b) = \begin{cases}
		\puncb{x+y}{a+b} & \text{if } \Delta_x(a)=\puncb{x}{a} \text{ and } \Delta_y(b)\in\{ \puncp{y}, \cP(y) \},\\
		\puncb{x+y}{a+b} & \text{if } \Delta_y(b)=\puncb{y}{b} \text{ and } \Delta_x(a)\in\{ \puncp{x}, \cP(x) \},\\
		\cB_{x+y}(a+b) & \text{otherwise},
	\end{cases} \]
	and $\cN_3$ is a subset of $\cP(x+y)\setminus\Delta_{x+y}(a+b)$ that contains no thin partitions.
\end{lemma}

\begin{proof}
	Let $X:=(\Delta_x(a) \sqcup \cN_1) \star (\Delta_y(b) \sqcup \cN_2)$, so $X\supseteq \Delta_x(a)\star\Delta_y(b)$. Moreover, $\Delta_x(a)\star\Delta_y(b)$ equals $\Delta_{x+y}(a+b)$ in each case, by Theorem~\ref{prop:smooth} (cf.~Table~\ref{tab:smooth}). Setting $\cN_3:=X\setminus \Delta_{x+y}(a+b)$, it remains to show that $\cN_3$ contains no thin partitions in order to complete the proof.
	
	Suppose $\lambda\in X$ is thin. Then $\LR(\lambda;\mu,\nu)\ne0$ for some $\mu\in(\Delta_x(a) \sqcup \cN_1)$ and $\nu\in(\Delta_y(b) \sqcup \cN_2)$. But $\mu,\nu\subset\lambda$ so it must be that $\mu$ and $\nu$ themselves are thin. That is, $\mu\in\Delta_x(a)$ and $\nu\in\Delta_y(b)$, whence $\lambda\in \Delta_x(a)\star\Delta_y(b) = \Delta_{x+y}(a+b)$. Hence $\lambda\notin\cN_3$, and we conclude that no partitions in $\cN_3$ are thin.
\end{proof}

\bigskip

\section{Representations of $P_n$ and associated p-ary Trees}\label{sec:trees}

Fix a prime number $p$. 
The main aim of this section is to provide the reader with the definitions of the statistics $\gamma_0(\theta)$ and $\gamma_1(\theta)$ for each $\theta\in\Irr(P)$, where $P$ is a Sylow $p$-subgroup of a symmetric group (see Definition~\ref{def:trees-stats}). These statistics will be integral to describing the positivity of Sylow branching coefficients for symmetric groups in Sections~\ref{sec:pk} and~\ref{sec:arbitrary-n}, and are defined using certain labelled $p$-ary trees naturally associated to $\theta\in\Irr(P)$.

We remark that this correspondence between such characters and trees is not new, as \cite[Proposition 3.1]{OOR} describes for any natural numbers $n$ and $r$ an indexing of the irreducible representations of the $n$-fold wreath product $C_r\wr\cdots\wr C_r$ (where $C_r$ denotes the cyclic group of order $r$) using certain labelled trees. Below we treat only the special case where $r=p$ is prime, describing this correspondence in detail in Definitions~\ref{def:Phi-tree} and~\ref{def:admissible}, since the Sylow $p$-subgroups of symmetric groups are isomorphic to direct products of copies of iterated wreath products of cyclic groups of order $p$; this description is well known and was first published in \cite{Kaloujnine}.

\medskip

The structure of this section is as follows: in order to study $\Irr(P)$, we start by fixing some notation for characters of wreath products, before recording the structure of Sylow $p$-subgroups of symmetric groups and their irreducible characters. We refer the reader to \cite[Chapter 4]{JK} for further detail on these topics. Then, we introduce the relevant $p$-ary trees in Definition \ref{def:Fk}, before describing how they relate to $\Irr(P)$ in Definitions~\ref{def:Phi-tree} and~\ref{def:admissible}. In Definition~\ref{def:trees-stats} we present a number of statistics on these trees, and we end the section with some useful properties of characters that these combinatorial objects capture, such as Theorem~\ref{thm:equiv-classes} and Proposition~\ref{prop:arrows}.

\medskip

Let $K$ be a finite group, $n$ be a natural number and $H$ be a subgroup of $S_n$. We denote by $K^{\times n}$ the direct product of $n$ copies of $K$. The permutation action of $S_n$ on the direct factors of $K^{\times n}$ induces an action of $S_n$ (and thus of $H$) via automorphisms on $K^{\times n}$, giving the wreath product $K\wr H := K^{\times n} \rtimes H$. The normal subgroup $K^{\times n}$ is usually called \textit{base group} of the wreath product $K\wr H$. 
We denote the elements of $K\wr H$ by $(x_1,\dots,x_n ;h)$ for $x_i\in K$ and $h\in H$. Let $W$ be a $\C K$--module and suppose it affords the character $\phi$.
We let $W^{\otimes n} := W \otimes \cdots \otimes W$ ($n$ copies) be the corresponding $\C K^{\times n}$--module. The left action of $K\wr H$ on $W^{\otimes n}$ defined by linearly extending 
\[ (x_1,\dots ,x_n; h) : v_1 \otimes \cdots \otimes v_n \mapsto x_1 v_{h^{-1}(1)} \otimes \cdots \otimes x_n v_{h^{-1}(n) } \]
turns $W^{\otimes n}$ into a $\C(K\wr H)$--module, which we denote by $\widetilde{W}^{\otimes n}$.
We denote by $\widetilde{\phi^{\times n}}$ the character afforded by the $\C(K\wr H)$--module $\widetilde{W}^{\otimes n}$.
For any character $\psi$ of $H$, we let $\psi$ also denote its inflation to $K\wr H$ and define
\[ \cX(\phi; \psi) := \widetilde{\phi^{\times n}} \cdot \psi \in\Char(K\wr H). \]
Let $\phi \in\Irr(K)$ and let $\phi^{\times n}:= \phi \times \cdots \times \phi$ be the corresponding irreducible character of $K^{\times n}$. Observe that $\widetilde{\phi^{\times n}}\in\Irr(K\wr H)$ is an extension of $\phi^{\times n}$.
Hence, by Gallagher's Theorem \cite[Corollary 6.17]{IBook} we have
\[ \Irr(K\wr H \mid \phi^{\times n})= \{ \cX(\phi; \psi) \mid \psi \in \Irr(H)\}. \]
If $H=C_p$ is a cyclic group of prime order $p$, then every $\psi \in \Irr(K\wr C_p)$ is either of the form
\begin{itemize}
	\item[(i)] $\psi= \phi_1 \times \cdots \times \phi_p \up^{K\wr C_p}_{K^{\times p}}$ where $\phi_1 , \dots \phi_p \in \Irr(K)$ are mixed, or
	\item[(ii)] $\psi=\cX(\phi;\theta)$ for some $\phi\in\Irr(K)$ and $\theta\in\Irr(C_p)$.
\end{itemize}
We remark that in case (i) we have that $\Irr(K\wr C_p \mid \phi_1 \times \cdots \times \phi_p)=\{\psi\}$. On the other hand, for any $\phi\in\Irr(K)$ we have by basic Clifford theory that
\[ (\phi^{\times p})\up^{K\wr C_p}_{K^{\times p}} = \sum_{\theta\in\Irr(C_p)} \cX(\phi;\theta). \]
We record here two lemmas that we will use frequently later in the article. We refer the reader to \cite[Lemmas 2.18 and 2.19]{SLThesis} for complete proofs.

\begin{lemma}\label{lem:SL2.18}
	Let $p$ be a prime and $K$ a finite group. Let $\eta\in\Char(K)$ and $\phi\in\Irr(K)$. If $[\eta,\phi]\ge 2$, then $[\cX(\eta;\tau), \cX(\phi;\theta)] \ge \sum_{i=1}^{p-1}\tfrac{1}{p}\binom{p}{i}$ for all $\tau,\theta\in\Irr(C_p)$. In particular, if $p$ is odd then $[ \cX(\eta;\tau), \cX(\phi;\theta)]\ge 2$. 
\end{lemma}

\begin{lemma}\label{lem:SL2.19}
	Let $K$ and $H$ be finite groups. Let $\alpha\in\Irr(K)$ and $\Delta\in\Char(K)$ be such that $[\alpha,\Delta]=1$. Then $[\cX(\Delta;\theta),\cX(\alpha;\beta)] = \delta_{\beta,\theta}$ for any $\beta,\theta\in\Irr(H)$.
\end{lemma}

Given $n\in\N$, we let $P_n$ denote a Sylow $p$-subgroup of $S_n$. Observe that $P_1$ is the trivial group while $P_p\cong C_p$ is cyclic of order $p$. For each $k\in\N_{\ge 2}$, we have that $P_{p^k}=\big(P_{p^{k-1}}\big)^{\times p} \rtimes P_p=P_{p^{k-1}}\wr P_p\cong P_p\wr \cdots \wr P_p$ ($k$-fold wreath product). 

Now consider an arbitrary natural number $n\in\N$, not necessarily a power of $p$. In this article we will write the $p$-adic expansion of $n$ in the following form: 
\[ n = \sum_{i=1}^t p^{n_i} \]
for some $t\in\N$ and $n_i\in\N_0$ such that $n_1\le n_2\le \cdots\le n_t$, and for each $j\in\N_0$, $|\{i\in[1,t]\mid n_i=j \}|\le p-1$. In other words, if $n=\sum_{i=1}^k a_ip^{m_i}$ is the usual base $p$ expansion of $n$ (i.e.~$0\le m_1<m_2<\cdots<m_k$ and $a_i\in[1,p-1]$ for all $i$), then the sequence $n_1,n_2,\dotsc,n_t$ equals $m_1,\dotsc,m_1,m_2,\dotsc,m_2,\dotsc,m_k,\dotsc,m_k$ where each $m_i$ appears $a_i$ times.
Then $P_n \cong P_{p^{n_1}} \times \cdots \times P_{p^{n_t}}$, and it follows that
\[ \Irr(P_n) = \{\theta_1\times\theta_2\times\cdots\times\theta_t\ |\ \theta_i\in\Irr(P_{p^{n_i}}),\ \text{for all}\ i\in [1,t]\}. \]
Hence the study of irreducible characters of $P_n$ ultimately relies on the properties of irreducible characters of the Sylow $p$-subgroups $P_{p^k}$ of the symmetric group $S_{p^k}$, as $k$ varies over all non-negative integers.
From now on we will denote by $\phi_0, \phi_1,\ldots, \phi_{p-1}$ the irreducible characters of $P_p$, with $\phi_0$ being the trivial character $\triv_{P_p}$. 
Since $P_{p^k}=P_{p^{k-1}}\wr P_p$, the information gathered above about characters of wreath products implies that if $\theta\in\Irr(P_{p^k})$ then exactly one of the following two cases holds:
\begin{itemize}
	\item[(i)] $\theta= \theta_1 \times \cdots \times \theta_p \up^{P_{p^k}}$, where $\theta_1 , \dots \theta_p \in \Irr(P_{p^{k-1}})$ are mixed, or
	\item[(ii)] $\theta = \cX(\phi; \phi_\varepsilon)$ for some $\phi \in \Irr(P_{p^{k-1}})$ and $\phi_\varepsilon \in \Irr(P_p)$. 
\end{itemize}
Notice that in case (i), $\theta=\theta_{\sigma(1)}\times\cdots\times\theta_{\sigma(p)}\up^{P_{p^k}}$ for any cyclic permutation $\sigma\in\langle(12\dotsc p)\rangle$. In case (ii), the parameter $\varepsilon$ lies in $[0,p-1]$, and we may sometimes abbreviate $\cX(\phi;\phi_\varepsilon)$ to $\cX(\phi;\varepsilon)$ when the meaning is clear from context.

This description allows us to canonically associate to each $\theta\in\Irr(P_{p^k})$ an equivalence class of labelled, complete $p$-ary trees $\cT(\theta)$ (see Lemma~\ref{lem:tree1} below). As we will prove in the following sections, this combinatorial object will allow us to give an easy formula for calculating the integers $m(\theta)$ and $M(\theta)$ (as defined in (\ref{eq:m-and-M}) in the introduction) for all $\theta\in\Irr(P_n)$. In particular, Theorems~\ref{thm:M} and~\ref{thm:m-final-pk} will show that $M(\theta)$ and $m(\theta)$ are directly related to the parameters $\gamma_0(\theta)$ and $\gamma_1(\theta)$.

\begin{definition}\label{def:Fk}
	For each $k\in\N$, let $\cF_k$ be the set consisting of all the rooted, complete $p$-ary trees of height $k-1$ such that each vertex is labelled by an integer in $\{0,1,\dotsc,p-1,p\}$.
\end{definition}

\begin{remark}
	\begin{enumerate}[label=(\roman*)]
		\item In Definition~\ref{def:Fk}, height $k-1$ means that the maximal distance from a vertex to the root is $k-1$. By complete we mean that every vertex, other than those at distance $k-1$ from the root, is adjacent to exactly $p$ vertices of distance one greater to the root, and vertices at distance $k-1$ are leaves. See Figure~\ref{fig:people} for diagrams of examples of such trees; in particular, we will always draw such trees with the root vertex at the top.
		
		\item We use the notation $\substack{\varepsilon\\\bullet}$ for the tree in $\cF_1$ consisting of a single vertex whose label is $\varepsilon\in [0,p]$. In particular, $\cF_1=\{ \substack{0\\\bullet}, \substack{1\\\bullet}, \dotsc, \substack{p-1\\\bullet}, \substack{p\\\bullet} \}$.	
		For $k\ge 2$ and $T\in\cF_k$, we will use the notation $T=(T_1\mid T_2\mid \dotsc\mid T_p;\substack{\varepsilon\\\bullet})$ to say that $\varepsilon$ is the label of the root vertex of $T$, and $T_1,\dotsc,T_p\in\cF_{k-1}$ are the subtrees of $T$ rooted at the children of the root vertex of $T$, from left to right. 
		
		\item Sometimes the root vertex $\substack{\varepsilon\\\bullet}$ is replaced by just its label $\varepsilon$ in such notation, i.e.~$T=(T_1\mid\dotsc\mid T_p;\varepsilon)$. Similarly, if $T_i\in\cF_1$ (i.e.~$T_i$ contains only a single vertex, say $\substack{\delta\\\bullet}$), then we may just write the label $\delta$ in place of $\substack{\delta\\\bullet}$.
		\newqed
	\end{enumerate}
\end{remark}

Every rooted, unlabelled, complete $p$-ary tree of height $k$ has as its graph automorphism group $\Gamma_k$, isomorphic to the $k$-fold wreath product $S_p\wr S_p\wr \cdots\wr S_p$. A Sylow $p$-subgroup $\Pi_k$ of $\Gamma_k$ is isomorphic to the $k$-fold wreath product $C_p\wr C_p\wr \cdots\wr C_p$, acting cyclically on the $p$ subtrees under each non-leaf vertex. Note therefore that $\Pi_k$ acts on $\cF_{k+1}$ by permuting the vertex labels, in accordance with its action on vertices.

\begin{definition}\label{def:Phi-tree}
	Let $k\in\N$. 
	\begin{itemize}
		\item[(a)] Let $\cT_k$ denote the set of orbits of $\cF_k$ under the action of $\Pi_{k-1}$. In other words, each element of $\cT_k$ is an orbit, or an equivalence class of trees belonging to $\cF_k$, where two trees are equivalent if one may be obtained from the other via an element of $\Pi_{k-1}$. Given $T\in\cF_k$, we let $[T]$ denote its $\Pi_{k-1}$-orbit.
		
		\item[(b)] We define the map 
		\[ \Phi_k:\Irr(P_{p^k})\to\cT_k \]
		recursively as follows:
		\begin{itemize}
			\item[(i)] For $k=1$ and for any $\varepsilon\in[0,p-1]$, we define $\Phi_1(\phi_\varepsilon)=[\substack{\varepsilon\\\bullet}]$.
			
			\item[(ii)] For any $k\in\N_{\ge 2}$ and $\theta\in\Irr(P_{p^k})$, we define
			\[ \Phi_k(\theta) = \begin{cases}
				[(T_1\mid \dotsc\mid T_p; \varepsilon)] & \text{if } \theta=\cX(\phi;\phi_\varepsilon)\text{ for some }\phi\in\Irr(P_{p^{k-1}})\text{ and }\varepsilon\in[0,p-1],\\
				[(T_1\mid \dotsc\mid T_p; p)] & \text{if } \theta=\theta_1\times\cdots\times\theta_p\up^{P_{p^k}} \text{ for some mixed }\theta_1,\dotsc,\theta_p\in\Irr(P_{p^{k-1}}),
			\end{cases} \]
			where $T_i\in\Phi_{k-1}(\phi)$ for each $i\in[1,p]$ in the first case, and $T_i\in\Phi_{k-1}(\theta_i)$ for each $i\in[1,p]$ in the second case. 
		\end{itemize}
	\end{itemize}
\end{definition}


\begin{lemma}\label{lem:tree1}
	Let $k\in\N$. Then the map $\Phi_k$ is well-defined and injective.
\end{lemma}

\begin{proof}
	The statement is clear for $k=1$, so now assume $k\ge 2$. 
	
	To see that $\Phi_k$ is well-defined: 
	suppose $T_i, T_i'\in \cF_{k-1}$ satisfy $[T_i]=[T'_i]$ for each $i\in[1,p]$. Then there exist $g_i\in\Pi_{k-2}$ such that $g_i(T_i)=T'_i$. Then $g:=(g_1,\dotsc,g_p;1)\in\Pi_{k-1}$ and $g\big((T_1\mid\dotsc\mid T_p;\varepsilon)\big) = (T'_1\mid\dotsc\mid T'_p;\varepsilon)$ for any $\varepsilon\in[0,p-1]$. That is, $[(T_1\mid\dotsc\mid T_p;\varepsilon)] = [(T'_1\mid\dotsc\mid T'_p;\varepsilon)]$. Similarly, $[(T_1\mid\dotsc\mid T_p;p)] = [(T'_1\mid\dotsc\mid T'_p;p)]$. 
	It is also easy to see that if $\theta=\theta_1\times\cdots\times\theta_p\up^{P_{p^k}}=\eta_1\times\cdots\times\eta_p\up^{P_{p^k}}$ then there exists $\sigma\in C_p$ such that $\theta_{\sigma(i)}=\eta_i$ for all $i\in [1,p]$. Therefore, given $T_i\in\Phi_{k-1}(\theta_i)$ and $T'_i\in\Phi_{k-1}(\eta_i)$ for each $i\in[1,p]$ and setting $h:=(1,\dotsc,1 ;\sigma)$, we have that $h\big((T_1\mid\dotsc\mid T_p; p)\big) = (T_{\sigma(1)}\mid\dotsc\mid T_{\sigma(p)}; p)$. It follows that $[(T_1\mid\dotsc\mid T_p;p)] = [(T_{\sigma(1)}\mid\dotsc\mid T_{\sigma(p)}; p)] = [(T'_1\mid\dotsc\mid T'_p;p)]$, and thus that $\Phi_k$ is well-defined. That is, $\Phi_k(\theta)$ as defined in Definition~\ref{def:Phi-tree} does not depend on the choice of the tree $T_i\in \Phi_{k-1}(\phi)$ (or $T_i\in \Phi_{k-1}(\theta_i)$).
	
	Now, we show that $\Phi_k$ is injective. Proceeding by induction on $k$, we may assume as before that $k\ge 2$. Suppose $\theta,\chi\in\Irr(P_{p^k})$ satisfy $\Phi_k(\theta)=\Phi_k(\chi)=[(T_1\mid\dotsc\mid T_p;\varepsilon)]$ where $\varepsilon\in[0,p]$.
	
	If $\varepsilon\ne p$, then $\theta=\cX(\phi;\phi_\varepsilon)$ and $\chi=\cX(\eta;\phi_\varepsilon)$ for some $\phi,\eta\in\Irr(P_{p^{k-1}})$. Moreover, $T_i\in\Phi_{k-1}(\phi)\cap\Phi_{k-1}(\eta)$ for all $i\in[1,p]$. Since $\Phi_{k-1}(\phi)$ and $\Phi_{k-1}(\eta)$ are orbits, we must have that $\Phi_{k-1}(\phi)=\Phi_{k-1}(\eta)$. By inductive hypothesis, $\Phi_{k-1}$ is injective and so $\phi=\eta$, whence $\theta=\chi$.
	
	If $\varepsilon=p$, then $\theta=\theta_1\times\cdots\times\theta_p\up^{P_{p^k}}$ for some mixed $\theta_i\in\Irr(P_{p^{k-1}})$ and $\chi=\eta_1\times\cdots\times\eta_p\up^{P_{p^{k}}}$ for some mixed $\eta_i\in\Irr(P_{p^{k-1}})$. Furthermore, we have that $T_i\in\Phi_{k-1}(\theta_i)\cap\Phi_{k-1}(\eta_i)$ for each $i\in[1,p]$, after cyclically permuting the $\eta_i$ if necessary. Thus $\theta_i=\eta_i$ for all $i$ by the inductive hypothesis, and we conclude that $\theta=\chi$, as desired.
\end{proof}

\begin{remark}
	We note that $\Phi_k$ is not surjective for any $k\in\N$. Indeed, for $k=1$ we have that $\cT_1 = \{ [\substack{0\\\bullet}], [\substack{1\\\bullet}], \dotsc, [\substack{p-1\\\bullet}], [\substack{p\\\bullet}] \}$ and $\cT_1\setminus\image(\Phi_1) = \{ [\substack{p\\\bullet}] \}$. 
	For $k\ge 2$, we can see for example that $[T]\notin\image(\Phi_k)$ for any tree $T\in\cF_k$ of the form $T=(T'\mid\dotsc\mid T';p)$ where $T'\in\cF_{k-1}$, or any $T$ of the form $T=(T_1\mid\dotsc\mid T_p;\varepsilon)$ where $[T_1], [T_2], \ldots, [T_p]\in\cT_{k-1}$ are mixed and $\varepsilon\in[0,p-1]$.
	\newqed
\end{remark}

Lemma~\ref{lem:tree1} implies that for every $k\in\N$, there exists a unique bijection $\Psi_k: \image(\Phi_k)\longrightarrow \Irr(P_{p^k})$ such that $\Psi_k\circ\Phi_k=\mathrm{Id}_{\Irr(P_{p^k})}$ and $\Phi_k\circ\Psi_k=\mathrm{Id}_{\image(\Phi_k)}$. By induction on $k$, it is routine to check that $\Psi_1([\substack{\varepsilon\\\bullet}])=\phi_\varepsilon$ for every $\varepsilon\in[0,p-1]$, and that for any $k\ge 2$ and any $T=(T_1\mid\dotsc\mid T_p;\varepsilon)$ such that $T_i\in\cF_{k-1}$ and $\varepsilon\in[0,p]$ satisfying $[T]\in\image(\Phi_k)$, one has
\[ \Psi_k([T]) = \begin{cases}
	\cX\big(\Psi_{k-1}([T_1]); \phi_\varepsilon\big) & \text{if }\varepsilon\in[0,p-1],\\
	\big( \Psi_{k-1}([T_1]) \times \cdots \times \Psi_{k-1}([T_p]) \big)\up^{P_{p^k}}_{(P_{p^{k-1}})^{\times p}} & \text{if }\varepsilon=p.
\end{cases} \]

\begin{definition}\label{def:admissible}
	Let $k\in\N$. We say that a tree $T\in\cF_k$ is \emph{admissible} if $[T]\in\image(\Phi_k)$. In this case, we say that $\theta(T):=\Psi_k([T])$ is the irreducible character of $P_{p^k}$ corresponding to $T$. Similarly, given $\theta\in\Irr(P_{p^k})$ we say that $\cT(\theta):=\Phi_k(\theta)$ is the $\Pi_{k-1}$-orbit of trees associated to $\theta$.
\end{definition}

\begin{remark}
	\begin{enumerate}[label=(\roman*)]
		\item From the definition of $\Phi_k$ we can immediately see that if $T\in\cF_k$ is an admissible tree and $x$ is a vertex of $T$ whose label is $\varepsilon\in[0,p]$, then $\varepsilon\ne p$ whenever $x$ is a leaf. On the other hand, if $x$ is not a leaf, then the $p$ subtrees rooted at the children of $x$ are all in the same orbit if and only if $\varepsilon\ne p$. In Tables~\ref{tab:5} and~\ref{tab:25}, we illustrate all admissible trees in $\cF_k$ (up to the action of $\Pi_{k-1}$) where $k=1$ and $k=2$, respectively, for $p=5$.
		
		\item We note that the vertex labels $1,2,\dotsc,p-1$ play very similar roles since they correspond to the non-trivial irreducible characters of $C_p$, which are all Galois conjugates of one another.
		
		On the other hand, $0$ and $p$ reflect very different representation-theoretic notions. In relation to $\Irr(P_{p^k})$ where $k\in\N$, a vertex label $0$ denotes a trivial character of $C_p$, at some level of the iterated wreath product structure of $P_{p^k}$. 
		
		In contrast, a vertex label $p$ indicates a place where we have induced from a mixed product of irreducible characters of $P_{p^i}$ for some $i<k$, rather than extending some $P_{p^{i+1}}$-invariant product $\phi^{\times p}$ where $\phi\in\Irr(P_{p^i})$, to give an irreducible character of $P_{p^{i+1}}$.
		(Indeed, the label $p$ is technically not required since it simply records when the $p$ subtrees under the vertex which it labels are mixed, but we have decided to include it so that all vertices in our $p$-ary trees are labelled.)
		\newqed
	\end{enumerate}
\end{remark}

Now, we record some graph-theoretic notions and fix some useful notation that will appear in the rest of this article.
\begin{notation}
	Let $T$ be a rooted, labelled tree.
	\begin{itemize}
		\item[(i)] We write $x\in T$ to mean that $x$ is a vertex of $T$.
		\item[(ii)] The label of a vertex $x\in T$ is denoted by $\ell(x)$.
		\item[(iii)] We use $d(x)$ to denote the distance between $x\in T$ and the root of $T$.
		\item[(iv)] Between any two vertices in $T$ there exists a unique path since $T$ is a tree. By a \emph{descending path} in $T$ we mean a sequence of vertices $x_1,\dotsc,x_m$ in $T$ such that $x_i$ and $x_{i+1}$ are adjacent and $d(x_{i+1})=d(x_i)+1$ for each $i\in[1, m-1]$. An ascending path is defined analogously. 
		%
	\end{itemize}
\end{notation}

We are now ready to introduce the combinatorial statistics on trees lying at the heart of our main results.
As we have seen via the maps $\Phi_k$ from Definition~\ref{def:Phi-tree}, the vertex labels $0,1,\dotsc,p-1$ on the (admissible) $p$-ary trees in $\cF_k$ relate to the irreducible characters $\phi_0,\phi_1,\dotsc,\phi_{p-1}$ of $C_p$, and we will call those corresponding to non-trivial characters `humans'. All of the tree statistics we introduce below will be based on the patterns of humans and non-humans in these $p$-ary trees.

\begin{definition}\label{def:trees-stats}
	Let $k\in\N$ and let $T$ be a tree in $\cF_k$.
	\begin{enumerate}[label=(\alph*)]\itemspace{5pt}
		\item Let $x$ be a vertex of $T$. Define
		\[ \fv(x) := \begin{cases}
			1 & \text{if}\ \ \ell(x)\in [1,p-1],\\
			0 & \text{if}\ \ \ell(x)\in\{0,p\}.
		\end{cases} \]
		We say that $x$ is a \emph{person} (or \emph{human}) if and only if $\fv(x)=1$, in other words, 
		we call a vertex a \emph{person} if and only if $p$ does not divide its label.
		
		\item Define $\fH(T):=\{x\in T\mid \fv(x)=1 \}$ and $\eta(T)=|\fH(T)|$. 
		In particular, $\fH(T)$ is the set of people (humans) in $T$, and $\eta(T)$ is the number of people in $T$.
		
		
		\item Let $x,y\in\fH(T)$ and $t\in\N_0$. We say that $y$ is a \emph{$t$-descendant} of $x$ if there is a descending path $x=x_0,x_1,\ldots,x_m=y$ in $T$ such that $|\{i\in[1,m] \mid \fv(x_i)=1\}|=t$. 
		
		In other words, $x$ and $y$ must both be people and $y$ is a \emph{$t$-descendant} of $x$ if and only if, among the unique path in $T$ strictly between $x$ and $y$, there are exactly $t-1$ vertices which are people.
		
		\item For $i\in\N_0$, define $\Gamma_i(T):= \{ x\in\fH(T) \mid x\text{ has an }i\text{-descendant but no }(i+1)\text{-descendants} \}$, and let $\gamma_i(T):=|\Gamma_i(T)|$.
		
		\item Finally, we define the \emph{value} of $T$ to be $\val(T):=\min\{i\in\N_0\mid \gamma_i(T)=0 \}$.
	\end{enumerate}
\end{definition}

In particular, $\gamma_0(T)$ counts the number of people (humans) in $T$ which are either leaves or such that no vertex amongst its $p$ subtrees are people. In genealogy terms, $\gamma_0(T)$ counts the number of 
people with no `children' (i.e.~humans with no human descendants). The statistic $\gamma_1(T)$ then counts the number of people who are parents but not grandparents (i.e.~humans with exactly $1$ generation of human descendants, see also Remark~\ref{rmk:tree-stats}).

We illustrate the various notions introduced in Definition~\ref{def:trees-stats} in Figure~\ref{fig:people}, giving several examples when $p=3$.
Further examples when $p=5$ and $n\in\{5,25,125\}$ are given in Examples~\ref{ex:5,25} and~\ref{ex:125}.

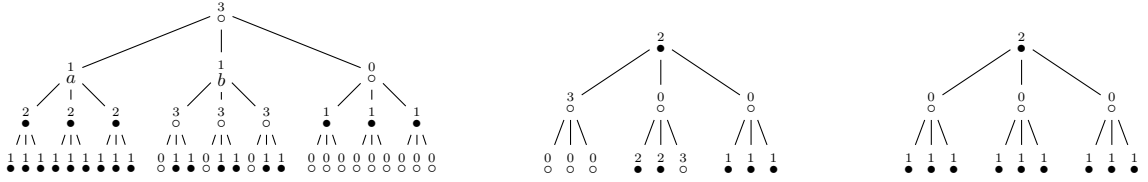
\begin{figure}[h!]
	\centering
	\begin{subfigure}[t]{0.4\textwidth}
		\centering
		\begin{tikzpicture}[scale=1.0, every node/.style={scale=0.7}]
			\draw (0,0.2) node(R) {$\overset{3}{\circ}$};
			\draw (-2,-0.6) node(1) {$\overset{1}{a}$};
			\draw (0,-0.6) node(2) {$\overset{1}{b}$};
			\draw (2,-0.6) node(3) {$\overset{0}{\circ}$};
			\draw (R) -- (1);
			\draw (R) -- (2);
			\draw (R) -- (3);
			\draw (-2.6,-1.2) node(11) {$\overset{2}{\bullet}$};
			\draw (-2,-1.2) node(12) {$\overset{2}{\bullet}$};
			\draw (-1.4,-1.2) node(13) {$\overset{2}{\bullet}$};
			\draw (1) -- (11);
			\draw (1) -- (12);
			\draw (1) -- (13);
			\draw (-0.6,-1.2) node(21) {$\overset{3}{\circ}$};
			\draw (0,-1.2) node(22) {$\overset{3}{\circ}$};
			\draw (0.6,-1.2) node(23) {$\overset{3}{\circ}$};
			\draw (2) -- (21);
			\draw (2) -- (22);
			\draw (2) -- (23);
			\draw (1.4,-1.2) node(31) {$\overset{1}{\bullet}$};
			\draw (2,-1.2) node(32) {$\overset{1}{\bullet}$};
			\draw (2.6,-1.2) node(33) {$\overset{1}{\bullet}$};
			\draw (3) -- (31);
			\draw (3) -- (32);
			\draw (3) -- (33);
			\draw (-2.8,-1.8) node(111) {$\overset{1}{\bullet}$};
			\draw (-2.6,-1.8) node(112) {$\overset{1}{\bullet}$};
			\draw (-2.4,-1.8) node(113) {$\overset{1}{\bullet}$};
			\draw (11) -- (111);
			\draw (11) -- (112);
			\draw (11) -- (113);
			\draw (-2.2,-1.8) node(121) {$\overset{1}{\bullet}$};
			\draw (-2,-1.8) node(122) {$\overset{1}{\bullet}$};
			\draw (-1.8,-1.8) node(123) {$\overset{1}{\bullet}$};
			\draw (12) -- (121);
			\draw (12) -- (122);
			\draw (12) -- (123);
			\draw (-1.6,-1.8) node(131) {$\overset{1}{\bullet}$};
			\draw (-1.4,-1.8) node(132) {$\overset{1}{\bullet}$};
			\draw (-1.2,-1.8) node(133) {$\overset{1}{\bullet}$};
			\draw (13) -- (131);
			\draw (13) -- (132);
			\draw (13) -- (133);
			\draw (-0.8,-1.8) node(211) {$\overset{0}{\circ}$};
			\draw (-0.6,-1.8) node(212) {$\overset{1}{\bullet}$};
			\draw (-0.4,-1.8) node(213) {$\overset{1}{\bullet}$};
			\draw (21) -- (211);
			\draw (21) -- (212);
			\draw (21) -- (213);
			\draw (-0.2,-1.8) node(221) {$\overset{0}{\circ}$};
			\draw (0,-1.8) node(222) {$\overset{1}{\bullet}$};
			\draw (0.2,-1.8) node(223) {$\overset{1}{\bullet}$};
			\draw (22) -- (221);
			\draw (22) -- (222);
			\draw (22) -- (223);
			\draw (0.4,-1.8) node(231) {$\overset{0}{\circ}$};
			\draw (0.6,-1.8) node(232) {$\overset{1}{\bullet}$};
			\draw (0.8,-1.8) node(233) {$\overset{1}{\bullet}$};
			\draw (23) -- (231);
			\draw (23) -- (232);
			\draw (23) -- (233);
			\draw (1.2,-1.8) node(311) {$\overset{0}{\circ}$};
			\draw (1.4,-1.8) node(312) {$\overset{0}{\circ}$};
			\draw (1.6,-1.8) node(313) {$\overset{0}{\circ}$};
			\draw (31) -- (311);
			\draw (31) -- (312);
			\draw (31) -- (313);
			\draw (1.8,-1.8) node(321) {$\overset{0}{\circ}$};
			\draw (2,-1.8) node(322) {$\overset{0}{\circ}$};
			\draw (2.2,-1.8) node(323) {$\overset{0}{\circ}$};
			\draw (32) -- (321);
			\draw (32) -- (322);
			\draw (32) -- (323);
			\draw (2.4,-1.8) node(331) {$\overset{0}{\circ}$};
			\draw (2.6,-1.8) node(332) {$\overset{0}{\circ}$};
			\draw (2.8,-1.8) node(333) {$\overset{0}{\circ}$};
			\draw (33) -- (331);
			\draw (33) -- (332);
			\draw (33) -- (333);
		\end{tikzpicture}
		\caption{An admissible tree $T\in\cF_4$ with $\val(T)=3$ and $\eta(T)=23$. The vertex $a$ is a person and has two generations of `human' descendants, while $b$ only has one. Also, $\gamma_0(T)=18$, $\gamma_1(T)=4$, $\gamma_2(T)=1$ and $\gamma_i(T)=0$ for all $i\ge 3$.}
		\label{fig:people-1}
	\end{subfigure}
	\hfill
	\begin{subfigure}[t]{0.27\textwidth}
		\centering
		\begin{tikzpicture}[scale=1.0, every node/.style={scale=0.7}]
			\draw (0,0.2) node(R) {$\overset{2}{\bullet}$};
			\draw (-1.2,-0.6) node(1) {$\overset{3}{\circ}$};
			\draw (0,-0.6) node(2) {$\overset{0}{\circ}$};
			\draw (1.2,-0.6) node(3) {$\overset{0}{\circ}$};
			\draw (R) -- (1);
			\draw (R) -- (2);
			\draw (R) -- (3);
			\draw (-1.5,-1.4) node(11) {$\overset{0}{\circ}$};
			\draw (-1.2,-1.4) node(12) {$\overset{0}{\circ}$};
			\draw (-0.9,-1.4) node(13) {$\overset{0}{\circ}$};
			\draw (1) -- (11);
			\draw (1) -- (12);
			\draw (1) -- (13);
			\draw (-0.3,-1.4) node(21) {$\overset{2}{\bullet}$};
			\draw (0,-1.4) node(22) {$\overset{2}{\bullet}$};
			\draw (0.3,-1.4) node(23) {$\overset{3}{\circ}$};
			\draw (2) -- (21);
			\draw (2) -- (22);
			\draw (2) -- (23);
			\draw (0.9,-1.4) node(31) {$\overset{1}{\bullet}$};
			\draw (1.2,-1.4) node(32) {$\overset{1}{\bullet}$};
			\draw (1.5,-1.4) node(33) {$\overset{1}{\bullet}$};
			\draw (3) -- (31);
			\draw (3) -- (32);
			\draw (3) -- (33);
		\end{tikzpicture}
		\caption{An inadmissible tree $T\in\cF_3$ with $\val(T)=2$, $\eta(T)=6$, $\gamma_0(T)=5$ and $\gamma_1(T)=1$.}
		\label{fig:people-2}
	\end{subfigure}
	\hfill
	\begin{subfigure}[t]{0.27\textwidth}
		\centering
		\begin{tikzpicture}[scale=1.0, every node/.style={scale=0.7}]
			\draw (0,0.2) node(R) {$\overset{2}{\bullet}$};
			\draw (-1.2,-0.6) node(1) {$\overset{0}{\circ}$};
			\draw (0,-0.6) node(2) {$\overset{0}{\circ}$};
			\draw (1.2,-0.6) node(3) {$\overset{0}{\circ}$};
			\draw (R) -- (1);
			\draw (R) -- (2);
			\draw (R) -- (3);
			\draw (-1.5,-1.4) node(11) {$\overset{1}{\bullet}$};
			\draw (-1.2,-1.4) node(12) {$\overset{1}{\bullet}$};
			\draw (-0.9,-1.4) node(13) {$\overset{1}{\bullet}$};
			\draw (1) -- (11);
			\draw (1) -- (12);
			\draw (1) -- (13);
			\draw (-0.3,-1.4) node(21) {$\overset{1}{\bullet}$};
			\draw (0,-1.4) node(22) {$\overset{1}{\bullet}$};
			\draw (0.3,-1.4) node(23) {$\overset{1}{\bullet}$};
			\draw (2) -- (21);
			\draw (2) -- (22);
			\draw (2) -- (23);
			\draw (0.9,-1.4) node(31) {$\overset{1}{\bullet}$};
			\draw (1.2,-1.4) node(32) {$\overset{1}{\bullet}$};
			\draw (1.5,-1.4) node(33) {$\overset{1}{\bullet}$};
			\draw (3) -- (31);
			\draw (3) -- (32);
			\draw (3) -- (33);
		\end{tikzpicture}
		\caption{An admissible tree $T\in\cF_3$ with $\val(T)=2$, $\eta(T)=10$, $\gamma_0(T)=9$ and $\gamma_1(T)=1$.}
		\label{fig:people-3}
	\end{subfigure}
	\caption{Examples illustrating Definition~\ref{def:trees-stats} with $p=3$. 
		Here, and in subsequent figures, filled circles $\bullet$ indicate (unnamed) vertices in $\fH(T)$, i.e.~those with label in $[1,p-1]$, and unfilled circles $\circ$ indicate (unnamed) vertices not in $\fH(T)$, i.e.~those with label in $\{0,p\}$. The label of each vertex is written above it. This is to aid with visualising and counting which vertices are `people' (Definition~\ref{def:trees-stats}(a)) in our figures illustrating such labelled trees.}
	\label{fig:people}
\end{figure}

\begin{remark}\label{rmk:tree-stats}
	\begin{itemize}
		\item[(i)] Every $x\in\fH(T)$ is its own unique 0-descendant. 		
		Moreover, in Definition~\ref{def:trees-stats}(d) notice that $x_1,\dotsc,x_{m-1}$ need not be people; indeed, $y$ being a $t$-descendant of $x$ means that exactly $t-1$ of $x_1,\dotsc,x_{m-1}$ are people.		
		Informally, $y$ is in the $t$-th generation of `human' descendants of $x$ (we do not includes vertices that are not people for the purposes of counting such generations).
		
		\item[(ii)] Since $T$ has finite height, clearly $\{i\in\N_0\mid\gamma_i(T)=0\}\ne\emptyset$ and the value $\val(T)$ of $T$ is well-defined. Furthermore, we observe that if $\gamma_i(T)=0$, then $\gamma_{i+1}(T)=0$. 
		
		Indeed, suppose $\gamma_{i+1}(T)>0$, so there is some $x\in T$ with a $(i+1)$-descendant, say $y\in T$, and $x$ has no $(i+2)$-descendants. Let $x=x_0,x_1,\dotsc,x_m=y$ be the unique path in $T$ between $x$ and $y$, so $|\{ l\in[1,m]\mid \fv(x_l)=1 \}|=i+1$. Suppose $j=\min\{l\in[1,m]\mid x_l\in\fH(T)\}$. Then $x_j\in\Gamma_i(T)$ because $y$ is an $i$-descendant of $x_j$, but if $x_j$ has any $(i+1)$-descendants then those would be $(i+2)$-descendants of $x$. We have thus shown that $\gamma_{i+1}(T)>0$ implies $\gamma_i(T)>0$.
		
		\item[(iii)] We note that $\{ \Gamma_i(T) \mid i\in\N_0 \}$ gives a partition of $\fH(T)$, so the number of people in $T$ equals
		$\eta(T)=\sum_{i=0}^\infty\gamma_i(T)$. With this in mind, viewing the subset of those vertices in $\fH(T)$ as forming a `family tree', the value $\val(T)$ is then the number of different generations of people in $T$. (Here, people in the same generation may have different distances to the root of $T$.) 
		
		For example, $\val(T)=0$ if and only if $T$ has no people in it, while $\val(T)=1$ if and only if there is at least one person in $T$ but no person in $T$ has any 1-descendants. In particular, when $\val(T)\le 1$ then $\eta(T)=\gamma_0(T)$.
		\newqed
	\end{itemize}
\end{remark}

\begin{lemma}\label{lem:tree-bounds}
	Let $k\in\N$ and $T\in\cF_k$. Then the following hold:
	\begin{itemize}\itemspace{5pt}
		\item[(a)] $\gamma_i(T)\le p^{k-1-i}$ for all $i\in[0,k-1]$.
		\item[(b)] If $T$ is an admissible tree and $\gamma_0(T)<p$, then $\val(T)\le 1$.
	\end{itemize}
\end{lemma}

\begin{proof}
	\begin{itemize}
		\item[(a)] If $x\in\Gamma_i(T)$ then $d(x)\le k-1-i$ since $x$ must have an $i$-descendant. Consider the $p^{k-1-i}$ vertices $y\in T$ such that $d(y)=k-1-i$. Clearly each path between the root of $T$ and such a vertex $y$ intersects $\Gamma_i(T)$ in at most one vertex. Thus $\gamma_i(T)\le p^{k-1-i}$.
		\item[(b)] Suppose for a contradiction that $\val(T)\ge 2$. Then $\gamma_1(T)\ne 0$, so there exists some $x\in\Gamma_1(T)$. Suppose the subtree of $T$ rooted at $x$ is $(T_1\mid\dotsc\mid T_p;x)$. Since $x\in\Gamma_1(T)$ there must exist some $i\in[1,p]$ and a vertex $y_i$ in $T_i$ such that $\ell(y_i)\in[1,p-1]$. Since $T$ is admissible and since $\ell(x)\ne p$ we must have $[T_1]=[T_2]=\cdots=[T_p]$. It follows that for every $j\in[1,p]$ there exists a vertex $y_j\in T_j$ such that $\ell(y_j)=\ell(y_i)\in[1,p-1]$. We conclude that $y_1,y_2,\dotsc,y_p\in\Gamma_0(T)$ and so $\gamma_0(T)\ge p$.
	\end{itemize}
\end{proof}

Now, we wish to associate the tree statistics introduced in Definition~\ref{def:trees-stats} to characters of the Sylow subgroups $P_{p^k}$.
Suppose $T,T'\in\cF_k$ are trees lying in the same $\Pi_{k-1}$-orbit. Then $\eta(T)=\eta(T')$, $\gamma_i(T)=\gamma_i(T')$ for all $i\in\N_0$, and $\val(T)=\val(T')$, since the action of $\Pi_{k-1}$ preserves adjacency of vertices. This guarantees that the following quantities are well-defined.

\begin{definition}\label{def:theta-stats}
	Let $k\in\N$ and $\theta\in\Irr(P_{p^k})$. We define $\eta(\theta):=\eta(T)$, $\gamma_i(\theta):=\gamma_i(T)$ for all $i\in\N_0$ and $\val(\theta):=\val(T)$ where $T$ is any tree in $\cT(\theta)$.
\end{definition}

The following result is particularly interesting for our purposes, as it allows us to drastically reduce the irreducible characters of $P_{p^k}$ we need to analyse in order to study the Sylow branching coefficients for symmetric groups. In particular, we can see that `all people are equal'. That is to say, mathematically: let $\theta\in\Irr(P_{p^k})$ and $T\in\cT(\theta)$. Suppose that $T$ has a vertex labelled by $i\in[1,p-1]$. If when we replace the label $i$ by some $j\in[1,p-1]$ we obtain an admissible tree $T'$, corresponding to some character $\zeta\in\Irr(P_{p^k})$, 
then we will show that $\theta\up^{S_{p^k}} = \zeta\up^{S_{p^k}}$. Of course, we may iterate this and perform any number of such label replacements. We first describe the resulting equivalence relation on characters, before proceeding with proving that their inductions to $S_{p^k}$ are equal.

\begin{definition}\label{def:label-replace}
	Let $k\in\N$.
	\begin{itemize}
		\item[(a)] Suppose $T,T'\in\cF_k$. We say that $T'$ is obtained from $T$ by a \emph{good label replacement} if $T'$ may be obtained from $T$ by selecting a vertex $x\in T$ such that $i:=\ell(x)\in[1,p-1]$, and replacing $i$ by some $j\in[1,p-1]$.
		
		\item[(b)] Given $\theta,\zeta\in\Irr(P_{p^k})$, we say that $\theta\sim_{0,p}\zeta$ if some $T'\in\cT(\zeta)$ may be obtained from some $T\in\cT(\theta)$ by a sequence of good label replacements. 
	\end{itemize}
\end{definition}

In other words, $\theta\sim_{0,p}\zeta$ if and only if there is some $T\in\cT(\theta)$ and $T'\in\cT(\zeta)$ such that $T$ and $T'$ have label 0 in exactly the same places as each other, and also $T$ and $T'$ have label $p$ in the exactly the same places as each other. 
Clearly $\sim_{0,p}$ is well-defined and an equivalence relation on $\Irr(P_{p^k})$, and its importance is highlighted by the following result.

\begin{theorem}\label{thm:equiv-classes}
	Let $k\in\N$ and $\theta,\zeta\in\Irr(P_{p^k})$. If $\theta\sim_{0,p}\zeta$, then $\theta\up^{S_{p^k}}=\zeta\up^{S_{p^k}}$. In particular, $\Omega(\theta)=\Omega(\zeta)$.
\end{theorem}

\begin{proof}
	We proceed by induction on $k$. If $k=1$ then the fact that $\phi_i\up^{S_p}=\phi_j\up^{S_p}$ for all $i,j\in[1,p-1]$ follows from a straightforward application of the Murnaghan--Nakayama rule (see \cite[Corollary 2.4]{SLThesis}). Let us now assume $k\ge 2$. To ease the notation when considering various subgroups of $S_{p^k}$, we let $P:=P_{p^k}$, $Q:=P_{p^{k-1}}$ and $R:=S_{p^{k-1}}$. Further, let $Y:=(S_{p^{k-1}})^{\times p}$, $W=Y\rtimes P_p\cong S_{p^{k-1}}\wr P_p$ and $V:=Y\rtimes S_p\cong S_{p^{k-1}}\wr S_p$.
	
	If $\theta=\theta_1\times\cdots\times\theta_p\up^P$ for some mixed $\theta_1,\dotsc,\theta_p\in\Irr(Q)$, then the root of any $T\in\cT(\theta)$ is labelled by $p$. 
	Since $\theta\sim_{0,p}\zeta$, it follows that the root of any $T'\in\cT(\zeta)$ is also labelled by $p$, and so $\zeta=\zeta_1\times\cdots\times\zeta_p\up^P$ for some mixed $\zeta_1,\dotsc,\zeta_p\in\Irr(Q)$. Moreover, from $\theta\sim_{0,p}\zeta$ we also deduce that $\theta_i\sim_{0,p}\zeta_i$ for all $i\in[1,p]$, after cyclically permuting the $\zeta_i$ if necessary.
	By inductive hypothesis, $\theta_i\up^R=\zeta_i\up^R$ for all $i$. 
	Since $\theta\up^{S_{p^k}} = (\theta_1\times\cdots\times\theta_p\up^Y)\up^{S_{p^k}} = (\theta_1\up^R\times\cdots\times\theta_p\up^R)\up^{S_{p^k}}$, and similarly for $\zeta$, we conclude that $\theta\up^{S_{p^k}}=\zeta\up^{S_{p^k}}$.
	
	If $\theta=\cX(\theta_1;\phi_0)$ for some $\theta_1\in\Irr(Q)$, then the root of any $T\in\cT(\theta)$ is labelled by 0. Then, arguing as above, we deduce that $\zeta=\cX(\zeta_1;\phi_0)$ for some $\zeta_1\in\Irr(Q)$. Moreover, $\theta_1\sim_{0,p}\zeta_1$, and so $\theta_1\up^R=\zeta_1\up^R$. Hence
	\[ \theta\up^W = \cX(\theta_1;\phi_0)\up_P^W = \cX(\theta_1\up_Q^R;\phi_0) = \cX(\zeta_1\up_Q^R;\phi_0) = \cX(\zeta_1;\phi_0)\up_P^W = \zeta\up^W, \]
	where the second equality follows from \cite[Lemma 3.2]{CT}, and we conclude that $\theta\up^{S_{p^k}}=\zeta\up^{S_{p^k}}$.
	
	Finally, if $\theta=\cX(\theta_1;\phi_i)$ for some $i\in[1,p-1]$ and $\theta_1\in\Irr(Q)$, then the root of any $T\in\cT(\theta)$ is labelled by $i$. Then we deduce that $\zeta=\cX(\zeta_1;\phi_j)$ for some $j\in[1,p-1]$ and $\zeta_1\in\Irr(Q)$ and $\theta_1\sim_{0,p}\zeta_1$. Thus
	\[ \theta\up^V = \cX(\theta_1;\phi_i)\up_P^V = \cX(\theta_1\up_Q^R;\phi_i\up_{P_p}^{S_p}) = \cX(\zeta_1\up_Q^R;\phi_j\up_{P_p}^{S_p}) = \cX(\zeta_1;\phi_j)\up_P^V = \zeta\up^V, \]
	where again the second equality follows from \cite[Lemma 3.2]{CT}, and so $\theta\up^{S_{p^k}}=\zeta\up^{S_{p^k}}$.
\end{proof}

\begin{remark}
	While good label replacements can be defined for all trees $T\in\cF_k$, an admissible tree need not remain admissible after such a replacement. For instance, $T=(\substack{2\\\bullet}\mid\substack{2\\\bullet}\mid\substack{2\\\bullet};1)\in\cF_2$ is admissible (where $p=3$), but $T'=(\substack{1\\\bullet}\mid\substack{2\\\bullet}\mid\substack{2\\\bullet};1)$ which differs from $T$ by a good label replacement is not. 
	
	The hypothesis that both $\theta$ and $\zeta$ belong to $\Irr(P_{p^k})$ in Definition~\ref{def:label-replace} and Theorem~\ref{thm:equiv-classes} guarantees that the trees $T$ and $T'$ considered therein are admissible. 
	\newqed
\end{remark}

A key purpose of Theorem \ref{thm:equiv-classes} is to observe that $\Omega(\theta)=\Omega(\zeta)$ whenever $\theta\sim_{0,p}\zeta$. A second, natural equivalence relation on $\Irr(P_{p^k})$ that preserves the sets $\Omega(\theta)$ is the one induced by Galois conjugation. 
We now briefly explain the connection between these two equivalence relations. 
In particular, we will see that if $\theta$ and $\zeta$ are Galois conjugates then $\theta\sim_{0,p}\zeta$, showing that each $\sim_{0,p}$-equivalence class is a union of equivalence classes under Galois conjugacy, and therefore that the equivalence relation $\sim_{0,p}$ is more useful for the purposes of this article. 

Indeed, it is well known that given any $\theta\in\Irr(P_{p^k})$, the corresponding field of values $\Q(\theta)$ is a subfield of $\Q(\omega)$, where $\omega=e^{\frac{2\pi i}{p}}$ is a primitive $p$-th root of unity. (A proof of this fact can be obtained using \cite[Lemma 4.3.9]{JK} and induction on the parameter $k$.)
Setting $\cG:=\Gal(\Q(\omega)|\Q)$, we have that $\cG$ is a cyclic group of order $p-1$, acting naturally on $\Irr(P_{p^k})$. In particular, $(\theta^{\sigma})\up^{S_{p^k}}=\theta\up^{S_{p^k}}$ for any $\theta\in\Irr(P_{p^k})$ and any $\sigma\in\cG$.
In the following remark, we describe the effect of Galois conjugation on the trees corresponding to irreducible characters.

\begin{remark}
	The Galois group $\cG$ acts on $\Irr(P_p)$ by fixing the trivial character $\phi_0$ and by transitively permuting the remaining $p-1$ linear characters $\phi_1,\phi_2,\ldots, \phi_{p-1}$. Let $\sigma$ be a generator of $\cG$ and let us denote by $\tau$ the element of $S_{p-1}$ such that $\phi_j^{\sigma}=\phi_{j\tau}$, for all $j\in [1,p-1]$.
	
	Now, for any $k\in\N$ and any $\theta\in\Irr(P_{p^k})$, let us fix $T\in\cT(\theta)$ and denote by $T\cdot\tau$ the tree obtained from $T$ by replacing every label $j$ with $j\tau$, for any $j\in [1,p-1]$. Notice that $T\cdot\tau$ is obtained from $T$ by a sequence of good label replacements, since no action was taken on vertices labelled by $0$ or by $p$. 
	It is now enough to proceed by induction on the parameter $k$ to show that $T\cdot\tau$ is an element of $\cT(\theta^\sigma)$. As an immediate consequence of the above discussion we deduce that $\theta\sim_{0,p}\theta^\sigma$. 
	We have omitted the details of this proof as this result is not directly relevant to the main goals of this article. 
	
	We conclude with a small example to highlight the properties described above. For the prime $p=3$ we have that $\cG=\langle\sigma\rangle\cong C_2$ and $\tau=(1,2)\in S_2$. Let  $T=(\substack{0\\\bullet}\mid\substack{1\\\bullet}\mid\substack{2\\\bullet};3)\in\cF_2$ and let $\theta=\theta(T)$ be the corresponding irreducible character of $P_9$. Its $\cG$-conjugacy class is $\{\theta, \theta^\sigma\}$ and a tree corresponding to $\theta^\sigma$ is given by 
	\[ T\cdot\tau=(\substack{0\\\bullet}\mid\substack{2\\\bullet}\mid\substack{1\\\bullet};3). \]
	On the other hand, we observe that the $\sim_{0,p}$-equivalence class of $\theta$ is the set $\{\theta, \theta^\sigma, \eta, \eta^\sigma\}$, where $\eta$ and $\eta^\sigma$ are the irreducible characters of $P_9$ whose corresponding trees are, respectively,
	\[ S=(\substack{0\\\bullet}\mid\substack{1\\\bullet}\mid\substack{1\\\bullet};3)\ \text{and}\ \ S\cdot\tau=(\substack{0\\\bullet}\mid\substack{2\\\bullet}\mid\substack{2\\\bullet};3). \]
	\newqed
\end{remark}

We conclude this section by extending Definitions~\ref{def:admissible} and~\ref{def:trees-stats} from powers of $p$ to arbitrary natural numbers. 

\begin{definition}\label{def:trees-arbitrary}
	Let $n\in\N$ and let $n=\sum_{i=1}^t p^{n_i}$ be the $p$-adic expansion of $n$. Let $\theta\in\Irr(P_n)$, so $\theta=\theta_1\times\cdots\times\theta_t$ for some uniquely determined $\theta_i\in\Irr(P_{p^{n_i}})$ for each $i\in[1,t]$. We let 
		\[ \cT(\theta) := \cT(\theta_1)\times\cT(\theta_2)\times \cdots\times\cT(\theta_t). \]
	By $T\in\cT(\theta)$, we mean a $t$-tuple of trees $T=(T_1,\dotsc,T_t)$ such that $T_i\in\cT(\theta_i)$ for each $i$. By the set of vertices of $T$, we mean the disjoint union of the vertices of $T_1,\dotsc,T_t$. Moreover, we define
	\[ \eta(\theta):=\sum_{i=1}^t\eta(\theta_i),\quad \val(\theta)=\max_{i\in[1,t]} \val(\theta_i) \quad \text{and} \quad \gamma_y(\theta):=\sum_{i=1}^t\gamma_y(\theta_i)\ \forall\ y\in\N_0. \]
\end{definition}

A first piece of algebraic information on the irreducible characters of $P_n$ that can be easily read off their corresponding trees is the character degree. 

\begin{proposition}\label{prop:arrows}
	Let $n\in\N$, let $\theta\in\Irr(P_n)$ and $T\in\cT(\theta)$. Let $a=|\{x\in T\mid\ell(x)=p\}|$ denote the number of vertices in $T$ labelled by $p$.
	Then $\theta(1)=p^a$.
\end{proposition}

\begin{proof}
	We first consider the case where $n=p^k$ for some $k\in\N$, and proceed by induction on $k$. If $k=1$ the assertion is clear because $a=0$ and $\theta(1)=1$. Assume now that $k\ge 2$ and suppose that $T=(T_1\mid \dotsc\mid\ldots|T_p;\varepsilon)$ for some $T_1,\ldots, T_p\in\cF_{k-1}$ and $\varepsilon\in[0,p]$.
	
	If $\varepsilon\neq p$, then $[T_1]=[T_2]=\cdots=[T_p]$. Let $\alpha$ be the number of vertices labelled by $p$ in $T_1$ and let $\zeta\in\Irr(P_{k-1})$ be the character corresponding to $T_1$, i.e.~$\zeta=\theta(T_1)=\Psi_{k-1}([T_1])$.
	Given any $i\in [1,p]$, it is easy to see that $\alpha$ is the number of vertices labelled by $p$ in $T_i$ because $[T_1]=[T_i]$. It follows that
	 $a=p\cdot\alpha$, and by inductive hypothesis we have that $\zeta(1)=p^{\alpha}$. Since $\theta=\cX(\zeta;\varphi_{\varepsilon})$, we conclude that $\theta(1)=\zeta(1)^p=p^{p\cdot\alpha}=p^a$, as desired. 
	
	If $\varepsilon=p$, then let $\alpha_i$ be the number of vertices of $T_i$ that are labelled by $p$ and let $\zeta_i:=\theta(T_i)\in\Irr(P_{p^{k-1}})$ for each $i\in[1,p]$. Then $\theta=(\zeta_1\times\cdots\times\zeta_p)\up^{P_{p^k}}$, and $a=\alpha_1+\cdots+\alpha_p+1$. We conclude using the inductive hypothesis that $\theta(1) = |P_{p^k}:(P_{p^{k-1}})^{\times p}| \cdot\zeta_1(1)\cdot\zeta_2(1)\cdots\zeta_p(1)=p\cdot p^{\alpha_1}\cdot p^{\alpha_2}\cdots p^{\alpha_p}=p^a$.
	
	The assertion for general $n\in\N$ then follows immediately.
\end{proof}

\bigskip

\section{The Structure of the Set $\Omega(\theta)$: Prime Power Case}\label{sec:pk}

In this section, we fix a prime number $p\ge 5$, and let $n=p^k$ for some $k\in\N$.
We investigate the structure of the sets $\Omega(\theta)$ for every $\theta\in\Irr(P_{p^k})$. 
In Section~\ref{sec:arbitrary-n}, we will then extend our results to any arbitrary $n\in\N$. 
From now on we denote by $\triv_n$ (instead of $\triv_{P_n}$) the trivial character of the Sylow $p$-subgroup $P_n$.

We start by recalling a number of useful results that we will use throughout this paper. We include their statements below in our present notation for the convenience of the reader. 
The first one was obtained in \cite[Theorem A]{GL1}, and it can be considered as the starting point of this entire line of research. 
\begin{theorem}\label{thm:GL1}
	Let $p\ge 5$ be a prime, $n\in\N$ and $P_n\in\Syl_p(S_n)$. Then
	\[ \Omega(\triv_n) = \begin{cases}
		\puncp{p^k} & \text{if }n=p^k\text{ for some }k\in\N,\\
		\cP(n) & \text{otherwise}.
	\end{cases} \]
\end{theorem}

Theorem~\ref{thm:GL1} describes $\Omega(\theta)$ for the specific case where $\theta$ is the trivial character of $P_n$. The set $\Omega(\triv_n)$ was also determined for $p=3$ in \cite{GL1}, although this will not be needed for the present paper as we will only consider $p\ge 5$ here.

The next step in this investigation for $p\ge 5$ was done in \cite{GL2}, where the authors considered $\Omega(\theta)$ for every linear character $\theta$ of $P_n$ (see \cite{L} for the case of $p=3$). In particular, \cite[Lemma 4.3]{GL2} deals with a specific family of linear characters. 

\begin{lemma}\label{lem:GL2-4.3}
	Let $p\ge 5$ be prime and $k\in\N_{\geq 2}$. Let $\theta=\cX(\triv_{p^{k-1}};\phi_\varepsilon)$ where $\varepsilon\in[1, p-1]$. Then $\Omega(\theta)=\cB_{p^{k}}(p^{k}-1)$ and $[\chi^{(p^{k}-1,1)}\down_{P_{p^{k}}},\theta]=1$.
\end{lemma}

The next two results concern character restrictions and inductions involving wreath products of symmetric groups.
We let $\Irr_{p'}(G)$ denote the set of irreducible characters of $G$ whose degree is coprime to $p$, and we recall that $\Irr_{p'}(S_{p^k})$ consists precisely of those irreducible characters labelled by hook partitions of $p^k$. 

\begin{theorem}[{\cite[Corollary 9.1]{PW}}]\label{thm:PW9.1}
	Let $m,n\in\N$, $\mu=(\mu_1,\ldots,\mu_k)\in \cP(m)$ and $\nu=(\nu_1,\ldots,\nu_l)\in \cP(n)$. Then the lexicographically greatest $\lambda\in \cP(mn)$ such that $[\chi^\lambda, \cX(\chi^\mu;\chi^\nu)\up_{S_m\wr S_n}^{S_{mn}}] \ne 0$ is
	\[ \lambda=(n\mu_1,\ldots,n\mu_{k-1},n(\mu_k-1)+\nu_1,\nu_2,\ldots,\nu_l). \]
	Moreover, $[ \chi^\lambda, \cX(\chi^\mu;\chi^\nu)\up_{S_m\wr S_n}^{S_{mn}}] =1$.
\end{theorem}

\begin{theorem}[{\cite[Theorem 3.5]{GTT}}]\label{thm:GTT3.5}
	Let $p$ be an odd prime and $k\in\N$. Let $K:=S_{p^{k-1}}\wr S_p\le S_{p^k}$ and $\chi\in\Irr_{p'}(S_{p^k})$. Then $\chi\down_K$ has a unique irreducible constituent $\chi^*$ lying in $\Irr_{p'}(K)$, and $[\chi\down_K,\chi^*]=1$. Moreover, 
	\[ \Irr_{p'}(S_{p^k})\to\Irr_{p'}(K),\qquad \chi\mapsto\chi^* \]
	is a bijection.	
	More precisely, if $\chi=\chi^\lambda$ where 
	$\lambda=(p^k-(mp+x),1^{mp+x})$ with $x\in\{0,1,\ldots,p-1\}$, then $\chi^*\in\{\cX(\mu;\nu_1), \cX(\mu;\nu_2) \}$ where
	\[ \mu=(p^{k-1}-m,1^m),\quad \nu_1=(p-x,1^x)\quad\text{and}\quad\nu_2=\nu_1'. \]
\end{theorem}

It is also convenient to state the following easy consequence of Mackey's Theorem (see \cite[Problem 5.2]{IBook}),
as it will be used frequently in some of our proofs. 

\begin{lemma}\label{lem:mackey}
	Let $G$ be a finite group. Suppose $H,K\le G$ satisfy $KH=G$. Let $\phi\in\Char(H)$. Then 
	\[ \phi\up_H^G\down_K = \phi\down_{H\cap K}^H\up^K. \]
\end{lemma}

We record some notation that we will use throughout in this section. 

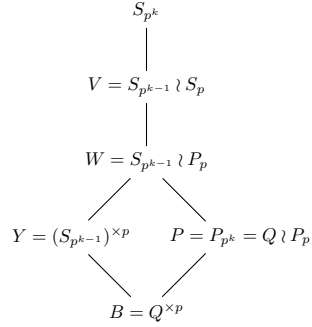
\begin{figure}[h!]
	\centering
		\begin{tikzpicture}[scale=1.0, every node/.style={scale=0.7}]
			\draw (0,0) node(S) {$S_{p^k}$};
			\draw (0,-0.7) node(V) {$V=S_{p^{k-1}}\wr S_p$};
			\draw (0,-1.4) node(W) {$W=S_{p^{k-1}}\wr P_p$};
			\draw (-1,-2) node(Y) {$Y=(S_{p^{k-1}})^{\times p}$\qquad};
			\draw (1,-2) node(P) {\qquad$P=P_{p^k}=Q\wr P_p$};
			\draw (0,-2.6) node(B) {$B=Q^{\times p}$};
			\draw (S) -- (V) -- (W) -- (Y) -- (B) -- (P) -- (W);
		\end{tikzpicture}
		\caption{Subgroups of $S_{p^k}$, where $Q:=P_{p^{k-1}}$}\label{fig:subgps}
\end{figure}

\begin{notation}\label{not:subgps}
	We let $P:=P_{p^k}$ and $Q:=P_{p^{k-1}}$. Recall from Section~\ref{sec:trees} that $P=Q^{\times p}\rtimes P_p = Q\wr P_p$: we denote by $B$ the base group of this wreath product, so that the subgroup $B$ of $P$ satisfies $B\cong Q^{\times p}$.
	We further take $Y$ to be the Young subgroup of $S_{p^k}$ acting on $[1,p^k]$ with the same orbits as $B$. Clearly $Y\cong (S_{p^{k-1}})^{\times p}$, and $B=Y\cap P$.
	We let $W=Y\rtimes P_p\cong S_{p^{k-1}}\wr P_p$ and observe that $P\leq W$. 
	Finally, we let $V=Y\rtimes S_p\cong S_{p^{k-1}}\wr S_p$. These subgroups and their inclusions are illustrated in Figure~\ref{fig:subgps}.
\end{notation}

When dealing with irreducible characters of symmetric groups, we will often write $\lambda$ to denote both the partition $\lambda$ and the irreducible character $\chi^\lambda$. The meaning will always be clear from context, and this will allow us to greatly ease the notation and reduce the levels of superscripts and subscripts needed.

\begin{lemma}\label{lem:mixed}
	Let $\theta\in\Irr(P)$. Suppose $\theta_1,\dotsc,\theta_p\in\Irr(Q)$ satisfy $\theta\in\Irr(P\mid \theta_1\times\cdots\times\theta_p)$.
	Suppose $\lambda\in\cP(p^k)$ satisfies $\chi^\lambda\in\Irr(S_{p^k}\mid \lambda^1\times\cdots\times\lambda^p)$ where $\lambda^i\in\Omega(\theta_i)$	for each $i\in[1,p]$.
	
	If either $\lambda^1,\ldots,\lambda^p$ are mixed or $\theta_1,\ldots,\theta_p$ are mixed, then $\lambda\in\Omega(\theta)$.
	
	In particular, if $\theta=\cX(\theta_1;x)$ for some $\theta_1\in\Irr(P_{p^{k-1}})$ and $x\in[0, p-1]$, then $\cM(p,\Omega(\theta_1))\subseteq\Omega(\theta)$.
\end{lemma}

\begin{proof}
	Let $\overline{\theta}:=\theta_1\times\cdots\times\theta_p\in\Irr(B)$ and $\overline{\lambda}:=\lambda^1\times\cdots\times\lambda^p\in\Irr(Y)$. The assumptions imply that $\overline{\theta}\mid\lambda\down_B$, so there exists $\phi\in\Irr(P\mid\overline{\theta})$ such that $\phi\mid\lambda\down_P$. Suppose first that $\theta_1,\ldots, \theta_p$ are mixed. Then $\Irr(P\mid\overline{\theta})=\{\theta\}$, whence $\theta\mid\lambda\down_P$ and therefore $\lambda\in\Omega(\theta)$.
	
	Let us now suppose that $\theta_i=\theta_1$ for all $i\in [1,p]$. Then $\theta=\cX(\theta_1;\varepsilon)$ for some $\varepsilon\in[0,p-1]$. Moreover, by assumption we must have that $\lambda^1,\ldots,\lambda^p$ are mixed. 
	This implies that $\Irr(W\mid \overline{\lambda})=\{\overline\lambda\up^W\}$. Since $\overline{\lambda}\mid\lambda\down_Y$ by assumption, then $\overline\lambda\up^W\mid \lambda\down_W$.
	Moreover, we observe that $P\cdot Y=W$ and that $P\cap Y=B$. Hence, by Lemma~\ref{lem:mackey} we see that $(\overline{\lambda}\up^W)\down_P = (\overline{\lambda}\down_B)\up^P$. Since $\overline{\theta}\mid\overline{\lambda}\down_B$ by assumption, it follows that $\overline{\theta}\up^P$ is a summand of $\lambda\down_P$.
	Finally, $\overline{\theta}\up^P=\sum_{i=0}^{p-1}\cX(\theta_1;i)$, hence $\theta=\cX(\theta_1;\varepsilon)$ is an irreducible constituent of $\lambda\down_P$. Thus $\lambda\in\Omega(\theta)$.
	
	The second part of the assertion of the lemma then follows immediately. 
\end{proof}

\begin{corollary}\label{cor:omega-mixed}
	Let $\theta_1,\ldots, \theta_p\in\Irr(Q)$ be mixed and let $\phi:=(\theta_1\times\cdots\times\theta_p)\up^P\in\Irr(P)$. Then 
	\[ \Omega(\phi) = \Omega(\theta_1)\star\cdots\star\Omega(\theta_p). \]
\end{corollary}

\begin{proof}
	Write $\overline{\theta}:=\theta_1\times\cdots\times\theta_p\in\Irr(B)$.
	First suppose that $\lambda\in\Omega(\theta_1)\star\cdots\star\Omega(\theta_p)$.
	Then there exists some $\overline{\lambda}\in\Irr(Y\mid\overline{\theta})$ such that $\overline{\lambda}\mid\lambda\down_Y$, and $\overline{\lambda}=\lambda^1\times\cdots\times\lambda^p$ where $\lambda^i\in\Omega(\theta_i)$ for each $i\in[1,p]$.
	Since $\theta_1,\dotsc,\theta_p$ are mixed, it follows that $\lambda\in\Omega(\phi)$ by Lemma~\ref{lem:mixed}.
	
	Conversely, suppose that $\lambda\in\Omega(\phi)$. Since $\overline{\theta}\mid\phi\down_B$ and $\phi\mid\lambda\down_P$, we find that $\overline{\theta}\mid\lambda\down_B$. Hence there exists $\overline{\lambda}\in\Irr(Y\mid\overline{\theta})$ such that $\overline{\lambda}\mid\lambda\down_Y$. But $\overline{\lambda}\in\Irr(Y\mid\overline{\theta})$ implies that $\overline{\lambda}=\lambda^1\times\cdots\times\lambda^p$ for some $\lambda^i\in\cP(p^{k-1})$ such that $\lambda^i\in\Omega(\theta_i)$ for all $i\in[1,p]$. Thus $\overline{\lambda}\mid\lambda\down_P$ implies that $\lambda\in \Omega(\theta_1)\star\cdots\star\Omega(\theta_p)$, as desired.
\end{proof}

Notice that the statement of Corollary~\ref{cor:omega-mixed} does not hold when $\theta_1,\ldots, \theta_p$ are not mixed. Consider the following example, for instance. Suppose $p=5$, $k=2$ and $\phi=\cX(\phi_1;\phi_2)$. Then $\Omega(\phi_1)=\cB_5(4)$ and $\cB_5(4)^{\star 5}=\cB_{25}(20)$, but $\cX(\phi_1;\phi_2)\nmid (20,5)\down_{P_{25}}$, as shown in \cite[Lemma 4.7]{GL2}.

Theorems~\ref{thm:value-1} and~\ref{thm:m-final-pk} below are the main results of this section. Their proof relies on an involved inductive argument. The following lemma provides us with the base case of this induction. 

\begin{lemma}\label{lem:small-k}
	Let $p\ge 5$ be a prime.
	\begin{itemize}\itemspace{5pt}
		\item[(a)] We have that $\Omega(\phi_0)= \puncp{p}$. Also, $\Omega(\phi_i)=\cB_p(p-1)$ and $[(p-1,1)\down_{P_p}, \phi_i]=1$ for all $i\in [1,p-1]$. 
		
		\item[(b)] Write $T=(\underline{a};b)$ to denote the tree $(a\mid\dotsc\mid a;b)\in\cF_2$, which is admissible if $a,b\in[0,p-1]$. 
		In this case let $\theta(\underline{a};b)$ denote the corresponding irreducible character $\theta(T)=\Psi_2([T])\in\Irr(P_{p^2})$.
	
		Then for $i,j\in[1, p-1]$, we have
		\[ \begin{array}{crl}
			\Omega(\theta(\underline{0};0)) &= & \puncp{p^2}, \\
			\Omega(\theta(\underline{0};i)) &= & \cB_{p^2}(p^2-1), \\
			\Omega(\theta(\underline{i};0)) &= & \puncb{p^2}{p^2-p}, \\
			\Omega(\theta(\underline{i};j)) &= & \cB_{p^2}(p^2-p-1) \sqcup \close{\{(p^2-p,\mu)\mid \mu\in\cB_p(p-1)\}}.
		\end{array} \]
		In particular, for all $i,j\in [1,p-1]$ and $\theta:=\theta((\underline{i};j))$ we have that $\eta(\theta)=p+1=\gamma_0(\theta)+\gamma_1(\theta)$ and $\Omega(\theta)\setminus \cB_{p^2}(p^2-\eta(\theta))$ contains no thin partitions. Moreover,  
		\[ \begin{array}{rclrl}
			[\lambda\down_{P_{p^2}}, \theta(\underline{0};i)] & = & 1 & \text{for all}\ \lambda\in & \close{\{(p^2-1,1)\}}, \\
			{}[\lambda\down_{P_{p^2}}, \theta(\underline{i};0)] & = & 1 & \text{for all}\ \lambda\in & \close{ \{ \tworow{p^2}{p^2-p}, \hook{p^2}{p^2-p} \} }, \\
			{}[\lambda\down_{P_{p^2}}, \theta(\underline{i};j)] & \geq & 2 & \text{for all}\ \lambda\in & \close{\{\tworow{p^2}{p^2-p-1},  \hook{p^2}{p^2-p-1}\}}.
		\end{array} \]
			
		\item[(c)] Suppose $i_1,\ldots, i_p\in [0,p-1]$ are mixed. Let $T$ be the admissible tree $(i_1\mid\dotsc\mid i_p;p)\in\cF_2$
		and let $\theta=\theta(T)$. Then
		\[ \Omega(\theta) = \cB_{p^2}(p^2-\eta(\theta)). \]
		Moreover, $[\lambda\down_P,\theta]=1$ for all $\lambda\in \close{ \{ \tworow{p^2}{p^2-\eta(T)}, \hook{p^2}{p^2-\eta(T)} \} }$.
	\end{itemize}
\end{lemma}

\begin{proof}
	\begin{itemize}\itemspace{5pt}
		\item[(a)] This follows from the Murnaghan--Nakayama rule (see \cite[Lemma 4.2]{GL2}).
		
		\item[(b)] Using Lemma~\ref{lem:conj-restr}, Theorem~\ref{thm:GL1}, Lemma~\ref{lem:GL2-4.3} and \cite[Lemma 4.7]{GL2}, we only need to check that $[\lambda\down_P,\theta]=1$ for $\lambda\in\{\tworow{p^2}{p^2-p}, \hook{p^2}{p^2-p} \}$ and $\theta=\theta(\underline{i};0)$. 
		
		Let $\lambda=\tworow{p^2}{p^2-p}$. First observe that $[\lambda\down_Y, (p-1,1)^{\times p}]=1$. This is an easy application of the Littlewood--Richardson rule. 
		Since $[\lambda\down_V, \cX\big((p-1,1);(p)\big)]=1$ by Theorem~\ref{thm:PW9.1} and $\cX((p-1,1);\nu)\down_Y=\chi^\nu(1)\cdot(p-1,1)^{\times p}$ for all $\nu\in\cP(p)$, then
		\[ [\lambda\down_V, \cX\big((p-1,1);\nu\big)]=\begin{cases}
			1 & \text{if }\nu=(p),\\
			0 & \text{if }\nu\in\cP(p)\setminus\{(p)\}.
		\end{cases}\]
		In particular, $\theta\mid\cX\big((p-1,1);(p)\big)\down_P$ because $\phi_i\mid(p-1,1)\down_{P_p}$ and $\phi_0\mid (p)\down_{P_p}$, so $[\lambda\down_P,\theta]\ge 1$.
		
		On the other hand, $\theta\down_B=\phi_i^{\times p}$ and so
		\[ [\lambda\down_P,\theta]\le [\lambda\down_B,\phi_i^{\times p}] = [(p-1,1)^{\times p}\down_B,\phi_i^{\times p}] = [(p-1,1)\down_{P_p},\phi_i]^p = 1,\]
		where the first equality follows from $\Omega(\phi_i)=\cB_p(p-1)$, Lemma~\ref{lem:LR-first-part}  and the fact that $\lambda_1=p^2-p=p(p-1)$.
		The case $\lambda=\hook{p^2}{p^2-p}$ is analogous, replacing the partition $(p)$ by $(1^p)$.
		
		\item[(c)] Using Corollary~\ref{cor:omega-mixed}, we know that $\Omega(\theta)=\Omega(\phi_{i_1})\star\Omega(\phi_{i_2})\star\cdots\star\Omega(\phi_{i_p})$. Let $J=\{s\in [1,p]\ |\ i_s\neq 0\}$. Then $\eta(\theta)=|J|$ and $\Omega(\phi_{i_s})=\cB_p(p-1)$ for all $s\in J$, while $\Omega(\phi_{i_s})=\puncp{p}$ whenever $s\notin J$ by (a). 
		Since $i_1,\ldots,i_p$ are mixed, we have that $J\neq\emptyset$ and therefore that $\eta(\theta)\neq 0$. Hence we can use Theorem~\ref{prop:smooth}(ii) to deduce that $\Omega(\theta)=\cB_{p^2}(p^2-\eta(\theta))$.
		
		Let us now fix $\lambda=\tworow{p^2}{p^2-\eta(\theta)}$. let $\theta_t:=\phi_{i_t}$ for $t\in[1,p]$ and write $\overline{\theta}:=\theta_1\times\cdots\times\theta_p$.
		Since $\Irr(P_{p^2}\mid\overline{\theta}) = \{\theta\}$ and $[\theta\down_B, \overline{\theta}]=1$, we have that $[\lambda\down_P,\theta] = [\lambda\down_B, \overline{\theta}]$. 
		Moreover, 
		\[ [\lambda\down_B, \overline{\theta}] = \sum_{\mu^1,\ldots,\mu^p\in\cP(p^{k-1})} \LR(\lambda;\mu^1,\ldots,\mu^p)\cdot \prod_{t=1}^p [\mu^t\down_{P_p}, \theta_t]. \]
		By Lemma~\ref{lem:LR-first-part} we have that $\LR(\lambda;\mu^1,\ldots,\mu^p)>0$ only if $\mu^t$ is a two-row partition for every $t\in [1,p]$, since $\lambda$ is a two-row partition. At the same time, $[\mu^t\down_{P_p}, \theta_t]>0$ implies that $\mu^t\in \Omega(\theta_t)$. From (a), we deduce that $\mu^t\in\cB_{p}(p-\eta(\theta_t))$. 
		From $\LR(\lambda;\mu^1,\ldots,\mu^p)>0$ and Lemma~\ref{lem:LR-first-part}, we also obtain that 
		\[ p^2-\eta(\theta)=\lambda_1\le \sum_{t=1}^p (\mu^t)_1\leq \sum_{i=1}^p (p-\eta(\theta_t))=p^2-\eta(\theta), \]
		where the last equality follows because $\eta(\theta)=\eta(\theta_1)+\cdots+\eta(\theta_p)$. Thus 	
		\[ [\lambda\down_B,\overline{\theta}] = \LR(\lambda; \tworow{p}{p-\eta(\theta_1)},\ldots,\tworow{p}{p-\eta(\theta_p)}) \cdot \prod_{t=1}^p [\tworow{p}{p-\eta(\theta_t)}\down_{P_p}, \theta_t] = 1, \]
		where $[\tworow{p}{p-\eta(\theta_t)}\down_{P_p}, \theta_t] = 1$ follows again from (a).
		An analogous argument allows us to treat the case where $\lambda=\hook{p^2}{p^2-\eta(\theta)}$. By Lemma~\ref{lem:conj-restr}, the same holds for all partitions $\lambda\in \close{ \{ \tworow{p^2}{p^2-\eta(T)}, \hook{p^2}{p^2-\eta(T)} \} }$.
	\end{itemize}
\end{proof}

From Lemma~\ref{lem:small-k} we immediately obtain the following. We recall the definition of $M(\theta)$ from \eqref{eq:m-and-M} in the introduction.

\begin{corollary}\label{cor:M-p-p2}
	Let $k\in\{1,2\}$ and let $\theta\in\Irr(P_{p^k})$. Then $M(\theta)=p^k-\gamma_0(\theta)$.
\end{corollary}

We are now ready to describe $\Omega(\theta)$ for all those irreducible characters $\theta$ whose corresponding trees, i.e.~any $T\in\cT(\theta)$, have very few people without 1-descendants. In this case, the structure of the set $\Omega(\theta)$ is particularly simple.

\begin{lemma}\label{lem:small-gamma_0}
	Let $k\in\N$ and $\theta\in\Irr(P_{p^k})$. If $0<\gamma_0(\theta)<p$, then $\Omega(\theta)=\cB_{p^k}(p^k-\gamma_0(\theta))$.
\end{lemma}

\begin{proof}
	We proceed by induction on $k$, noting that if $k\le 2$ then the statement follows from Lemma~\ref{lem:small-k}. Now suppose $k\ge 3$. If $\theta=\cX(\varphi;\varepsilon)$ for some $\varphi\in\Irr(P_{p^{k-1}})$ and some $\varepsilon\in [0, p-1]$, then $\gamma_0(\theta)\in\{p\cdot\gamma_0(\varphi), p\cdot\gamma_0(\varphi)+1\}$. Since $0<\gamma_0(\theta)<p$, we must have $\gamma_0(\varphi)=0$ and $\varepsilon\ne 0$. That is, $\theta=\cX(\triv_{p^{k-1}}; \varepsilon)$, and so $\Omega(\theta)=\cB_{p^k}(p^k-1)$ from Lemma~\ref{lem:GL2-4.3}.
	
	Otherwise, $\theta=\theta_1\times\cdots\times\theta_p\up^P$ for some mixed $\theta_i\in\Irr(Q)$. Then $\gamma_0(\theta)=\sum_{i=1}^p\gamma_0(\theta_i)$. We have that $\Omega(\theta_i)=\puncp{p^{k-1}}$ if $\gamma_0(\theta_i)=0$ by Theorem~\ref{thm:GL1}, or $\Omega(\theta_i)=\cB_{p^{k-1}}(p^{k-1}-\gamma_0(\theta_i))$ by inductive hypothesis if $\gamma_0(\theta_i)\in[1, p-1]$. But $\Omega(\theta)=\Omega(\theta_1)\star\cdots\star\Omega(\theta_p)$ by Corollary~\ref{cor:omega-mixed}, and so $\Omega(\theta)=\cB_{p^k}(p^k-\sum_{i=1}^p \gamma_0(\theta_i))$ by Theorem~\ref{prop:smooth} as $\gamma_0(\theta_i)\ge 1$ for some $i\in [1, p]$.
\end{proof}

\begin{example}\label{ex:5,25}
	Let $p=5$. In Tables~\ref{tab:5} and~\ref{tab:25} we record all of the elements of $\theta\in\Irr(P_{p^k})$ for $k=1$ and $k=2$ respectively, up to the equivalence relation $\sim_{0,p}$ introduced in Definition~\ref{def:label-replace}. For each representative $\theta$, we give an admissible tree $T\in\cT(\theta)$, the tree statistics introduced in Definition~\ref{def:trees-stats}, and $\Omega(\theta)$. We remark that if $\theta\sim_{0,p}\zeta$, then $\eta(\theta)=\eta(\zeta)$, $\gamma_i(\theta)=\gamma_i(\zeta)$ for all $i\in\N_0$, and $\val(\theta)=\val(\zeta)$. Moreover, $\theta\up^{S_{p^k}}=\zeta\up^{S_{p^k}}$ and $\Omega(\theta)=\Omega(\zeta)$ by Theorem~\ref{thm:equiv-classes}.
	\hfill$\lozenge$
\end{example}

\begin{example}\label{ex:125}
	In Figure~\ref{fig:125}, we illustrate several examples of $\theta\in\Irr(P_{p^3})$ where $p=5$, listing various associated tree statistics, as well as information on $\Omega(\theta)$, $m(\theta)$ and $M(\theta)$.
	
	The fourth subfigure illustrates a tree $T\in\cT(\theta)$ for the linear character $\cX(a;b;c):=\cX( \cX(\phi_a;\phi_b);\phi_c ) \in\Irr(P_{125})$. Recalling the equivalence relation $\sim_{0,p}$ from Definition~\ref{def:label-replace} and Theorem~\ref{thm:equiv-classes}, there are only 8 cases to consider when calculating the Sylow branching coefficients which describe $\theta\up^{S_{125}}$: these are tabulated in Table~\ref{tab:125}. 
	
	Results on $\Omega(\theta)$ follow from Theorems~\ref{thm:GL1}, \ref{thm:value-1} and~\ref{thm:value-2}, and \cite[Example 5.9]{GL2}. Results on $M(\theta)$ and $m(\theta)$ follow from Theorem~\ref{thm:M}, Corollary~\ref{cor:m-pk} and Theorem~\ref{thm:m-val-2}.
	In particular, for $\theta=\cX(1;1;\varepsilon)$ where $\varepsilon\in\{0,1\}$, we have that 
	$\cB_{125}(95)\subseteq\Omega(\theta)\subseteq\cB_{125}(100)$
	and $\Omega(\theta)\setminus\cB_{125}(95)$ contains no thin partitions. Moreover, 
	\[ \Omega(\theta)\cap\close{\{ (100,\mu)\mid \mu\in\cP(25)\}} = \close{\{ (100,\mu)\mid \mu\in\Omega(\cX(1;\varepsilon))\}}, \]
	where $\Omega(\cX(1;\varepsilon))$ is recorded in Lemma~\ref{lem:small-k}.
	\newqed
\end{example}

\begin{theorem}\label{thm:M}
Let $k\in\N$ and let $\theta\in\Irr(P_{p^k})$. Then $M(\theta)=p^k-\gamma_0(\theta)$. 
\end{theorem}

\begin{proof}
	We proceed by induction on $k\in\N$. If $k\in\{1, 2\}$ then the statement holds by Corollary~\ref{cor:M-p-p2}. 
	
	Let us now assume that $k\geq 3$. We will use the notation introduced in Notation~\ref{not:subgps}. 
	
	Let $\overline{\theta}=\theta_1\times\theta_2\times\cdots\times\theta_p$ be an irreducible constituent of $\theta\down_B$. Here, $\theta_i\in\Irr(Q)$ for every $i\in [1,p]$. The inductive hypothesis implies that $M(\theta_i)=p^{k-1}-\gamma_0(\theta_i)$ for all $i\in [1,p]$. Moreover, $\gamma_0(\theta)=\gamma_0(\theta_1)+\gamma_0(\theta_2)+\cdots +\gamma_0(\theta_p)$. This shows that it will suffice to prove that $M(\theta)=\Sigma$, where $\Sigma:=M(\theta_1)+M(\theta_2)+\cdots+M(\theta_p)$. 
	
	We first focus on showing that $M(\theta)\le\Sigma$. Let $\lambda\in\Omega(\theta)$ and assume, without loss of generality, that $\lambda_1\ge\ell(\lambda)$. Since $\overline{\theta}\mid\theta\down_B$, we have that $\overline{\theta}$ is an irreducible constituent of $\lambda\down_B$. Since 
	\[ \lambda\down_B=(\lambda\down_Y)\down_B=\sum_{\mu^1,\ldots, \mu^p\in\cP(p^{k-1})} \LR(\lambda;\mu^1,\ldots, \mu^p)\cdot \big(\mu^1\times\cdots\times\mu^p\big)\down_B, \]
	we deduce that there exist $\mu^1,\ldots,\mu^p\in\cP(p^{k-1})$ such that $\mu^i\in\Omega(\theta_i)$ for every $i\in [1,p]$ and such that $\LR(\lambda;\mu^1,\ldots, \mu^p)\neq 0$. Using Lemma~\ref{lem:LR-first-part} and recalling that $M(\theta_i)\geq (\mu^i)_1$ for every $i\in [1,p]$, we deduce that $\Sigma\geq (\mu^1)_1+ (\mu^2)_1+\cdots + (\mu^p)_1\geq \lambda_1$. This implies that $\lambda\in\cB_{p^k}(\Sigma)$. Therefore we have shown $\Omega(\theta)\subseteq \cB_{p^k}(\Sigma)$, giving $M(\theta)\le\Sigma$ as desired. 
	
	In order to show that $\Sigma\leq M(\theta)$, we exhibit a partition $\lambda\in\Omega(\theta)$ with $\lambda_1=\Sigma$.
	In order to construct such a $\lambda$ we distinguish two cases, depending on the structure of $\theta$. 
	
	Suppose first that $\theta_1, \theta_2,\ldots, \theta_p$ are mixed. For each $i\in [1,p]$ choose $\lambda^i\in\Omega(\theta_i)$ such that $(\lambda^i)_1=M(\theta_i)$. Now consider $\lambda=\lambda^1+\lambda^2+\cdots+\lambda^p$. 
	Then clearly $\lambda_1=(\lambda^1)_1+\cdots+(\lambda^p)_1=\Sigma$, and $\LR(\lambda;\lambda^1,\ldots, \lambda^p)=1$. It follows from Lemma~\ref{lem:mixed} that $\lambda\in\Omega(\theta)$, and so $\Sigma\le M(\theta)$.
	
	Let us now consider the case where $\theta_1=\theta_2=\cdots=\theta_p$. In this case, $\theta=\cX(\theta_1;\phi_{\varepsilon})$ for some $\varepsilon\in [0,p-1]$ and $\Sigma=p\cdot M(\theta_1)$.
	
	If $\theta_1=\triv_Q$, then the assertion of the theorem holds by Theorem~\ref{thm:GL1} if $\varepsilon=0$ and by Lemma~\ref{lem:GL2-4.3} if $\varepsilon\ne 0$.
	
	Otherwise, $\theta_1\ne\triv_Q$ and therefore $M(\theta_1)<p^{k-1}$. 
	Choose $\mu\in\Omega(\theta_1)$ with $\mu_1=M(\theta_1)$. Define 
	\[ \lambda=(p\mu_1,\ldots,p\mu_{l(\mu)-1}, p(\mu_{l(\mu)}-1)+\nu_1,\nu_2,\ldots,\nu_{\ell(\nu)})\]
	where $\nu\in\Omega(\phi_\varepsilon)$.
	Then Lemmas~\ref{lem:SL2.18} and~\ref{lem:SL2.19} imply that $[\cX(\mu;\nu)\down_P, \theta]\neq 0$ as $\mu\in\Omega(\theta_1)$ and $\nu\in\Omega(\phi_\varepsilon)$. Moreover, $[\cX(\mu;\nu), \lambda\down_V]=1$ by Theorem~\ref{thm:PW9.1}. Hence $\lambda\in\Omega(\theta)$, and $\lambda_1=p\mu_1=\Sigma$. We conclude that $\Sigma\leq M(\theta)$.
\end{proof}

\begin{table}[h!]
	\[ \begin{array}{|cccc|cccccc|l|}
		\hline
		\theta && T\in\cT(\theta) & \theta(1) & \eta(\theta) & \gamma_0(\theta) & \gamma_1(\theta) & \val(\theta) & n-\gamma_0(\theta) & n-\gamma_0(\theta)-\gamma_1(\theta) & \Omega(\theta) \\
		\hline
		\hline
		\phi_0=\triv_{5} && \overset{0}{\circ} & 1 & 0 & 0 & 0 & 0 & 5 & 5 & \puncp{5} \\[3pt]
		\phi_\varepsilon, & {\scriptstyle\varepsilon\in[1,4]} & \overset{\varepsilon}{\bullet} & 1 & 1 & 1 & 0 & 1 & 4 & 4 & \cB_5(4)\\[3pt]
		\hline
	\end{array} \]
	\caption{$\Irr(P_p)$ for $p=5$, their corresponding admissible trees and statistics.}\label{tab:5}
\end{table}

\begin{table}[!h]
	\begin{small}
		\[ \begin{array}{|clc|cc|l|}
			\hline
			\theta & T\in\cT(\theta) & \theta(1) & \eta, \gamma_0, \gamma_1, \val & {\scriptscriptstyle n-\gamma_0, n-\gamma_0-\gamma_1} & \Omega(\theta) \\
			\hline
			\hline
			& \multirow[c]{2}{2cm}{\begin{tikzpicture}[scale=0.65, every node/.style={scale=0.7}]
					\draw (0,0) node(R) {$\overset{0}{\circ}$};
					\draw (-1,-1) node(0) {$\overset{0}{\circ}$};
					\draw (-0.5,-1) node(1) {$\overset{0}{\circ}$};
					\draw (0,-1) node(2) {$\overset{0}{\circ}$};
					\draw (0.5,-1) node(3) {$\overset{0}{\circ}$};
					\draw (1,-1) node(4) {$\overset{0}{\circ}$};
					\draw (R) -- (0);
					\draw (R) -- (1);
					\draw (R) -- (2);
					\draw (R) -- (3);
					\draw (R) -- (4);
			\end{tikzpicture}}
			& & & & \\
			\cX(0;0)=\triv_{25} & & 1 & 0, 0, 0, 0 & 25, 25 & \puncp{25} \\[8pt]
			& \multirow[c]{2}{2cm}{\begin{tikzpicture}[scale=0.65, every node/.style={scale=0.7}]
					\draw (0,0) node(R) {$\overset{\varepsilon}{\bullet}$};
					\draw (-1,-1) node(0) {$\overset{0}{\circ}$};
					\draw (-0.5,-1) node(1) {$\overset{0}{\circ}$};
					\draw (0,-1) node(2) {$\overset{0}{\circ}$};
					\draw (0.5,-1) node(3) {$\overset{0}{\circ}$};
					\draw (1,-1) node(4) {$\overset{0}{\circ}$};
					\draw (R) -- (0);
					\draw (R) -- (1);
					\draw (R) -- (2);
					\draw (R) -- (3);
					\draw (R) -- (4);
			\end{tikzpicture}}
			& & & & \\
			\cX(0;\varepsilon) & & 1 & 1, 1, 0, 1 & 24, 24 & \cB_{25}(24) \\[8pt]
			& \multirow[c]{2}{2cm}{\begin{tikzpicture}[scale=0.65, every node/.style={scale=0.7}]
					\draw (0,0) node(R) {$\overset{0}{\circ}$};
					\draw (-1,-1) node(0) {$\overset{\varepsilon}{\bullet}$};
					\draw (-0.5,-1) node(1) {$\overset{\varepsilon}{\bullet}$};
					\draw (0,-1) node(2) {$\overset{\varepsilon}{\bullet}$};
					\draw (0.5,-1) node(3) {$\overset{\varepsilon}{\bullet}$};
					\draw (1,-1) node(4) {$\overset{\varepsilon}{\bullet}$};
					\draw (R) -- (0);
					\draw (R) -- (1);
					\draw (R) -- (2);
					\draw (R) -- (3);
					\draw (R) -- (4);
			\end{tikzpicture}}
			& & & & \\
			\cX(\varepsilon;0) & & 1 & 5, 5, 0, 1 & 20, 20 & \puncb{25}{20} \\[8pt]
			& \multirow[c]{2}{2cm}{\begin{tikzpicture}[scale=0.65, every node/.style={scale=0.7}]
					\draw (0,0) node(R) {$\overset{\sigma}{\bullet}$};
					\draw (-1,-1) node(0) {$\overset{\varepsilon}{\bullet}$};
					\draw (-0.5,-1) node(1) {$\overset{\varepsilon}{\bullet}$};
					\draw (0,-1) node(2) {$\overset{\varepsilon}{\bullet}$};
					\draw (0.5,-1) node(3) {$\overset{\varepsilon}{\bullet}$};
					\draw (1,-1) node(4) {$\overset{\varepsilon}{\bullet}$};
					\draw (R) -- (0);
					\draw (R) -- (1);
					\draw (R) -- (2);
					\draw (R) -- (3);
					\draw (R) -- (4);
			\end{tikzpicture}}
			& & & & \\
			\cX(\varepsilon;\sigma) & & 1 & 6, 5, 1, 2 & 20, 19 & {\scriptstyle\cB_{25}(20)\sqcup \close{\{ (20,\mu)\mid\mu\in\cB_5(4) \}} } \\[8pt]
			& \multirow[c]{2}{2cm}{\begin{tikzpicture}[scale=0.65, every node/.style={scale=0.7}]
					\draw (0,0) node(R) {$\overset{5}{\circ}$};
					\draw (-1,-1) node(0) {$\overset{0}{\circ}$};
					\draw (-0.5,-1) node(1) {$\overset{0}{\circ}$};
					\draw (0,-1) node(2) {$\overset{0}{\circ}$};
					\draw (0.5,-1) node(3) {$\overset{0}{\circ}$};
					\draw (1,-1) node(4) {$\overset{\varepsilon}{\bullet}$};
					\draw (R) -- (0);
					\draw (R) -- (1);
					\draw (R) -- (2);
					\draw (R) -- (3);
					\draw (R) -- (4);
			\end{tikzpicture}}
			& & & & \\
			(\phi_0^{\times 4}\times\phi_\varepsilon)\up^{P_{25}} & & 5 & 1, 1, 0, 1 & 24, 24 & \cB_{25}(24) \\[8pt]
			& \multirow[c]{2}{2cm}{\begin{tikzpicture}[scale=0.65, every node/.style={scale=0.7}]
					\draw (0,0) node(R) {$\overset{5}{\circ}$};
					\draw (-1,-1) node(0) {$\overset{0}{\circ}$};
					\draw (-0.5,-1) node(1) {$\overset{0}{\circ}$};
					\draw (0,-1) node(2) {$\overset{0}{\circ}$};
					\draw (0.5,-1) node(3) {$\overset{\varepsilon}{\bullet}$};
					\draw (1,-1) node(4) {$\overset{\sigma}{\bullet}$};
					\draw (R) -- (0);
					\draw (R) -- (1);
					\draw (R) -- (2);
					\draw (R) -- (3);
					\draw (R) -- (4);
			\end{tikzpicture}}
			& & & & \\
			(\phi_0^{\times 3}\times\phi_\varepsilon \times\phi_\sigma)\up^{P_{25}} & & 5 & 2, 2, 0, 1 & 23, 23 & \cB_{25}(23) \\[8pt]
			& \multirow[c]{2}{2cm}{\begin{tikzpicture}[scale=0.65, every node/.style={scale=0.7}]
					\draw (0,0) node(R) {$\overset{5}{\circ}$};
					\draw (-1,-1) node(0) {$\overset{0}{\circ}$};
					\draw (-0.5,-1) node(1) {$\overset{0}{\circ}$};
					\draw (0,-1) node(2) {$\overset{\varepsilon}{\bullet}$};
					\draw (0.5,-1) node(3) {$\overset{0}{\circ}$};
					\draw (1,-1) node(4) {$\overset{\sigma}{\bullet}$};
					\draw (R) -- (0);
					\draw (R) -- (1);
					\draw (R) -- (2);
					\draw (R) -- (3);
					\draw (R) -- (4);
			\end{tikzpicture}}
			& & & & \\
			(\phi_0^{\times 2}\times\phi_\varepsilon \times\phi_0\times\phi_\sigma)\up^{P_{25}} & & 5 & 2, 2, 0, 1 & 23, 23 & \cB_{25}(23) \\[8pt]
			& \multirow[c]{2}{2cm}{\begin{tikzpicture}[scale=0.65, every node/.style={scale=0.7}]
					\draw (0,0) node(R) {$\overset{5}{\circ}$};
					\draw (-1,-1) node(0) {$\overset{0}{\circ}$};
					\draw (-0.5,-1) node(1) {$\overset{0}{\circ}$};
					\draw (0,-1) node(2) {$\overset{\varepsilon}{\bullet}$};
					\draw (0.5,-1) node(3) {$\overset{\sigma}{\bullet}$};
					\draw (1,-1) node(4) {$\overset{\tau}{\bullet}$};
					\draw (R) -- (0);
					\draw (R) -- (1);
					\draw (R) -- (2);
					\draw (R) -- (3);
					\draw (R) -- (4);
			\end{tikzpicture}}
			& & & & \\
			(\phi_0^{\times 2}\times\phi_\varepsilon \times\phi_\sigma \times\phi_\tau)\up^{P_{25}} & & 5 & 3, 3, 0, 1 & 22, 22 & \cB_{25}(22) \\[8pt]
			& \multirow[c]{2}{2cm}{\begin{tikzpicture}[scale=0.65, every node/.style={scale=0.7}]
					\draw (0,0) node(R) {$\overset{5}{\circ}$};
					\draw (-1,-1) node(0) {$\overset{0}{\circ}$};
					\draw (-0.5,-1) node(1) {$\overset{\varepsilon}{\bullet}$};
					\draw (0,-1) node(2) {$\overset{0}{\circ}$};
					\draw (0.5,-1) node(3) {$\overset{\sigma}{\bullet}$};
					\draw (1,-1) node(4) {$\overset{\tau}{\bullet}$};
					\draw (R) -- (0);
					\draw (R) -- (1);
					\draw (R) -- (2);
					\draw (R) -- (3);
					\draw (R) -- (4);
			\end{tikzpicture}}
			& & & & \\
			{\scriptscriptstyle (\phi_0 \times \phi_\varepsilon \times \phi_0 \times \phi_\sigma \times\phi_\tau)\up^{P_{25}}} & & 5 & 3, 3, 0, 1 & 22, 22 & \cB_{25}(22) \\[8pt]
			& \multirow[c]{2}{2cm}{\begin{tikzpicture}[scale=0.65, every node/.style={scale=0.7}]
					\draw (0,0) node(R) {$\overset{5}{\circ}$};
					\draw (-1,-1) node(0) {$\overset{0}{\circ}$};
					\draw (-0.5,-1) node(1) {$\overset{\varepsilon}{\bullet}$};
					\draw (0,-1) node(2) {$\overset{\sigma}{\bullet}$};
					\draw (0.5,-1) node(3) {$\overset{\tau}{\bullet}$};
					\draw (1,-1) node(4) {$\overset{\nu}{\bullet}$};
					\draw (R) -- (0);
					\draw (R) -- (1);
					\draw (R) -- (2);
					\draw (R) -- (3);
					\draw (R) -- (4);
			\end{tikzpicture}}
			& & & & \\
			{\scriptscriptstyle (\phi_0 \times \phi_\varepsilon \times \phi_\sigma \times\phi_\tau \times\phi_\nu )\up^{P_{25}}} & & 5 & 4, 4, 0, 1 & 21, 21 & \cB_{25}(21) \\[8pt]
			& \multirow[c]{2}{2cm}{\begin{tikzpicture}[scale=0.65, every node/.style={scale=0.7}]
					\draw (0,0) node(R) {$\overset{5}{\circ}$};
					\draw (-1,-1) node(0) {$\overset{\varepsilon}{\bullet}$};
					\draw (-0.5,-1) node(1) {$\overset{\sigma}{\bullet}$};
					\draw (0,-1) node(2) {$\overset{\tau}{\bullet}$};
					\draw (0.5,-1) node(3) {$\overset{\nu}{\bullet}$};
					\draw (1,-1) node(4) {$\overset{\omega}{\bullet}$};
					\draw (R) -- (0);
					\draw (R) -- (1);
					\draw (R) -- (2);
					\draw (R) -- (3);
					\draw (R) -- (4);
			\end{tikzpicture}}
			& & & & \\
			{\scriptscriptstyle (\phi_\varepsilon \times \phi_\sigma \times\phi_\tau \times\phi_\nu \times\phi_\omega )\up^{P_{25}}} & & 5 & 5, 5, 0, 1 & 20, 20 & \cB_{25}(20) \\[8pt]
			\hline
		\end{array} \]
	\end{small}
	\caption{$\Irr(P_{p^2})$ for $p=5$, their corresponding admissible trees and statistics.\\Here, $\varepsilon,\sigma,\tau,\nu,\omega$ are arbitrary elements of $[1,4]$, and, only in the last row, they must be mixed.}\label{tab:25}
\end{table}

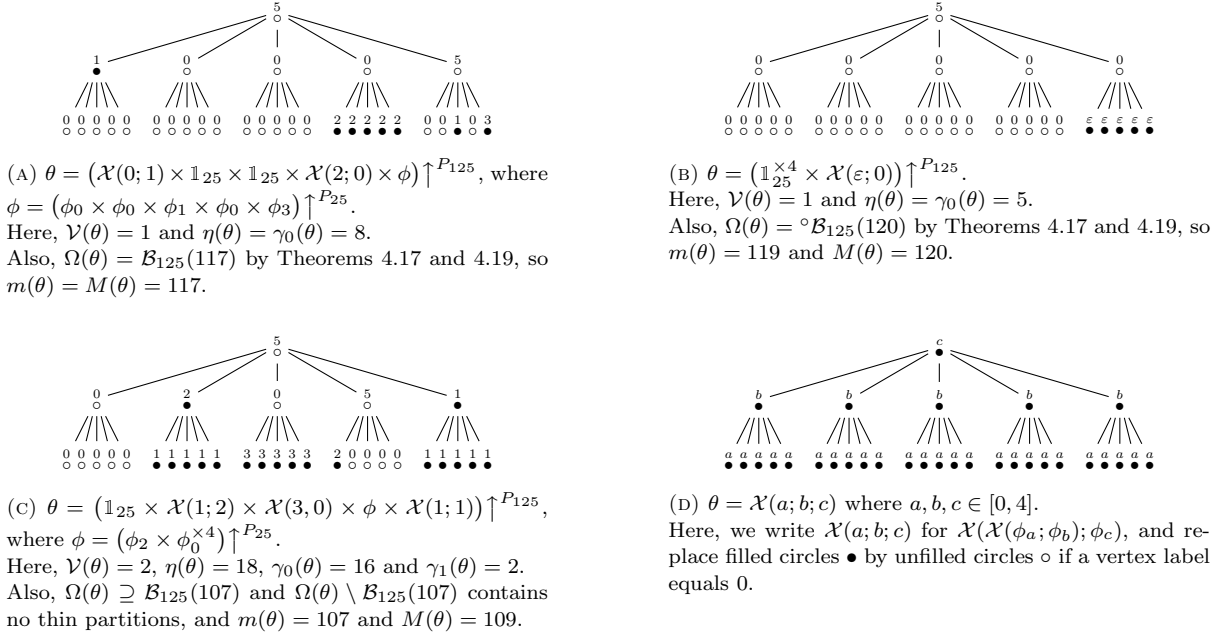
\begin{figure}[t!]
	\centering
	\begin{subfigure}[t]{0.45\textwidth}
		\centering
		\begin{tikzpicture}[scale=1.0, every node/.style={scale=0.7}]
			\draw (0,0.1) node(R) {$\overset{5}{\circ}$};
			\draw (-2.4,-0.6) node(0) {$\overset{1}{\bullet}$};
			\draw (-1.2,-0.6) node(1) {$\overset{0}{\circ}$};
			\draw (0,-0.6) node(2) {$\overset{0}{\circ}$};
			\draw (1.2,-0.6) node(3) {$\overset{0}{\circ}$};
			\draw (2.4,-0.6) node(4) {$\overset{5}{\circ}$};
			\draw (R) -- (0);
			\draw (R) -- (1);
			\draw (R) -- (2);
			\draw (R) -- (3);
			\draw (R) -- (4);
			\draw (-2.8,-1.4) node(00) {$\overset{0}{\circ}$};
			\draw (-2.6,-1.4) node(01) {$\overset{0}{\circ}$};
			\draw (-2.4,-1.4) node(02) {$\overset{0}{\circ}$};
			\draw (-2.2,-1.4) node(03) {$\overset{0}{\circ}$};
			\draw (-2.0,-1.4) node(04) {$\overset{0}{\circ}$};
			\draw (0) -- (00);
			\draw (0) -- (01);
			\draw (0) -- (02);
			\draw (0) -- (03);
			\draw (0) -- (04);
			\draw (-1.6,-1.4) node(10) {$\overset{0}{\circ}$};
			\draw (-1.4,-1.4) node(11) {$\overset{0}{\circ}$};
			\draw (-1.2,-1.4) node(12) {$\overset{0}{\circ}$};
			\draw (-1.0,-1.4) node(13) {$\overset{0}{\circ}$};
			\draw (-0.8,-1.4) node(14) {$\overset{0}{\circ}$};
			\draw (1) -- (10);
			\draw (1) -- (11);
			\draw (1) -- (12);
			\draw (1) -- (13);
			\draw (1) -- (14);
			\draw (-0.4,-1.4) node(20) {$\overset{0}{\circ}$};
			\draw (-0.2,-1.4) node(21) {$\overset{0}{\circ}$};
			\draw (0,-1.4) node(22) {$\overset{0}{\circ}$};
			\draw (0.2,-1.4) node(23) {$\overset{0}{\circ}$};
			\draw (0.4,-1.4) node(24) {$\overset{0}{\circ}$};
			\draw (2) -- (20);
			\draw (2) -- (21);
			\draw (2) -- (22);
			\draw (2) -- (23);
			\draw (2) -- (24);
			\draw (0.8,-1.4) node(30) {$\overset{2}{\bullet}$};
			\draw (1.0,-1.4) node(31) {$\overset{2}{\bullet}$};
			\draw (1.2,-1.4) node(32) {$\overset{2}{\bullet}$};
			\draw (1.4,-1.4) node(33) {$\overset{2}{\bullet}$};
			\draw (1.6,-1.4) node(34) {$\overset{2}{\bullet}$};
			\draw (3) -- (30);
			\draw (3) -- (31);
			\draw (3) -- (32);
			\draw (3) -- (33);
			\draw (3) -- (34);
			\draw (2.0,-1.4) node(40) {$\overset{0}{\circ}$};
			\draw (2.2,-1.4) node(41) {$\overset{0}{\circ}$};
			\draw (2.4,-1.4) node(42) {$\overset{1}{\bullet}$};
			\draw (2.6,-1.4) node(43) {$\overset{0}{\circ}$};
			\draw (2.8,-1.4) node(44) {$\overset{3}{\bullet}$};
			\draw (4) -- (40);
			\draw (4) -- (41);
			\draw (4) -- (42);
			\draw (4) -- (43);
			\draw (4) -- (44);
		\end{tikzpicture}
		\caption{$\theta = \big( \cX(0;1)\times \triv_{25}\times \triv_{25}\times \cX(2;0)\times\phi \big)\up^{P_{125}}$, where $\phi=\big( \phi_0\times\phi_0\times\phi_1\times\phi_0\times\phi_3 \big)\up^{P_{25}}$.\\
			Here, $\val(\theta)=1$ and $\eta(\theta)=\gamma_0(\theta)=8$.\\
			Also, $\Omega(\theta)=\cB_{125}(117)$ by Theorems~\ref{thm:value-1} and~\ref{thm:punctures}, so $m(\theta)=M(\theta)=117$.}
		\label{fig:125-1}
	\end{subfigure}
	\hfill
	\begin{subfigure}[t]{0.45\textwidth}
		\centering
		\begin{tikzpicture}[scale=1.0, every node/.style={scale=0.7}]
			\draw (0,0.1) node(R) {$\overset{5}{\circ}$};
			\draw (-2.4,-0.6) node(0) {$\overset{0}{\circ}$};
			\draw (-1.2,-0.6) node(1) {$\overset{0}{\circ}$};
			\draw (0,-0.6) node(2) {$\overset{0}{\circ}$};
			\draw (1.2,-0.6) node(3) {$\overset{0}{\circ}$};
			\draw (2.4,-0.6) node(4) {$\overset{0}{\circ}$};
			\draw (R) -- (0);
			\draw (R) -- (1);
			\draw (R) -- (2);
			\draw (R) -- (3);
			\draw (R) -- (4);
			\draw (-2.8,-1.4) node(00) {$\overset{0}{\circ}$};
			\draw (-2.6,-1.4) node(01) {$\overset{0}{\circ}$};
			\draw (-2.4,-1.4) node(02) {$\overset{0}{\circ}$};
			\draw (-2.2,-1.4) node(03) {$\overset{0}{\circ}$};
			\draw (-2.0,-1.4) node(04) {$\overset{0}{\circ}$};
			\draw (0) -- (00);
			\draw (0) -- (01);
			\draw (0) -- (02);
			\draw (0) -- (03);
			\draw (0) -- (04);
			\draw (-1.6,-1.4) node(10) {$\overset{0}{\circ}$};
			\draw (-1.4,-1.4) node(11) {$\overset{0}{\circ}$};
			\draw (-1.2,-1.4) node(12) {$\overset{0}{\circ}$};
			\draw (-1.0,-1.4) node(13) {$\overset{0}{\circ}$};
			\draw (-0.8,-1.4) node(14) {$\overset{0}{\circ}$};
			\draw (1) -- (10);
			\draw (1) -- (11);
			\draw (1) -- (12);
			\draw (1) -- (13);
			\draw (1) -- (14);
			\draw (-0.4,-1.4) node(20) {$\overset{0}{\circ}$};
			\draw (-0.2,-1.4) node(21) {$\overset{0}{\circ}$};
			\draw (0,-1.4) node(22) {$\overset{0}{\circ}$};
			\draw (0.2,-1.4) node(23) {$\overset{0}{\circ}$};
			\draw (0.4,-1.4) node(24) {$\overset{0}{\circ}$};
			\draw (2) -- (20);
			\draw (2) -- (21);
			\draw (2) -- (22);
			\draw (2) -- (23);
			\draw (2) -- (24);
			\draw (0.8,-1.4) node(30) {$\overset{0}{\circ}$};
			\draw (1.0,-1.4) node(31) {$\overset{0}{\circ}$};
			\draw (1.2,-1.4) node(32) {$\overset{0}{\circ}$};
			\draw (1.4,-1.4) node(33) {$\overset{0}{\circ}$};
			\draw (1.6,-1.4) node(34) {$\overset{0}{\circ}$};
			\draw (3) -- (30);
			\draw (3) -- (31);
			\draw (3) -- (32);
			\draw (3) -- (33);
			\draw (3) -- (34);
			\draw (2.0,-1.4) node(40) {$\overset{\varepsilon}{\bullet}$};
			\draw (2.2,-1.4) node(41) {$\overset{\varepsilon}{\bullet}$};
			\draw (2.4,-1.4) node(42) {$\overset{\varepsilon}{\bullet}$};
			\draw (2.6,-1.4) node(43) {$\overset{\varepsilon}{\bullet}$};
			\draw (2.8,-1.4) node(44) {$\overset{\varepsilon}{\bullet}$};
			\draw (4) -- (40);
			\draw (4) -- (41);
			\draw (4) -- (42);
			\draw (4) -- (43);
			\draw (4) -- (44);
		\end{tikzpicture}
		\caption{$\theta = \big( \triv_{25}^{\times 4} \times \cX(\varepsilon;0) \big)\up^{P_{125}}$.\\
			Here, $\val(\theta)=1$ and $\eta(\theta)=\gamma_0(\theta)=5$.\\
			Also, $\Omega(\theta)=\puncb{125}{120}$ by Theorems~\ref{thm:value-1} and~\ref{thm:punctures}, so $m(\theta)=119$ and $M(\theta)=120$.}
		\label{fig:125-2}
	\end{subfigure}
	
	\vspace{10pt}
	
	\begin{subfigure}[t]{0.45\textwidth}
		\centering
		\begin{tikzpicture}[scale=1.0, every node/.style={scale=0.7}]
			\draw (0,0.1) node(R) {$\overset{5}{\circ}$};
			\draw (-2.4,-0.6) node(0) {$\overset{0}{\circ}$};
			\draw (-1.2,-0.6) node(1) {$\overset{2}{\bullet}$};
			\draw (0,-0.6) node(2) {$\overset{0}{\circ}$};
			\draw (1.2,-0.6) node(3) {$\overset{5}{\circ}$};
			\draw (2.4,-0.6) node(4) {$\overset{1}{\bullet}$};
			\draw (R) -- (0);
			\draw (R) -- (1);
			\draw (R) -- (2);
			\draw (R) -- (3);
			\draw (R) -- (4);
			\draw (-2.8,-1.4) node(00) {$\overset{0}{\circ}$};
			\draw (-2.6,-1.4) node(01) {$\overset{0}{\circ}$};
			\draw (-2.4,-1.4) node(02) {$\overset{0}{\circ}$};
			\draw (-2.2,-1.4) node(03) {$\overset{0}{\circ}$};
			\draw (-2.0,-1.4) node(04) {$\overset{0}{\circ}$};
			\draw (0) -- (00);
			\draw (0) -- (01);
			\draw (0) -- (02);
			\draw (0) -- (03);
			\draw (0) -- (04);
			\draw (-1.6,-1.4) node(10) {$\overset{1}{\bullet}$};
			\draw (-1.4,-1.4) node(11) {$\overset{1}{\bullet}$};
			\draw (-1.2,-1.4) node(12) {$\overset{1}{\bullet}$};
			\draw (-1.0,-1.4) node(13) {$\overset{1}{\bullet}$};
			\draw (-0.8,-1.4) node(14) {$\overset{1}{\bullet}$};
			\draw (1) -- (10);
			\draw (1) -- (11);
			\draw (1) -- (12);
			\draw (1) -- (13);
			\draw (1) -- (14);
			\draw (-0.4,-1.4) node(20) {$\overset{3}{\bullet}$};
			\draw (-0.2,-1.4) node(21) {$\overset{3}{\bullet}$};
			\draw (0,-1.4) node(22) {$\overset{3}{\bullet}$};
			\draw (0.2,-1.4) node(23) {$\overset{3}{\bullet}$};
			\draw (0.4,-1.4) node(24) {$\overset{3}{\bullet}$};
			\draw (2) -- (20);
			\draw (2) -- (21);
			\draw (2) -- (22);
			\draw (2) -- (23);
			\draw (2) -- (24);
			\draw (0.8,-1.4) node(30) {$\overset{2}{\bullet}$};
			\draw (1.0,-1.4) node(31) {$\overset{0}{\circ}$};
			\draw (1.2,-1.4) node(32) {$\overset{0}{\circ}$};
			\draw (1.4,-1.4) node(33) {$\overset{0}{\circ}$};
			\draw (1.6,-1.4) node(34) {$\overset{0}{\circ}$};
			\draw (3) -- (30);
			\draw (3) -- (31);
			\draw (3) -- (32);
			\draw (3) -- (33);
			\draw (3) -- (34);
			\draw (2.0,-1.4) node(40) {$\overset{1}{\bullet}$};
			\draw (2.2,-1.4) node(41) {$\overset{1}{\bullet}$};
			\draw (2.4,-1.4) node(42) {$\overset{1}{\bullet}$};
			\draw (2.6,-1.4) node(43) {$\overset{1}{\bullet}$};
			\draw (2.8,-1.4) node(44) {$\overset{1}{\bullet}$};
			\draw (4) -- (40);
			\draw (4) -- (41);
			\draw (4) -- (42);
			\draw (4) -- (43);
			\draw (4) -- (44);
		\end{tikzpicture}
		\caption{$\theta=\big( \triv_{25}\times \cX(1;2)\times \cX(3,0)\times \phi\times\cX(1;1) \big)\up^{P_{125}}$, where $\phi=\big( \phi_2\times\phi_0^{\times 4} \big)\up^{P_{25}}$.\\
			Here, $\val(\theta)=2$, $\eta(\theta)=18$, $\gamma_0(\theta)=16$ and $\gamma_1(\theta)=2$.\\
			Also, $\Omega(\theta)\supseteq\cB_{125}(107)$ and $\Omega(\theta)\setminus\cB_{125}(107)$ contains no thin partitions, and $m(\theta)=107$ and $M(\theta)=109$.}
		\label{fig:125-3}
	\end{subfigure}
	\hfill
	\begin{subfigure}[t]{0.45\textwidth}
		\centering
		\begin{tikzpicture}[scale=1.0, every node/.style={scale=0.7}]
			\draw (0,0.1) node(R) {$\overset{c}{\bullet}$};
			\draw (-2.4,-0.6) node(0) {$\overset{b}{\bullet}$};
			\draw (-1.2,-0.6) node(1) {$\overset{b}{\bullet}$};
			\draw (0,-0.6) node(2) {$\overset{b}{\bullet}$};
			\draw (1.2,-0.6) node(3) {$\overset{b}{\bullet}$};
			\draw (2.4,-0.6) node(4) {$\overset{b}{\bullet}$};
			\draw (R) -- (0);
			\draw (R) -- (1);
			\draw (R) -- (2);
			\draw (R) -- (3);
			\draw (R) -- (4);
			\draw (-2.8,-1.4) node(00) {$\overset{a}{\bullet}$};
			\draw (-2.6,-1.4) node(01) {$\overset{a}{\bullet}$};
			\draw (-2.4,-1.4) node(02) {$\overset{a}{\bullet}$};
			\draw (-2.2,-1.4) node(03) {$\overset{a}{\bullet}$};
			\draw (-2.0,-1.4) node(04) {$\overset{a}{\bullet}$};
			\draw (0) -- (00);
			\draw (0) -- (01);
			\draw (0) -- (02);
			\draw (0) -- (03);
			\draw (0) -- (04);
			\draw (-1.6,-1.4) node(10) {$\overset{a}{\bullet}$};
			\draw (-1.4,-1.4) node(11) {$\overset{a}{\bullet}$};
			\draw (-1.2,-1.4) node(12) {$\overset{a}{\bullet}$};
			\draw (-1.0,-1.4) node(13) {$\overset{a}{\bullet}$};
			\draw (-0.8,-1.4) node(14) {$\overset{a}{\bullet}$};
			\draw (1) -- (10);
			\draw (1) -- (11);
			\draw (1) -- (12);
			\draw (1) -- (13);
			\draw (1) -- (14);
			\draw (-0.4,-1.4) node(20) {$\overset{a}{\bullet}$};
			\draw (-0.2,-1.4) node(21) {$\overset{a}{\bullet}$};
			\draw (0,-1.4) node(22) {$\overset{a}{\bullet}$};
			\draw (0.2,-1.4) node(23) {$\overset{a}{\bullet}$};
			\draw (0.4,-1.4) node(24) {$\overset{a}{\bullet}$};
			\draw (2) -- (20);
			\draw (2) -- (21);
			\draw (2) -- (22);
			\draw (2) -- (23);
			\draw (2) -- (24);
			\draw (0.8,-1.4) node(30) {$\overset{a}{\bullet}$};
			\draw (1.0,-1.4) node(31) {$\overset{a}{\bullet}$};
			\draw (1.2,-1.4) node(32) {$\overset{a}{\bullet}$};
			\draw (1.4,-1.4) node(33) {$\overset{a}{\bullet}$};
			\draw (1.6,-1.4) node(34) {$\overset{a}{\bullet}$};
			\draw (3) -- (30);
			\draw (3) -- (31);
			\draw (3) -- (32);
			\draw (3) -- (33);
			\draw (3) -- (34);
			\draw (2.0,-1.4) node(40) {$\overset{a}{\bullet}$};
			\draw (2.2,-1.4) node(41) {$\overset{a}{\bullet}$};
			\draw (2.4,-1.4) node(42) {$\overset{a}{\bullet}$};
			\draw (2.6,-1.4) node(43) {$\overset{a}{\bullet}$};
			\draw (2.8,-1.4) node(44) {$\overset{a}{\bullet}$};
			\draw (4) -- (40);
			\draw (4) -- (41);
			\draw (4) -- (42);
			\draw (4) -- (43);
			\draw (4) -- (44);
		\end{tikzpicture}
		\caption{$\theta=\cX(a;b;c)$ where $a,b,c\in[0,4]$.\\
			Here, we write $\cX(a;b;c)$ for $\cX( \cX(\phi_a;\phi_b); \phi_c)$, and replace filled circles $\bullet$ by unfilled circles $\circ$ if a vertex label equals 0.}
		\label{fig:125-4}
	\end{subfigure}
	\caption{Examples of $\theta\in\Irr(P_{p^3})$ where $p=5$, together with a representative admissible tree $T\in\cT(\theta)$, tree statistics from Definition~\ref{def:trees-stats} and information on their Sylow branching coefficients.}
	\label{fig:125}
\end{figure}

\begin{table}[t!]
	\[ \def\arraystretch{1.3}
	\begin{array}{|cc|clcc|lcc|}
		\hline
		\theta && \val(\theta) & \gamma_i(\theta)\text{ for }i<\val(\theta) & \eta(\theta) && \Omega(\theta) & m(\theta) & M(\theta) \\
		\hline
		\hline
		\cX(0;0;0) && 0 & & 0 && \puncp{125} & 123 & 125\\
		\hline
		\cX(0;0;1) && 1 & \gamma_0=1 & 1 && \cB_{125}(124) & 124 & 124\\
		\cX(0;1;0) && 1 & \gamma_0=5 & 5 && \puncb{125}{120} & 119 & 120\\
		\cX(1;0;0) && 1 & \gamma_0=25 & 25 && \puncb{125}{100} & 99 & 100\\
		\hline
		\cX(0;1;1) && 2 & \gamma_0=5,\ \gamma_1=1 & 6 && {\scriptstyle \cB_{125}(119) \sqcup \close{\{(120,\mu)\mid\mu\in\cB_5(4) \}}} & 119 & 120\\
		\cX(1;0;1) && 2 & \gamma_0=25,\ \gamma_1=1 & 26 && {\scriptstyle \cB_{125}(99) \sqcup \close{\{(100,\mu)\mid\mu\in\cB_{25}(24) \}}} & 99 & 100\\
		\cX(1;1;0) && 2 & \gamma_0=25,\ \gamma_1=5 & 30 && {\scriptstyle\text{(see Example~\ref{ex:125})}} & 95 & 100\\
		\hline
		\cX(1;1;1) && 3 & \gamma_0=25,\ \gamma_1=5,\ \gamma_2=1 & 31 && {\scriptstyle\text{(see Example~\ref{ex:125})}} & 95 & 100\\
		\hline
	\end{array} \]
	\caption{Equivalence class representatives for $\sim_{0,p}$ on the linear characters in $\Irr(P_{p^3})$, where $p=5$, their tree statistics, and information on their Sylow branching coefficients.}
	\label{tab:125}
\end{table}

We move now to the analysis of $m(\theta)$ for any $\theta\in\Irr(P_{p^k})$, recalling its definition from \eqref{eq:m-and-M} in the introduction. We start by studying the situation for all those irreducible characters $\theta$ such that $\val(\theta)\leq 1$. As we will see, in this case we will be able to completely determine the set $\Omega(\theta)$. 
In the following remark, we first settle the case where $\val(\theta)=0$.

\begin{remark}\label{rem: Value0}
	It is easy to see that $\val(\theta)=0$ if and only if $\theta=\triv_{p^k}$. 
	Let $\triv=\triv_{p^k}$. Then by Theorem~\ref{thm:GL1} we know that $\Omega(\triv)=\puncp{p^k}$, and therefore $m(\triv)=p^k-2$. Moreover, we observe that $\gamma_0(\triv)=0$ and that for any $\lambda\in\close{\{\tworow{n}{n-\gamma_0(\triv)},\hook{n}{n-\gamma_0(\triv)} \}}=\{(n),(1^n)\}$, we clearly have that $[\chi^\lambda\down_{P_n}, \triv]=1$.
	\hfill$\lozenge$
\end{remark}

Before moving on to studying $\Omega(\theta)$ for $\theta$ such that $\val(\theta)= 1$, we state and prove a technical lemma that will be repeatedly used in the proofs of our main results. 

\begin{lemma}\label{lem:usual-PW/GTT+LR}
	Let $k\in\N_{\geq 2}$ and let $N\in [1, p^{k-1}]$. Let $\theta=\cX(\varphi;\varepsilon)\in\Irr(P)$ for some $\varphi\in\Irr(Q)$ and $\varepsilon\in[0, p-1]$. Suppose that $\lambda\in\cP(p^k)$, $\nu\in\cP(p)$ and $\alpha\in\{\tworow{p^{k-1}}{N}, \hook{p^{k-1}}{N}\}$ satisfy the following conditions:
	\begin{enumerate}[label=(\roman*)]
		\item[(i)] $[\alpha\down_Q,\varphi]=1$, 
		\item[(ii)] If $\lambda^1,\dotsc,\lambda^p\in\Omega(\varphi)$ and $\LR(\lambda;\lambda^1,\ldots,\lambda^p)\neq 0$, then $\lambda^i=\alpha$ for all $i\in [1,p]$,
		\item[(iii)] $[\lambda\down_V,\cX(\alpha;\nu)]=1$, and
		\item[(iv)] $[\lambda\down_Y,\alpha^{\times p}]=\chi^\nu(1)$.
	\end{enumerate}
	Then $[\lambda\down_P,\theta]=[\cX(\alpha;\nu)\down_P,\theta] = [\nu\down_{P_p},\phi_\varepsilon]$. In particular, if $\lambda=p\alpha$ then there is some $\nu\in\{(p),(1^p)\}$ satisfying (iii) and (iv).
\end{lemma}

\begin{proof}
	We first prove the final assertion, that if $\lambda=p\alpha$ then there is some $\nu\in\{(p),(1^p)\}$ satisfying (iii) and (iv). Indeed, if $\alpha=\hook{p^{k-1}}{N}$ then $[\lambda\down_V,\cX(\alpha;\nu)]=1$ for some $\nu\in\{(p),(1^p)\}$ by Theorem~\ref{thm:GTT3.5}, while if $\alpha=\tworow{p^{k-1}}{N}$ then $[\lambda\down_V,\cX(\alpha;(p))]=1$ by Theorem~\ref{thm:PW9.1}. Moreover, clearly $[\lambda\down_Y,\alpha^{\times p}]=1$ from the Littlewood--Richardson rule, and $\chi^\nu(1)=1$ for $\nu\in\{(p),(1^p)\}$.
	
	Next, we turn to showing that $[\lambda\down_P,\theta]=[\cX(\alpha;\nu)\down_P,\theta]$.
	From (iii), we may write $\lambda\down_V=\cX(\alpha;\nu)+\Delta$ for some $\Delta\in\Char(V)$.
	Now, $[\lambda\down_Y,\alpha^{\times p}]=\chi^\nu(1)$ from (iv) and also $[\cX(\alpha;\nu)\down_Y,\alpha^{\times p}]=\chi^\nu(1)$, so $[\Delta\down_Y,\alpha^{\times p}]=0$. 
	Then since $\theta\down_B=\varphi^{\times p}$, we have that
	\begin{align*}
		[\Delta\down_P,\theta] \le [\Delta\down_B,\varphi^{\times p}] &= \sum_{\lambda^1,\ldots,\lambda^p\in\cP(p^{k-1})} [\Delta\down_Y,\lambda^1\times\cdots\times\lambda^p]\cdot \prod_{i=1}^p [\lambda^i\down_Q,\varphi] \\
		&= \sum_{\substack{\lambda^1,\dots\lambda^p\in\Omega(\varphi) \\ \LR(\lambda;\lambda^1,\ldots,\lambda^p)>0}} [\Delta\down_Y,\lambda^1\times\cdots\times\lambda^p]\cdot \prod_{i=1}^p [\lambda^i\down_Q,\varphi] = [\Delta\down_Y, \alpha^{\times p}] \cdot [\alpha\down_Q,\varphi]^{\times p},
	\end{align*}
	where the final equality follows from (ii). Thus $[\Delta\down_P,\theta] = 0$.
	Hence
	\[ [\lambda\down_P,\theta] = [\cX(\alpha;\nu)\down_P,\theta] + [\Delta\down_P,\theta] =  [\cX(\alpha;\nu)\down_P,\theta]. \]
	Finally, since we have $[\alpha\down_Q,\varphi]=1$ by (i) then $[\cX(\alpha;\nu)\down_P,\theta]=[\nu\down_{P_p},\phi_\varepsilon]$ follows from Lemma~\ref{lem:SL2.19}.
\end{proof}

We are now ready to prove one of the main results of this section. 

\begin{theorem}\label{thm:value-1}
	Let $k\in\N$ and $\theta\in\Irr(P_{p^k})$. Suppose $\val(\theta)=1$. Then
	\[ \Omega(\theta)\ \text{equals either}\ \cB_{p^k}(p^k-\gamma_0(\theta))\text{ or }\puncb{p^k}{p^k-\gamma_0(\theta)}. \]
	Moreover, for all $\lambda\in \close{\{\tworow{p^k}{p^k-\gamma_0(\theta)}, \hook{p^k}{p^k-\gamma_0(\theta)}\}}$, we have that $[\lambda\down_{P_{p^k}},\theta]=1$.
\end{theorem}

\begin{proof}
	We use the notation introduced in Notation~\ref{not:subgps}.
	We proceed by induction on $k$, noting that $k\le 2$ follows from Lemma~\ref{lem:small-k}. We may now suppose $k\ge 3$, and split our argument into three cases depending on the form of the character $\theta\in\Irr(P)$.

	\smallskip
	
	\noindent\textbf{Case 1:} For the first case, suppose that $\theta=\theta_1\times\cdots\times\theta_p\up^P$ for some mixed $\theta_1,\ldots, \theta_p\in\Irr(Q)$. Then 
	\[ \val(\theta)=\max_{i\in[1, p]} \val(\theta_i)\quad\text{and}\quad \gamma_0(\theta)=\sum_{i=1}^p \gamma_0(\theta_i). \]
	Since $\val(\theta)=1$, it follows that $\val(\theta_i)\in\{0,1\}$ for all $i\in [1,p]$. 
	If $\val(\theta_i)=0$ then Theorem~\ref{thm:GL1} shows that $\Omega(\theta_i)=\puncp{p^{k-1}}$. 
	On the other hand, the inductive hypothesis guarantees that if $\val(\theta_i)=1$, then $\Omega(\theta_i) = \cB_{p^{k-1}}(p^{k-1}-\gamma_0(\theta_i))$ or $\puncb{p^{k-1}}{p^{k-1}-\gamma_0(\theta_i)}$.
	
	Using Theorem~\ref{prop:smooth}, we deduce that $\Omega(\theta)=\cB_{p^k}(p^k-\gamma_0(\theta))$ or $\puncb{p^k}{p^k-\gamma_0(\theta)}$. We remark that we can invoke Theorem~\ref{prop:smooth} as the technical hypotheses involved in its statement are satisfied here. That is, $p^{k-1}\ge 5$ and $\gamma_0(\theta_i)\ge p$ whenever $\Omega(\theta_i)=\puncb{p^{k-1}}{p^{k-1}-\gamma_0(\theta_i)}$ by Lemma~\ref{lem:small-gamma_0} (since $\theta_i\ne\triv_Q$ in this case and so $\gamma_0(\theta_i)\ne 0$).
	
	Notice also that $\Omega(\theta)=\puncb{p^k}{p^k-\gamma_0(\theta)}$ can only occur if there is a unique $i\in[1, p]$ such that $\Omega(\theta_i)=\puncb{p^{k-1}}{p^{k-1}-\gamma_0(\theta_i)}$, and $\Omega(\theta_j)=\puncp{p^{k-1}}$ for all $j\in [1, p]\setminus\{i\}$ (that is, $\theta_j=\triv_Q$). 

	Finally, let $\lambda\in \close{\{\tworow{p^k}{p^k-\gamma_0(\theta)}, \hook{p^k}{p^k-\gamma_0(\theta)}\}}$. 
	By Lemma~\ref{lem:conj-restr}, we can assume without loss of generality that $\lambda\in\{\tworow{p^k}{p^k-\gamma_0(\theta)}, \hook{p^k}{p^k-\gamma_0(\theta)}\}$. 
	Suppose first that $\lambda=\tworow{p^k}{p^k-\gamma_0(\theta)}$. 			
	Let $\overline{\theta}:=\theta_1\times\cdots\times\theta_p\in\Irr(B)$. First, $[\lambda\down_P,\theta] = [\lambda\down_B,\overline{\theta}]$ since $\Irr(P\mid\overline{\theta})=\{\theta\}$ and $[\theta\down_B,\overline{\theta}]=1$. We have that
	\[ [\lambda\down_B,\overline{\theta}] = \sum_{\lambda^1,\dotsc,\lambda^p\in\cP(p^{k-1})} \LR(\lambda;\lambda^1,\ldots,\lambda^p) \cdot \prod_{i=1}^p [\lambda^i\down_Q, \theta_i]. \]
	Now $\LR(\lambda;\lambda^1,\ldots,\lambda^p)>0$ implies that $\ell(\lambda^i)\le \ell(\lambda)\le 2$ for each $i\in [1,p]$, and that $\lambda_1\le \sum_{i=1}^p (\lambda^i)_1$. On the other hand $[\lambda^i\down_Q, \theta_i]>0$ implies that $\lambda^i\in\Omega(\theta_i)$ and therefore that $(\lambda^i)_1\le p^{k-1}-\gamma_0(\theta_i)$ by Theorem~\ref{thm:M}. Since $\lambda_1 = p^k-\gamma_0(\theta)=\sum_{i=1}^p (p^{k-1}-\gamma_0(\theta_i))$, we have that
	\[ [\lambda\down_P,\theta] = [\lambda\down_B,\overline{\theta}] = \LR(\lambda; \ft^1,\ldots,\ft^p)\cdot \prod_{i=1}^p [\ft^i\down_Q,\theta_i] \]
	where $\ft^i:=\tworow{p^{k-1}}{p^{k-1}-\gamma_0(\theta_i)}$. 
	Clearly $\LR(\lambda; \ft^1,\ldots,\ft^p)=1$ since $\lambda=\ft^1+\cdots+\ft^p$. Moreover, if $\val(\theta_i)=0$ then $\theta_i=\triv_Q$ and so $[\ft^i\down_Q,\theta_i]=1$. On the other hand, if  $\val(\theta_i)=1$ then $[\ft^i\down_Q,\theta_i]=1$ by the inductive hypothesis. 
	It follows that $[\lambda\down_P,\theta]=1$, as desired.
	The case of $\lambda=\hook{p^k}{p^k-\gamma_0(\theta)}$ follows from entirely analogous arguments.

	\smallskip
	
	\noindent\textbf{Case 2:} For the second case, suppose now that $\theta=\cX(\varphi;\varepsilon)$ with $\varepsilon\in[1,p-1]$ and $\varphi\in\Irr(Q)$. Then $\val(\theta)=1$ implies that $\gamma_0(\theta)=1$ and $\val(\varphi)=0$, and so $\varphi=\triv_Q$. Then $\Omega(\theta)=\cB_{p^k}(p^k-1)$ and $[\lambda\down_P,\theta]=1$ for all $\lambda\in \close{\{\tworow{p^k}{p^k-1}, \hook{p^k}{p^k-1}\}}$ by Lemma~\ref{lem:GL2-4.3}.		
	
	\smallskip
	
	\noindent\textbf{Case 3:} For the third and final case, suppose that $\theta=\cX(\varphi;0)$ for some $\varphi\in\Irr(Q)$. Then $\gamma_0(\theta)=p\cdot\gamma_0(\varphi)$ and $\val(\varphi)=\val(\theta)=1$. By the inductive hypothesis, $\Omega(\varphi)$ equals either $\cB_{p^{k-1}}(p^{k-1}-\gamma_0(\varphi))$ or $\puncb{p^{k-1}}{p^{k-1}-\gamma_0(\varphi)}$. We analyse these two subcases separately. 
	
	\smallskip
	
	\noindent\textbf{Subcase 3.1:} For the first of these two subcases, suppose $\Omega(\varphi)=\cB_{p^{k-1}}(p^{k-1}-\gamma_0(\varphi))$. 
	Observe that
	\[ \cB_{p^k}(p^k-\gamma_0(\theta)-1) \subseteq \cM(p,\cB_{p^{k-1}}(p^{k-1}-\gamma_0(\varphi))) = \cM(p,\Omega(\varphi)) \subseteq \Omega(\theta), \]
	where the first inclusion follows from Proposition~\ref{prop:mixed}, and the final inclusion follows from Lemma~\ref{lem:mixed}. (Note the hypotheses of Proposition~\ref{prop:mixed} are satisfied as $\tfrac{p^{k-1}}{2}+1<p^{k-1}-\gamma_0(\varphi)$, because $k\ge 3$, $p\ge 5$ and $\gamma_0(\varphi)\le p^{k-2}$ from Lemma~\ref{lem:tree-bounds}.)
	We also remark that $\Omega(\theta)\subseteq\cB_{p^k}(p^k-\gamma_0(\theta))$ by Theorem~\ref{thm:M}. 
	Hence, it remains to study $\lambda\in\cP(p^k)$ with $\lambda_1=p^k-\gamma_0(\theta)$, by Lemma~\ref{lem:conj-restr}.
	Let us consider $\lambda=(p^k-\gamma_0(\theta),\mu)$ for some arbitrary $\mu\in\cP(\gamma_0(\theta))$.
	We split our analysis according to the value of $\gamma_0(\varphi)$.
	
	\noindent\textit{Subcase 3.1(i):} If $\gamma_0(\varphi)=1$, then $\lambda=(p^k-p,\mu)$ and $\mu\in\cP(p)$. Then using Lemma~\ref{lem:usual-PW/GTT+LR} with $\alpha=(p^{k-1}-1,1)$ and $\nu=\mu$ we obtain that $[\lambda\down_P,\theta]=[\mu\down_{P_p},\phi_0]$. 
	We observe that the four conditions required in the statement of Lemma~\ref{lem:usual-PW/GTT+LR} are satisfied in this setting with $N=p^{k-1}-\gamma_0(\varphi)=p^{k-1}-1$, since:
	\begin{itemize}
		\item Condition (i) of Lemma~\ref{lem:usual-PW/GTT+LR} is satisfied since $[\alpha\down_Q,\varphi]=1$ is given by the inductive hypothesis.
		\item Condition (ii) follows since $\LR(\lambda;\lambda^1,\ldots,\lambda^p)>0$ implies $\lambda^i\subseteq\lambda$ for all $i\in [1, p]$ and $p^k-p=\lambda_1\le \sum_{i=1}^p (\lambda^i)_1$, but $\lambda^i\in\Omega(\varphi)$ gives $(\lambda^i)_1\le p^{k-1}-1$, whence $\lambda^i=\alpha=(p^{k-1}-1,1)$ for each $i\in [1, p]$.
		\item Condition (iii) follows from Theorem~\ref{thm:PW9.1}.
		\item Condition (iv) follows from an easy application of the Littlewood--Richardson rule.
	\end{itemize}
	Then by Lemma~\ref{lem:small-k},
	\[ [\mu\down_{P_p},\phi_0] \begin{cases}
		= 0 & \text{if }\mu\in\close{\{(p-1,1)\}},\\
		= 1 & \text{if }\mu\in\close{\{(p)\}},\\
		\ge 1 & \text{if }\mu\in\cB_p(p-2).
	\end{cases} \]
	Thus, in this case we have proved both $\Omega(\theta)=\puncb{p^k}{p^k-\gamma_0(\theta)}$ as well as the part of the theorem statement concerning multiplicities for the specified thin partitions. 
				
	\noindent\textit{Subcase 3.1(ii):} Assume now that $\gamma_0(\varphi)\ge 2$. Observe that
	\[ \cB_{\gamma_0(\theta)}(\gamma_0(\theta)-1) = \cB_{p\gamma_0(\varphi)}(p\gamma_0(\varphi)-1) \subseteq \cM(p, \cP(\gamma_0(\varphi))) \]
	by Proposition~\ref{prop:mixed}. Recall that $\lambda=(p^k-\gamma_0(\theta),\mu)$ where $\mu\in\cP(\gamma_0(\theta))$. Hence if $\mu\in \cB_{\gamma_0(\theta)}(\gamma_0(\theta)-1)$, then there exist mixed $\mu^1,\ldots,\mu^p\in \cP(\gamma_0(\varphi))$ such that $\LR(\mu;\mu^1,\ldots,\mu^p)\neq 0$. Setting $\lambda^i:=(p^{k-1}-\gamma_0(\varphi), \mu^i)$ for each $i\in [1,p]$, we have that $\lambda^1,\ldots, \lambda^p\in \cB_{p^{k-1}}(p^{k-1}-\gamma_0(\varphi))$ are mixed and $\LR(\lambda;\lambda^1,\ldots,\lambda^p)=\LR(\mu;\mu^1,\ldots,\mu^p)\neq 0$ by Lemma~\ref{lem:iterated-LR}. It follows that $\lambda\in \cM(p, \cB_{p^{k-1}}(p^{k-1}-\gamma_0(\varphi)))$. Since $\cB_{p^{k-1}}(p^{k-1}-\gamma_0(\varphi))=\Omega(\varphi)$, we conclude that $\lambda\in\cM(p,\Omega(\varphi)) \subseteq \Omega(\theta)$ by Lemma~\ref{lem:mixed}.

	Finally, let $\lambda=\ft_{p^k}(p^k-\gamma_0(\theta))$; the argument for $\lambda=\fh_{p^k}(p^k-\gamma_0(\theta))$ is entirely analogous. Then $[\lambda\down_P,\theta]=[\nu\down_{P_p},\phi_0]=1$ for some $\nu\in\{(p),(1^p)\}$ from Lemma~\ref{lem:usual-PW/GTT+LR}, since $\lambda=p\alpha$ where $\alpha:=\ft_{p^{k-1}}(p^{k-1}-\gamma_0(\varphi))$. Note that conditions (i) and (ii) of Lemma~\ref{lem:usual-PW/GTT+LR} are satisfied since $\Omega(\varphi)=\cB_{p^{k-1}}(p^{k-1}-\gamma_0(\varphi))$ and $[\alpha\down_Q,\varphi]=1$ by the inductive hypothesis. Thus, also in this case we have proved both $\Omega(\theta)=\cB_{p^k}(p^k-\gamma_0(\theta))$ as well as	the statement on multiplicities. 			
		
	\smallskip
	
	\noindent\textbf{Subcase 3.2:} For the second of these two subcases, suppose now that $\Omega(\varphi)=\puncb{p^{k-1}}{p^{k-1}-\gamma_0(\varphi)}$. In particular, $\gamma_0(\varphi)\ge p$ by Lemma~\ref{lem:small-gamma_0}, since Theorem~\ref{thm:GL1} implies $\varphi\ne\triv_Q$. 
	Observe that 
	\[ \cB_{p^k}(p^k-\gamma_0(\theta)-1)\subseteq \cM(p,\Omega(\varphi))\subseteq\Omega(\theta), \]
	where the first inclusion follows from Lemma~\ref{lem:A'} (with $q=p$, $m=p^{k-1}$ and $t=p^{k-1}-\gamma_0(\varphi)$, which satisfy $q\ge 3$, $m-t=\gamma_0(\varphi)\ge p\ge 5$, and $t>\tfrac{m}{2}+2$ since $\gamma_0(\varphi)\le p^{k-2}$ by Lemma~\ref{lem:tree-bounds} and $k\ge 3$), and the second from Lemma~\ref{lem:mixed}. 
	Also, $\Omega(\theta)\subseteq\cB_{p^k}(p^k-\gamma_0(\theta))$ by Theorem~\ref{thm:M}. So it remains to consider $\lambda=(p^k-\gamma_0(\theta),\mu)$ for arbitrary $\mu\in\cP(\gamma_0(\theta))$. As we have done before, here we are assuming without loss of generality that $\lambda_1\ge \ell(\lambda)$. 
	
	\smallskip
	
	\noindent\textit{Subcase 3.2(i):} First suppose that $\mu\in\close{ \{ (\gamma_0(\theta)-1,1) \} }$.
	
	If $\mu=(\gamma_0(\theta)-1,1)$, then $\lambda\notin\Omega(\theta)$ because, letting $\alpha=\tworow{p^{k-1}}{p^{k-1}-\gamma_0(\varphi)}$ and $\nu=(p-1,1)$, Lemma~\ref{lem:usual-PW/GTT+LR} shows that $[\lambda\down_P,\theta] = [(p-1,1)\down_{P_p},\phi_0]=0$. 
	Indeed, the four conditions in the statement of Lemma~\ref{lem:usual-PW/GTT+LR} hold since (i) is given by the inductive hypothesis, (iii) follows from Theorem~\ref{thm:PW9.1} and (iv) from the Littlewood--Richardson rule. The reason that (ii) holds in the present situation where $\Omega(\varphi)=\puncb{p^{k-1}}{p^{k-1}-\gamma_0(\varphi)}$ is because if $\LR(\lambda;\lambda^1,\dotsc,\lambda^p)>0$ then $\lambda^i\subseteq\lambda$ for all $i\in[1,p]$ and $p^k-\gamma_0(\theta)=\lambda_1\le\sum_{i=1}^p (\lambda^i)_1$, while $\lambda^i\in\Omega(\varphi)$ gives $(\lambda^i)_1\le p^{k-1}-\gamma_0(\varphi)$. Then $p\gamma_0(\varphi)=\gamma_0(\theta)$ together with $\lambda^i\subseteq\lambda$ shows that $\lambda^i\in\{ \alpha,\beta \}$ where $\beta=(p^{k-1}-\gamma_0(\varphi),\gamma_0(\varphi)-1,1)$. But $\beta\notin\Omega(\varphi)$, and so this forces $\lambda^i=\alpha$ for all $i$.
	
	If $\mu=(\gamma_0(\theta)-1,1)'$, then a completely analogous argument shows that $\lambda\notin\Omega(\theta)$ in this instance also.
	
	\smallskip
	
	\noindent\textit{Subcase 3.2(ii):} Otherwise, $\mu\in\puncp{\gamma_0(\theta)}$.
	We will show that $\lambda\in\Omega(\theta)$ to conclude that $\Omega(\theta)=\puncb{p^{k-1}}{p^{k-1}-\gamma_0(\theta)}$. Moreover, we will also show that $[\lambda\down_P,\theta]=1$ whenever $\mu\in \close{\{(\gamma_0(\theta))\}}$, to complete the proof of the theorem.
	
	First, suppose that $\mu\in\close{\{(\gamma_0(\theta))\}}$. In this case we have that $\lambda\in\{\ft_{p^k}(p^k-\gamma_0(\theta)), \fh_{p^k}(p^k-\gamma_0(\theta)) \}$. Then we have that $\lambda=p\alpha$, for some $\alpha\in \{\ft_{p^{k-1}}(p^{k-1}-\gamma_0(\varphi)), \fh_{p^{k-1}}(p^{k-1}-\gamma_0(\varphi))\}$. Hence, we can use Lemma~\ref{lem:usual-PW/GTT+LR} to deduce that $[\lambda\down_P,\theta] = [\nu\down_{P_p},\phi_0]=1$.
	
	Now it remains to consider $\mu\in\cB_{\gamma_0(\theta)}(\gamma_0(\theta)-2)$.
	We will show below that $\cB_{\gamma_0(\theta)}(\gamma_0(\theta)-2)\subseteq \cM(p,\puncp{\gamma_0(\varphi)})$.
	Once we have that this holds, then this will mean that given any $\mu\in \cB_{\gamma_0(\theta)}(\gamma_0(\theta)-2)$ there exist mixed $\mu^1,\ldots,\mu^p\in\puncp{\gamma_0(\varphi)}$ such that $\LR(\mu;\mu^1,\ldots,\mu^p)\neq 0$. 
	This implies that $\LR(\lambda;(p^{k-1}-\gamma_0(\varphi),\mu^1),\ldots,(p^{k-1}-\gamma_0(\varphi),\mu^p))\neq 0$ by Lemma~\ref{lem:iterated-LR}. Notice also that for every $i\in [1,p]$ the partition $(p^{k-1}-\gamma_0(\varphi),\mu^i)$ is well-defined and belongs to $\puncb{p^{k-1}}{p^{k-1}-\gamma_0(\varphi)}$; this follows since $\gamma_0(\varphi)\le p^{k-2}$ and $k\ge 3$. 
	Using Lemma~\ref{lem:mixed}, we conclude that $\lambda\in\cM(p,\Omega(\varphi))\subseteq\Omega(\theta)$, as desired.
	
	Hence, to complete the proof of the theorem, we show that $\cB_{\gamma_0(\theta)}(\gamma_0(\theta)-2)\subseteq \cM(p,\puncp{\gamma_0(\varphi)})$.
	Let us fix $m:=\gamma_0(\varphi)$.
	Recall that $m\geq p$ and that $\gamma_0(\theta)=pm$. 
	\begin{itemize}
		\item If $p=5$ and $m\in\{5,6\}$ then one can see that $\cB_{25}(23)\subseteq \cM(5,\puncp{5})$ and that $\cB_{30}(28)\subseteq \cM(5,\puncp{6})$ by direct computation.
		Since $m\ge p$, we can now assume that $m>6$. Hence, we have $\tfrac{m}{2}+1<m-2$ and therefore we may use Proposition~\ref{prop:mixed} to see that $\cB_{pm}(pm-2p)\setminus \close{\{(pm-2p,2p),(pm-2p,1^{2p}) \}} \subseteq \cM(p,\puncp{m})$.
		
		\item If $\mu=(pm-2p,2p)$ then $\LR(\mu;(m),(m-2,2),\ldots,(m-2,2))\neq 0$, and so $\mu\in\cM(p,\puncp{m})$, while if $\mu=(pm-2p,1^{2p})$, then $\LR(\mu;(m),(m-2,1^2),\ldots,(m-2,1^2))\neq 0$	and so $\mu\in\cM(p,\puncp{m})$. 
		Similarly if $\mu\in\{ (pm-2p,2p)', (pm-2p,1^{2p})' \}$ then $\mu\in\cM(p,\puncp{m})$ also.
		Thus $\cB_{pm}(pm-2p)\subseteq\cM(p,\puncp{m})$.
		
		\item Next, suppose $\mu=(pm-p-r,\nu)$ for some $r\in [1,p-1]$ and $\nu\in\cP(p+r)$. 
		Clearly there exist $\nu^1,\ldots,\nu^r\in\cP(2)$ such that $\LR(\nu;\nu^1,\ldots,\nu^r,(1),\ldots,(1))\neq 0$. 
		It follows that $\LR(\mu;(m-2,\nu^1),\ldots,(m-2,\nu^r),(m),\ldots,(m))\neq 0$, since $r\ge 1$. 
		Also since $r\ge1$ and $p-r\ge1$, then $(m-2,\nu^1),\ldots,(m-2,\nu^r),(m),\ldots,(m)$ are mixed, whence $\mu\in\cM(p,\puncp{m})$.
		
		\item Finally, suppose $\mu=(pm-r,\nu)$ for some $r\in\{2,3,\ldots,p\}$ and $\nu\in\cP(r)$. Choose any $\omega\in\cP(2)$ such that $\omega\subset\nu$. Then clearly $\LR(\mu;(m-2,\omega),(m),\ldots,(m))\ne0$, and so $\mu\in\cM(p,\puncp{m})$.
	\end{itemize}
	Thus we have shown that $\cB_{pm}(pm-2)\subseteq\cM(p,\puncp{\gamma_0(\varphi)})$, which completes the proof.
\end{proof}

Theorem~\ref{thm:value-1} describes $\Omega(\theta)$ in terms of an easy combinatorial statistic, namely $\gamma_0(\theta)$, computed using the associated trees in $\cT(\theta)$, for every $\theta\in\Irr(P_{p^k})$ with $\val(\theta)=1$. 
Indeed, we have shown that $\Omega(\theta)$ is either $\cB_{p^k}(p^k-\gamma_0(\theta))$ or $\puncb{p^k}{p^k-\gamma_0(\theta)}$: we say that $\Omega(\theta)$ is \emph{punctured} if and only if $\Omega(\theta)=\puncb{p^k}{p^k-\gamma_0(\theta)}$.

We now complete the analysis of $\theta$ when $\val(\theta)=1$ by giving a concrete condition on the trees in $\cT(\theta)$ to distinguish when $\Omega(\theta)$ is punctured or not. 
To start, we collect in the following lemma something we have already observed in the course of the proof of Theorem~\ref{thm:value-1}. 

\begin{lemma}\label{lem:recursive-punctures}
	Let $k\in\N$ and let $\theta\in\Irr(P_{p^k})$ with $\val(\theta)=1$. Let $T\in\cT(\theta)$. Then $\Omega(\theta)$ is punctured if and only if one of the following holds:
	\begin{itemize}		
		\item[(i)] $[T]=[(T'\mid\dotsc\mid T';0)]$ where $\varphi:=\theta(T')\in\Irr(P_{p^{k-1}})$ has $\eta(\varphi)=1$; or
		\item[(ii)] $[T]=[(T'\mid\dotsc\mid T';0)]$ where $\varphi:=\theta(T')\in\Irr(P_{p^{k-1}})$ is such that $\Omega(\varphi)$ is punctured; or
		\item[(iii)] $[T]=[(T_1\mid\dotsc\mid T_p;p)]$ where, letting $\theta_i:=\theta(T_i)\in\Irr(P_{p^{k-1}})$ for each $i\in[1,p]$, we have that $\theta_1,\dotsc,\theta_p$ are mixed, and there exists a unique $j\in[1,p]$ such that $\Omega(\theta_j)$ is punctured and $\theta_i=\triv_{p^{k-1}}$ for all $i\ne j$.
	\end{itemize}
\end{lemma}

\begin{proof}
	Recall that $\val(\theta)=1$ implies that $\eta(\theta)=\gamma_0(\theta)$. With this in mind, one can recover the statement of the present lemma directly from the case by case analysis performed in the proof of Theorem~\ref{thm:value-1}. 
\end{proof}

Lemma~\ref{lem:recursive-punctures} gives a recursive algorithm to deduce whether $\Omega(\theta)$ is punctured or not. 
We are ready to translate this recursive information into a concrete combinatorial condition on the trees associated to $\theta$.

\begin{theorem}\label{thm:punctures}
	Let $k\in\N$ and let $\theta\in\Irr(P_{p^k})$ with $\val(\theta)=1$. Let $T\in\cT(\theta)$. Then $\Omega(\theta)$ is punctured if and only if both of the following conditions hold: 		
	\begin{itemize}
		\item[(a)] $\eta(\theta)\ge 2$, i.e.~$T$ contains at least two people.
		\item[(b)] For any two people in $T$, their closest common graph-theoretic ancestor has label 0.
	\end{itemize}
\end{theorem}

\begin{proof}
	From Lemma~\ref{lem:small-k} we observe that $\Omega(\theta)$ is never punctured when $k=1$. Moreover, if $k=2$ then $\Omega(\theta)$ is punctured if and only if $\theta=\cX(\phi_\varepsilon;\phi_0)$ for some $\varepsilon\in[1, p-1]$, and notice $\eta(\phi_\varepsilon)=1$. This shows that both the implications of the statement hold whenever $k\le 2$. 

	By Lemma~\ref{lem:recursive-punctures}, we know that if $\Omega(\theta)$ is punctured then $[T]$ is of form (i), (ii) or (iii) in the statement of Lemma~\ref{lem:recursive-punctures}. 
	Proceeding by induction on $k\in\N$, it is not difficult to see that both conditions (a) and (b) hold in each of the three cases. 

	The converse implication deserves a little more explanation. Suppose that $\theta$ satisfies (a) and (b). Again, we proceed by induction on $k$. 
	If $k=2$ then (a), (b) and $\val(\theta)=1$ imply that $T=(\substack{\varepsilon\\\bullet}\mid\ldots\mid \substack{\varepsilon\\\bullet};0)$ for some $\varepsilon\in[1, p-1]$. 
	Then $\Omega(\theta)$ is indeed punctured by Lemma~\ref{lem:small-k}. 
	Now suppose $k\ge 3$ and write $T=(T_1\mid\ldots\mid T_p;x)$. For each $i\in [1,p]$ let $\theta_i:=\theta(T_i)$. Since $T$ satisfies (a) and (b) and $\val(\theta)=1$, the root of $T$ cannot be a person, i.e.~we must have $\ell(x)\in\{0,p\}$.
	
	If $\ell(x)=p$ then $T$ satisfying (a) and (b) implies that all of the people in $T$ are in a single subtree $T_j$, for some unique $j\in[1, p]$. In particular, this means for each $i\in[1, p]\setminus\{j\}$, every vertex in $T_i$ must have label 0 or $p$. But $T_i$ is admissible, so all vertices in $T_i$ must have label 0, whence $\theta_i=\triv_{p^{k-1}}$. Moreover, that $T_j$ itself satisfies (a) and (b) is then immediately inherited from $T$, so by inductive hypothesis $\Omega(\theta_j)$ is punctured. Thus $T$ satisfies condition (iii) of Lemma~\ref{lem:recursive-punctures} and therefore  $\Omega(\theta)$ is punctured.
			
	On the other hand, if $\ell(x)=0$ then $[T_1]=\cdots=[T_p]$ and $\eta(\theta)=p\eta(T_1)\ge 2$. If $\eta(T_1)=1$ then $T$ satisfies condition (i) of Lemma~\ref{lem:recursive-punctures} and therefore $\Omega(\theta)$ is punctured. Otherwise, $\eta(T_1)\ge 2$, that is, $T_1$ satisfies condition (a). Clearly $T_1$ also satisfies (b), inherited from $T$, so by inductive hypothesis $\Omega(\theta_1)$ is punctured. Thus $T$ satisfies condition (ii) of Lemma~\ref{lem:recursive-punctures} and $\Omega(\theta)$ is punctured.
\end{proof}

\begin{example}\label{ex:closest-ancestor}
	We illustrate examples of the combinatorial condition (b) from Theorem~\ref{thm:punctures} in Figure~\ref{fig:puncture}. For simplicity, we have drawn 3-ary trees instead of $p$-ary trees where $p\ge 5$.
	
	The only $\theta\in\Irr(P_{p^2})$ with punctured $\Omega(\theta)$ are $\theta\big((\substack{\varepsilon\\\bullet}\mid\dotsc\mid\substack{\varepsilon\\\bullet};0)\big)$ where $\varepsilon\in[1,p-1]$. These correspond to $\theta=\cX(\phi_\varepsilon;\phi_0)$.
	In Figure~\ref{fig:example-punc}, we list all of the $\theta\in\Irr(P_{p^3})$ for which $\Omega(\theta)$ is punctured. 
	\hfill$\lozenge$
\end{example}

\begin{figure}[h!]
	\centering
	\begin{subfigure}[t]{0.3\textwidth}
		\centering
		\begin{tikzpicture}[scale=1.0, every node/.style={scale=0.7}]
			\draw (0,0.2) node(R) {$a$};
			\draw (-1.2,-0.6) node(1) {$b$};
			\draw (0,-0.6) node(2) {$c$};
			\draw (1.2,-0.6) node(3) {$d$};
			\draw (R) -- (1);
			\draw (R) -- (2);
			\draw (R) -- (3);
			\draw (-1.5,-1.4) node(11) {$e$};
			\draw (-1.2,-1.4) node(12) {$f$};
			\draw (-0.9,-1.4) node(13) {$g$};
			\draw (1) -- (11);
			\draw (1) -- (12);
			\draw (1) -- (13);
			\draw (-0.3,-1.4) node(21) {$h$};
			\draw (0,-1.4) node(22) {$i$};
			\draw (0.3,-1.4) node(23) {$j$};
			\draw (2) -- (21);
			\draw (2) -- (22);
			\draw (2) -- (23);
			\draw (0.9,-1.4) node(31) {$l$};
			\draw (1.2,-1.4) node(32) {$m$};
			\draw (1.5,-1.4) node(33) {$n$};
			\draw (3) -- (31);
			\draw (3) -- (32);
			\draw (3) -- (33);
		\end{tikzpicture}
		\caption{The closest common graph-theoretic ancestor of vertex $e$ and vertex $f$ is vertex $b$; that of $e$ and $c$ is $a$; that of $a$ and $j$ is $a$.}
		\label{fig:punc-1}
	\end{subfigure}
	\hfill
	\begin{subfigure}[t]{0.3\textwidth}
		\centering
		\begin{tikzpicture}[scale=1.0, every node/.style={scale=0.7}]
			\draw (0,0.2) node(R) {$\overset{3}{s}$};
			\draw (-1.2,-0.6) node(1) {$\overset{3}{\circ}$};
			\draw (0,-0.6) node(2) {$\overset{0}{\circ}$};
			\draw (1.2,-0.6) node(3) {$\overset{3}{z}$};
			\draw (R) -- (1);
			\draw (R) -- (2);
			\draw (R) -- (3);
			\draw (-1.5,-1.4) node(11) {$\overset{2}{r}$};
			\draw (-1.2,-1.4) node(12) {$\overset{1}{\bullet}$};
			\draw (-0.9,-1.4) node(13) {$\overset{2}{\bullet}$};
			\draw (1) -- (11);
			\draw (1) -- (12);
			\draw (1) -- (13);
			\draw (-0.3,-1.4) node(21) {$\overset{1}{\bullet}$};
			\draw (0,-1.4) node(22) {$\overset{1}{\bullet}$};
			\draw (0.3,-1.4) node(23) {$\overset{1}{\bullet}$};
			\draw (2) -- (21);
			\draw (2) -- (22);
			\draw (2) -- (23);
			\draw (0.9,-1.4) node(31) {$\overset{2}{x}$};
			\draw (1.2,-1.4) node(32) {$\overset{1}{y}$};
			\draw (1.5,-1.4) node(33) {$\overset{0}{\circ}$};
			\draw (3) -- (31);
			\draw (3) -- (32);
			\draw (3) -- (33);
		\end{tikzpicture}
		\caption{This tree has value 1, but condition (b) of Theorem~\ref{thm:punctures} is not satisfied. Observe $x$ and $y$ are people but $z$ has label not 0. Similarly, $r$ and $x$ are people but $s$ has label not 0.}
		\label{fig:punc-2}
	\end{subfigure}
	\hfill
	\begin{subfigure}[t]{0.3\textwidth}
		\centering
		\begin{tikzpicture}[scale=1.0, every node/.style={scale=0.7}]
			\draw (0,0.2) node(R) {$\overset{3}{\circ}$};
			\draw (-1.2,-0.6) node(1) {$\overset{0}{\circ}$};
			\draw (0,-0.6) node(2) {$\overset{0}{\circ}$};
			\draw (1.2,-0.6) node(3) {$\overset{0}{\circ}$};
			\draw (R) -- (1);
			\draw (R) -- (2);
			\draw (R) -- (3);
			\draw (-1.5,-1.4) node(11) {$\overset{0}{\circ}$};
			\draw (-1.2,-1.4) node(12) {$\overset{0}{\circ}$};
			\draw (-0.9,-1.4) node(13) {$\overset{0}{\circ}$};
			\draw (1) -- (11);
			\draw (1) -- (12);
			\draw (1) -- (13);
			\draw (-0.3,-1.4) node(21) {$\overset{0}{\circ}$};
			\draw (0,-1.4) node(22) {$\overset{0}{\circ}$};
			\draw (0.3,-1.4) node(23) {$\overset{0}{\circ}$};
			\draw (2) -- (21);
			\draw (2) -- (22);
			\draw (2) -- (23);
			\draw (0.9,-1.4) node(31) {$\overset{1}{\bullet}$};
			\draw (1.2,-1.4) node(32) {$\overset{1}{\bullet}$};
			\draw (1.5,-1.4) node(33) {$\overset{1}{\bullet}$};
			\draw (3) -- (31);
			\draw (3) -- (32);
			\draw (3) -- (33);
		\end{tikzpicture}
		\caption{This tree has value 1, and condition (b) of Theorem~\ref{thm:punctures} is satisfied.}
		\label{fig:punc-3}
	\end{subfigure}
	\caption{Examples of admissible trees and condition (b) from Theorem~\ref{thm:punctures}. Recall that a vertex $x$ is a person if and only if $\fv(x)=1$, i.e.~the label of $x$ belong to $[1,p-1]$.
	In the first subfigure we have omitted all vertex labels.}
	\label{fig:puncture}
\end{figure}
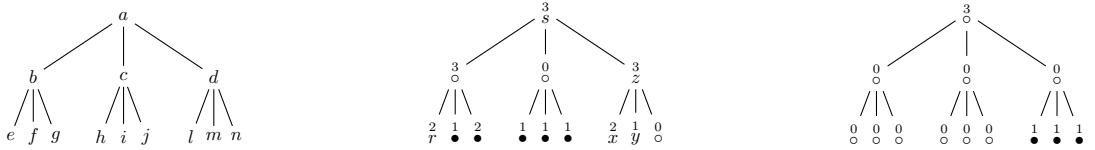

\begin{remark}
	We observe that if an admissible tree $T$ satisfies both conditions (a) and (b) of Theorem~\ref{thm:punctures}, then all people in $T$ are equidistant to the root vertex of $T$, and all people in $T$ have the same label. Moreover, the number of people in $T$ is a power of $p$, i.e.~$\eta(T)=p^i$ for some $i\in\N$, and also the root of $T$ is not a person (that is, has label $0$ or $p$). These properties can be easily verified for the trees in Figure~\ref{fig:example-punc}, which lists those $\theta\in\Irr(P_{p^3})$ for which $\Omega(\theta)$ is punctured.
	\newqed
\end{remark}

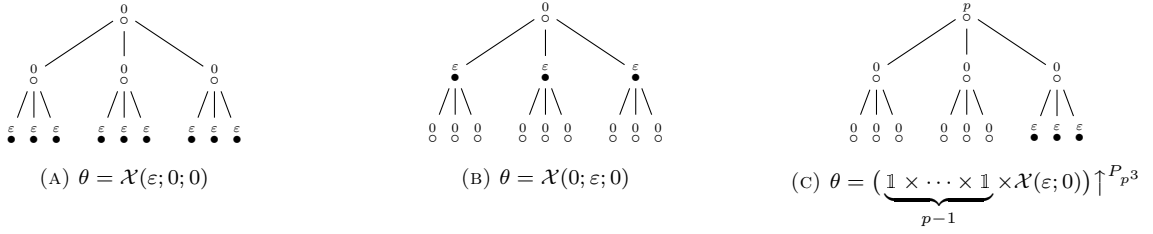
\begin{figure}[h!]
	\centering
	\begin{subfigure}[t]{0.3\textwidth}
		\centering
		\begin{tikzpicture}[scale=1.0, every node/.style={scale=0.7}]
			\draw (0,0.2) node(R) {$\overset{0}{\circ}$};
			\draw (-1.2,-0.6) node(1) {$\overset{0}{\circ}$};
			\draw (0,-0.6) node(2) {$\overset{0}{\circ}$};
			\draw (1.2,-0.6) node(3) {$\overset{0}{\circ}$};
			\draw (R) -- (1);
			\draw (R) -- (2);
			\draw (R) -- (3);
			\draw (-1.5,-1.4) node(11) {$\overset{\varepsilon}{\bullet}$};
			\draw (-1.2,-1.4) node(12) {$\overset{\varepsilon}{\bullet}$};
			\draw (-0.9,-1.4) node(13) {$\overset{\varepsilon}{\bullet}$};
			\draw (1) -- (11);
			\draw (1) -- (12);
			\draw (1) -- (13);
			\draw (-0.3,-1.4) node(21) {$\overset{\varepsilon}{\bullet}$};
			\draw (0,-1.4) node(22) {$\overset{\varepsilon}{\bullet}$};
			\draw (0.3,-1.4) node(23) {$\overset{\varepsilon}{\bullet}$};
			\draw (2) -- (21);
			\draw (2) -- (22);
			\draw (2) -- (23);
			\draw (0.9,-1.4) node(31) {$\overset{\varepsilon}{\bullet}$};
			\draw (1.2,-1.4) node(32) {$\overset{\varepsilon}{\bullet}$};
			\draw (1.5,-1.4) node(33) {$\overset{\varepsilon}{\bullet}$};
			\draw (3) -- (31);
			\draw (3) -- (32);
			\draw (3) -- (33);
		\end{tikzpicture}
		\caption{$\theta=\cX(\varepsilon;0;0)$}
		\label{fig:punc-i}
	\end{subfigure}
	\hfill
	\begin{subfigure}[t]{0.3\textwidth}
		\centering
		\begin{tikzpicture}[scale=1.0, every node/.style={scale=0.7}]
			\draw (0,0.2) node(R) {$\overset{0}{\circ}$};
			\draw (-1.2,-0.6) node(1) {$\overset{\varepsilon}{\bullet}$};
			\draw (0,-0.6) node(2) {$\overset{\varepsilon}{\bullet}$};
			\draw (1.2,-0.6) node(3) {$\overset{\varepsilon}{\bullet}$};
			\draw (R) -- (1);
			\draw (R) -- (2);
			\draw (R) -- (3);
			\draw (-1.5,-1.4) node(11) {$\overset{0}{\circ}$};
			\draw (-1.2,-1.4) node(12) {$\overset{0}{\circ}$};
			\draw (-0.9,-1.4) node(13) {$\overset{0}{\circ}$};
			\draw (1) -- (11);
			\draw (1) -- (12);
			\draw (1) -- (13);
			\draw (-0.3,-1.4) node(21) {$\overset{0}{\circ}$};
			\draw (0,-1.4) node(22) {$\overset{0}{\circ}$};
			\draw (0.3,-1.4) node(23) {$\overset{0}{\circ}$};
			\draw (2) -- (21);
			\draw (2) -- (22);
			\draw (2) -- (23);
			\draw (0.9,-1.4) node(31) {$\overset{0}{\circ}$};
			\draw (1.2,-1.4) node(32) {$\overset{0}{\circ}$};
			\draw (1.5,-1.4) node(33) {$\overset{0}{\circ}$};
			\draw (3) -- (31);
			\draw (3) -- (32);
			\draw (3) -- (33);
		\end{tikzpicture}
		\caption{$\theta=\cX(0;\varepsilon;0)$}
		\label{fig:punc-ii}
	\end{subfigure}
	\hfill
	\begin{subfigure}[t]{0.3\textwidth}
		\centering
		\begin{tikzpicture}[scale=1.0, every node/.style={scale=0.7}]
			\draw (0,0.2) node(R) {$\overset{p}{\circ}$};
			\draw (-1.2,-0.6) node(1) {$\overset{0}{\circ}$};
			\draw (0,-0.6) node(2) {$\overset{0}{\circ}$};
			\draw (1.2,-0.6) node(3) {$\overset{0}{\circ}$};
			\draw (R) -- (1);
			\draw (R) -- (2);
			\draw (R) -- (3);
			\draw (-1.5,-1.4) node(11) {$\overset{0}{\circ}$};
			\draw (-1.2,-1.4) node(12) {$\overset{0}{\circ}$};
			\draw (-0.9,-1.4) node(13) {$\overset{0}{\circ}$};
			\draw (1) -- (11);
			\draw (1) -- (12);
			\draw (1) -- (13);
			\draw (-0.3,-1.4) node(21) {$\overset{0}{\circ}$};
			\draw (0,-1.4) node(22) {$\overset{0}{\circ}$};
			\draw (0.3,-1.4) node(23) {$\overset{0}{\circ}$};
			\draw (2) -- (21);
			\draw (2) -- (22);
			\draw (2) -- (23);
			\draw (0.9,-1.4) node(31) {$\overset{\varepsilon}{\bullet}$};
			\draw (1.2,-1.4) node(32) {$\overset{\varepsilon}{\bullet}$};
			\draw (1.5,-1.4) node(33) {$\overset{\varepsilon}{\bullet}$};
			\draw (3) -- (31);
			\draw (3) -- (32);
			\draw (3) -- (33);
		\end{tikzpicture}
		\caption{$\theta= \big( \underbrace{\triv\times\cdots\times\triv}_{p-1} \times\cX(\varepsilon;0) \big) \up^{P_{p^3}}$}
		\label{fig:punc-iii}
	\end{subfigure}
	\caption{Those $\theta\in\Irr(P_{p^3})$ for which $\Omega(\theta)$ is punctured, drawn using 3-ary trees. Here, $\varepsilon\in[1,p-1]$. We write $\cX(a;b;c)$ for $\cX( \cX(\phi_a;\phi_b); \phi_c)$ and $\triv=\triv_{p^2}$.}
	\label{fig:example-punc}
\end{figure}

With Theorems~\ref{thm:GL1}, \ref{thm:value-1} and~\ref{thm:punctures} we have given a complete description of the set $\Omega(\theta)$, for every $\theta\in\Irr(P_{p^k})$ with $\val(\theta)\leq 1$. 
Of course, these results thus also provide us with the value $m(\theta)$ for all such irreducible characters $\theta$, which we record in the following corollary. 

\begin{corollary}\label{cor:m-pk}
	Let $\theta\in\Irr(P_{p^k})$ be such that $\val(\theta)\leq 1$. Then 
	\[ m(\theta) = \begin{cases}
		p^k-\gamma_0(\theta)-2 & \text{ if}\ \val(\theta)=0,\\ 
		p^k-\gamma_0(\theta)-1 & \text{ if}\ \val(\theta)=1\ \text{and conditions (a) and (b) of Theorem~\ref{thm:punctures} hold},\\
		p^k-\gamma_0(\theta) & \text{otherwise}.
	\end{cases} \]
\end{corollary}

We are now ready to tackle the problem in general. More precisely, we will devote the rest of this section to the calculation of $m(\theta)$ for all of the remaining $\theta\in\Irr(P_{p^k})$. 

\begin{theorem}\label{thm:m-val-2}
	Let $k\in\N$ and $\theta\in\Irr(P_{p^k})$. Suppose $\val(\theta)\ge 2$. Then 
	\[ m(\theta)=p^k-\gamma_0(\theta)-\gamma_1(\theta). \]
\end{theorem}

We will obtain Theorem~\ref{thm:m-val-2} as a corollary of the following more informative results. We start with a proposition dealing with irreducible characters of $P_{p^k}$ whose value is exactly $2$.  

\begin{proposition}\label{prop:3}
	Let $k\in\N_{\ge 2}$ and $\varepsilon\in[1, p-1]$. Let $\varphi\in\Irr(P_{p^{k-1}})$ satisfy $\val(\varphi)=1$, and let $\theta=\cX(\varphi;\varepsilon)\in\Irr(P_{p^k})$. Then the following hold: 
	\begin{itemize}
		\item[(i)] $\Omega(\theta) \supseteq \cB_{p^k}(p^k-\gamma_0(\theta)-\gamma_1(\theta))$ and
		\item[(ii)] $\Omega(\theta) \setminus \cB_{p^k}(p^k-\gamma_0(\theta)-\gamma_1(\theta))$ contains no thin partitions.
		\item[(iii)] Moreover, $[\lambda\down_{P_{p^k}},\theta]\ge 2$ for all $\lambda\in \close{\{\tworow{p^k}{p^k-\gamma_0(\theta)-\gamma_1(\theta)}, \hook{p^k}{p^k-\gamma_0(\theta)-\gamma_1(\theta)}\}}$.
	\end{itemize}
\end{proposition}

\begin{proof}
	Recall Notation~\ref{not:subgps}.
	The assertions hold for $k=2$ from Lemma~\ref{lem:small-k}(b), so we may assume $k\ge 3$. Recall that from Theorem~\ref{thm:value-1} we know that $\Omega(\varphi)$ equals either $\cB_{p^{k-1}}(p^{k-1}-\gamma_0(\varphi))$ or $\puncb{p^{k-1}}{p^{k-1}-\gamma_0(\varphi)}$. From now on we let $t:=p^{k-1}-\gamma_0(\varphi)$ and $N:=p^k-\gamma_0(\theta)-\gamma_1(\theta)$. Let $T\in\cT(\theta)$.
	
	Observe that $\gamma_1(\theta)=1$ as the root of $T$ is the only element of $\Gamma_1(T)$. Similarly, we observe that $\gamma_0(\theta)=p\cdot\gamma_0(\varphi)$. Hence $N=pt-1$.
	
	Now, we claim that $\cB_{p^k}(N)\subseteq\cM(p,\Omega(\varphi))$. 
	To prove this, we observe that if $\gamma_0(\varphi)\ge p$ then $\cB_{p^k}(pt-1) \subseteq\cM(p,\puncb{p^{k-1}}{t})$ by Lemma~\ref{lem:A'} with $q=p$, $m=p^{k-1}$ and $t=p^{k-1}-\gamma_0(\varphi)$. Hence $\cB_{p^k}(N)\subseteq\cM(p,\puncb{p^{k-1}}{t}) \subseteq \cM(p,\cB_{p^{k-1}}(t)) = \cM(p,\Omega(\varphi))$.
	On the other hand, if $\gamma_0(\varphi)<p$ then $\Omega(\varphi)=\cB_{p^{k-1}}(t)$ by Lemma~\ref{lem:small-gamma_0} (since $\val(\varphi)=1$ implies $\gamma_0(\theta)\ne 0$). Hence the above claim now follows from Proposition~\ref{prop:mixed}. 
	Since $\cM(p,\Omega(\varphi))\subseteq\Omega(\theta)$ by Lemma~\ref{lem:mixed}, we can now conclude that (i) holds as $\cB_{p^k}(N)\subseteq\cM(p,\Omega(\varphi))\subseteq\Omega(\theta)$.
	
	Next, we recall that $\Omega(\theta)\subseteq\cB_{p^k}(p^k-\gamma_0(\theta))=\cB_{p^k}(N+1)$ by Theorem~\ref{thm:M}. Using this and Lemma~\ref{lem:conj-restr}, to prove (ii) it therefore suffices to show that $\lambda\notin\Omega(\theta)$ for all $\lambda\in \{ \ft_{p^k}(N+1), \fh_{p^k}(N+1) \}$. Let $\lambda=\ft_{p^k}(N+1)$. 
	Setting $\alpha:=\ft_{p^{k-1}}(p^{k-1}-\gamma_0(\varphi))$, we observe that $\lambda=p\alpha$ and we use Lemma~\ref{lem:usual-PW/GTT+LR} (the hypotheses of which are satisfied by Theorem~\ref{thm:value-1}) to conclude that $[\lambda\down_P,\theta] = [\nu\down_{P_p},\phi_\varepsilon]=0$ for some $\nu\in\{(p),(1^p)\}$. If $\lambda=\fh_{p^k}(N+1)$ then $[\lambda\down_P,\theta]=0$ by the same argument, replacing $\alpha$ by $\fh_{p^{k-1}}(p^{k-1}-\gamma_0(\varphi))$.
	
	Finally, to prove (iii) it suffices to show that $[\lambda\down_P,\theta]\ge 2$ for $\lambda\in \{ \ft_{p^k}(N), \fh_{p^k}(N)\}$. Let $\lambda=\ft_{p^k}(N)$.
	Let $\mu=\ft_{p^{k-1}}(t)$ and $\nu=\ft_{p^{k-1}}(t-1)$. From Theorem~\ref{thm:value-1}, we know that $\mu,\nu\in\Omega(\varphi)$. 
	Let $\psi=\mu^{\times (p-1)}\times\nu\in\Irr(Y)$. Observe that $[\lambda\down_W,\psi\up^W]\neq 0$ because $[\lambda\down_Y,\psi]\neq 0$ and $\Irr(W\mid\psi)=\{\psi\up^W\}$. Moreover, using Lemma~\ref{lem:mackey} we have that $\psi\up^W\down_P=\psi\down_{P\cap Y}\up^{P}=\psi\down_{B}\up^{P}$. Since $\varphi^{\times p}$ is a constituent of $\psi\down_{B}$ and since $\theta$ is an irreducible constituent of $(\varphi^{\times p})\up^P$, we conclude that
	$[\psi\up^W\down_P, \theta]\geq 1$.
	On the other hand, using Theorem~\ref{thm:PW9.1} we deduce that there exists $\omega\in\close{\{(p-1,1)\}}$ such that $[\lambda\down_V,\cX(\mu;\omega)]=1$. Moreover, since $\mu\in\Omega(\varphi)$ and $\omega\in\Omega(\varepsilon)$ we use Lemmas~\ref{lem:SL2.18} and~\ref{lem:SL2.19} to obtain that $[\cX(\mu;\omega)\down_P,\theta]\ge 1$. Finally we observe that $[\cX(\mu;\omega)\down_W,\psi\up^W]=0$. This holds because 
	\[ [\cX(\mu;\omega)\down_W,\psi\up^W] \le [\cX(\mu;\omega)\down_Y, (\psi\up^W)\down_Y] = (p-1)\cdot [\mu^{\times p}, (\psi\up^W)\down_Y]=0. \]
	It follows that 
	\[ [\lambda\down_P,\theta]=  [(\lambda\down_W)\down_P,\theta]\ge [\cX(\mu;\omega)\down_P,\theta] + [(\psi\up^W)\down_P,\theta]\geq 2. \]
	If $\lambda=\hook{p^k}{N}$ then a similar argument, except where we use Theorem~\ref{thm:GTT3.5} to deduce that there exists $\omega\in\close{\{(p-1,1)\}}$ such that $[\lambda\down_V,\cX(\hook{p^{k-1}}{t};\omega)]=1$, shows that $[\lambda\down_P,\theta]\ge 2$ in this case also.
	This concludes the proof.
\end{proof}

We are now ready to extend Proposition~\ref{prop:3} to all $\theta\in\Irr(P_{p^k})$ with $\val(\theta)\ge 2$.

\begin{theorem}\label{thm:value-2}
	Let $k\in\N$ and $\theta\in\Irr(P_{p^k})$. Suppose $\val(\theta)\ge 2$. Then 
	\begin{itemize}
		\item[(i)] $\Omega(\theta) \supseteq \cB_{p^k}(p^k-\gamma_0(\theta)-\gamma_1(\theta))$, and
		\item[(ii)] $\Omega(\theta) \setminus \cB_{p^k}(p^k-\gamma_0(\theta)-\gamma_1(\theta))$ contains no thin partitions.
		\item[(iii)] Moreover, $[\lambda\down_{P_{p^k}},\theta]\ge 2$ for all $\lambda\in \close{\{\tworow{p^k}{p^k-\gamma_0(\theta)-\gamma_1(\theta)}, \hook{p^k}{p^k-\gamma_0(\theta)-\gamma_1(\theta)} \}}$.
	\end{itemize}
\end{theorem}

\begin{proof}
	We proceed by induction on $k$, noting that $\val(\theta)\ge 2$ implies $k\ge 2$, and the case of $k=2$ is given by Lemma~\ref{lem:small-k}. Now we may assume $k\ge 3$. Recall Notation~\ref{not:subgps} and let $\overline{\theta}:=\theta_1\times\cdots\times\theta_p$ be any irreducible constituent of $\theta\down_B$, where $\theta_i\in\Irr(Q)$. 
	
	There are two forms that $\theta$ may take: either $\theta=\cX(\varphi;\varepsilon)$ for some $\varepsilon\in[0, p-1]$ and $\varphi\in\Irr(Q)$, or $\theta=\theta_1\times\cdots\times\theta_p\up^P$ for some mixed $\theta_1,\ldots, \theta_p\in\Irr(Q)$. In the former case, we may further assume that $\val(\varphi)\ge 2$ since the case of $\val(\varphi)=1$ follows from Proposition~\ref{prop:3}. Then, in both cases, we have that $\gamma_0(\theta)+\gamma_1(\theta)=\sum_{i=1}^p (\gamma_0(\theta_i)+\gamma_1(\theta_i))$.
	
	\smallskip
	
	We first prove that (ii) holds. 
	Let $t\in\N_0$ be such that $\lambda:=\tworow{p^k}{p^k-t}\in\Omega(\theta)$. Since $\overline{\theta}$ is an irreducible constituent of $\theta\down_B$ and $\theta$ is a constituent of $\lambda\down_P$, then $\overline{\theta}\mid\lambda\down_B$. Thus there exist $\lambda^1,\ldots,\lambda^p\in \cP(p^{k-1})$ such that $\lambda^i\in\Omega(\theta_i)$ for all $i\in[1,p]$ and $\overline{\lambda}:=\lambda^1\times\cdots\times\lambda^p\in\Irr(Y\mid\overline{\theta})$ is an irreducible constituent of $\lambda\down_Y$. Then $\ell(\lambda^i)\le \ell(\lambda)\le 2$ for all $i$ by Lemma~\ref{lem:LR-first-part}.
	Using the inductive hypothesis together with Theorems~\ref{thm:GL1} and~\ref{thm:value-1}, we see that
	\begin{equation}\label{eqn:A(theta)}
		\Omega(\theta_i) = \begin{cases}
			\puncp{p^{k-1}} & \text{if }\val(\theta_i)=0,\\
			\puncb{p^{k-1}}{p^{k-1}-\gamma_0(\theta_i)} \text{ or }	\cB_{p^{k-1}}(p^{k-1}-\gamma_0(\theta_i)) & \text{if }\val(\theta_i)=1,\\
			\cB_{p^{k-1}}(p^{k-1}-\gamma_0(\theta_i)-\gamma_1(\theta_i))\sqcup\Gamma(\theta_i) & \text{if }\val(\theta_i)\ge 2,
		\end{cases}
	\end{equation}
	for some subset $\Gamma(\theta_i)$ of $\cP(p^{k-1})\setminus\cB_{p^{k-1}}(p^{k-1}-\gamma_0(\theta_i)-\gamma_1(\theta_i))$ containing no thin partitions. 
	Thus Lemma~\ref{lem:LR-first-part} implies that 
	\[ p^k-t = \lambda_1\le \sum_{i=1}^p (\lambda^i)_1 \le \sum_{i=1}^p (p^{k-1}-\gamma_0(\theta_i)-\gamma_1(\theta_i)), \]
	where we used that $\gamma_1(\theta_i)=0$ if $\val(\theta_i)=1$, and that $\gamma_0(\theta_i)=\gamma_1(\theta_i)=0$ if $\val(\theta_i)=0$. 
	We deduce that $t\geq \sum_{i=1}^p (\gamma_0(\theta_i)+\gamma_1(\theta_i)) = \gamma_0(\theta)+\gamma_1(\theta)$ and therefore that $\lambda_1=p^k-t\leq p^k- \gamma_0(\theta)-\gamma_1(\theta)$. This implies that $\lambda\in \cB_{p^k}(p^k- \gamma_0(\theta)-\gamma_1(\theta))$. With a completely analogous argument, we see that if $\lambda=\hook{p^k}{p^k-t}$ then $\lambda\in\cB_{p^k}(p^k-\gamma_0(\theta)-\gamma_1(\theta))$ also.	
	
	\smallskip
	
	Next, we turn to the proofs of (i) and (iii). We split our argument into two cases depending on the form of $\theta$.
	
	In the first case, suppose that $\theta_1=\cdots=\theta_p=:\varphi$. That is, $\theta=\cX(\varphi;\varepsilon)$ for some $\varepsilon\in[0, p-1]$ and $\varphi\in\Irr(Q)$ with $\val(\varphi)\ge 2$. 
	Let $\zeta:=\gamma_0(\varphi)+\gamma_1(\varphi)$. Then by the inductive hypothesis, $\Omega(\varphi)$ contains $\cB_{p^{k-1}}(p^{k-1}-\zeta)$ and no other thin partitions. Also we have that $[\alpha\down_Q,\varphi]\ge 2$ for all $\alpha\in \close{\{\tworow{p^{k-1}}{p^{k-1}-\zeta}, \hook{p^{k-1}}{p^{k-1}-\zeta}\}}$. 
	Moreover, note that $\gamma_0(\theta)+\gamma_1(\theta)=p\cdot\zeta$ since $\val(\varphi)\ge 2$.
	
	Now, Proposition~\ref{prop:mixed} and Lemma~\ref{lem:mixed} guarantee that
	\[ \cB_{p^k}(p^{k}-\gamma_0(\theta)-\gamma_1(\theta)-1)=\cB_{p^k}(p\cdot(p^{k-1}-\zeta)-1) \subseteq \cM(p,\cB_{p^{k-1}}(p^{k-1}-\zeta)) \subseteq \cM(p,\Omega(\varphi)) \subseteq\Omega(\theta). \]
	Let $\lambda=(p^k-p\zeta,\mu)$ for some $\mu\in\cP(p\zeta)$. 
	If $\mu\notin \close{\{(p\zeta)\}}$, then $\mu\in\cB_{p\zeta}(p\zeta-1)=\cM(p,\cP(\zeta))$ by Proposition~\ref{prop:mixed}, since $\val(\varphi)\ge 2$ certainly implies that $\zeta\ge 2$. 
	Hence, there exist mixed $\mu^1,\ldots, \mu^p\in\cP(\zeta)$ such that $\LR(\mu;\mu^1,\ldots,\mu^p)\neq 0$. Therefore, setting $\lambda^i:=(p^{k-1}-\zeta,\mu^i)$ for each $i\in [1,p]$ we have that $\LR(\lambda;\lambda^1,\ldots,\lambda^p)\neq 0$ by Lemma~\ref{lem:iterated-LR}. Since $\lambda^1,\ldots, \lambda^p$ are mixed and all belong to $\Omega(\varphi)$, it follows that $\lambda\in\Omega(\theta)$ by Lemma~\ref{lem:mixed}.
	If $\mu=(p\zeta)$, then $[\lambda\down_V,\cX(\tworow{p^{k-1}}{p^{k-1}-\zeta};(p))]=1$ by Theorem~\ref{thm:PW9.1}. Hence
	\[ [\lambda\down_P,\theta]\ge[\cX(\tworow{p^{k-1}}{p^{k-1}-\zeta};(p))\down_P,\theta] \ge 2, \]
	where the final inequality follows from Lemma~\ref{lem:SL2.18} since $[\tworow{p^{k-1}}{p^{k-1}-\zeta}\down_Q,\varphi]\ge 2$ by the inductive hypothesis. Finally, if $\mu=(1^{p\zeta})$, then $[\lambda\down_V,\cX(\hook{p^{k-1}}{p^{k-1}-\zeta};\nu)]=1$ for some $\nu\in\{(p),(1^p)\}$ by Theorem~\ref{thm:GTT3.5}. Then similarly $[\lambda\down_P,\theta]\ge 2$ by Lemma~\ref{lem:SL2.18} and the inductive hypothesis.
	Thus we have proven (i) and (iii) in this case.
	
	Let us now consider the second case, where $\theta_1,\ldots,\theta_p\in\Irr(Q)$ are mixed and $\theta=\overline{\theta}\up^P$. 
	By Corollary~\ref{cor:omega-mixed}, $\Omega(\theta)=\Omega(\theta_1)\star\cdots\star\Omega(\theta_p)$. Recall the possible forms of $\Omega(\theta_i)$ from \eqref{eqn:A(theta)}, and notice that if $\Omega(\theta_i)=\puncb{p^{k-1}}{p^{k-1}-\gamma_0(\theta_i)}$ then $\gamma_0(\theta_i)\ge p$ from  Lemma~\ref{lem:small-gamma_0} (as $\theta_i\ne\triv_Q$ and so $\gamma_0(\theta)\ne 0$). Moreover, since $\val(\theta)\geq 2$ then there exists $j\in [1,p]$ such that $\val(\theta_j)\geq 2$. Hence the inductive hypothesis implies that $\cB_{p^{k-1}}(p^{k-1}-\gamma_0(\theta_j)-\gamma_1(\theta_j))\subseteq \Omega(\theta_j)$. 
	With this in mind, using Theorem~\ref{prop:smooth} we have that
	\[ \Omega(\theta_1)\star\cdots\star\Omega(\theta_p) \supseteq \cB_{p^k}\left(\sum_{i=1}^p (p^{k-1}-\gamma_0(\theta_i)-\gamma_1(\theta_i)) \right), \]
	whence $\Omega(\theta) \supseteq \cB_{p^k}(p^k-\gamma_0(\theta)-\gamma_1(\theta))$ by Corollary~\ref{cor:omega-mixed}.
	
	Finally, let $\lambda=\tworow{p^k}{p^k-\gamma_0(\theta)-\gamma_1(\theta)}$ and $\lambda^i=\tworow{p^{k-1}}{p^{k-1}-\gamma_0(\theta_i)-\gamma_1(\theta_i)}$ for each $i\in [1,p]$. Write $\overline{\lambda}:=\lambda^1\times\cdots\times\lambda^p$ and observe that $\overline{\lambda}\mid\lambda\down_Y$. Since $\theta=\overline{\theta}\up^P$, we must have that $\val(\theta_j)=\val(\theta)\ge 2$ for some $j\in[1, p]$. Since $\Irr(P\mid\overline{\theta})=\{\theta\}$ and $[\theta\down_B,\overline{\theta}]=1$, we have that
	\[ [\lambda\down_P,\theta] = [\lambda\down_B,\overline{\theta}] \ge [\overline{\lambda}\down_B,\overline{\theta}] = \prod_{i=1}^p [\lambda^i\down_Q,\theta_i] \ge 2\cdot 1^{p-1}, \]
	where the last inequality follows by observing that $\lambda^i\in\Omega(\theta_i)$ for all $i\in [1,p]$ and that $[\lambda^j\down_Q,\theta_j]\ge 2$ by the inductive hypothesis. The case $\lambda=\hook{p^k}{p^k-\gamma_0(\theta)-\gamma_1(\theta)}$ is treated similarly. 
\end{proof}

\begin{corollary}
	Theorem~\ref{thm:m-val-2} holds. 
\end{corollary}

\begin{proof}
	That $m(\theta)=p^k-\gamma_0(\theta)-\gamma_1(\theta)$ follows directly from Theorem~\ref{thm:value-2}(i) and (ii).
\end{proof}

We can put our results on $m(\theta)$ together in the following theorem. 

\begin{theorem}\label{thm:m-final-pk}
	Let $k\in\N$ and $\theta\in\Irr(P_{p^k})$. Then 
	\[ m(\theta) = \begin{cases}
		p^k-\gamma_0(\theta)-\gamma_1(\theta)-2 & \text{ if}\ \val(\theta)=0,\\
		p^k-\gamma_0(\theta)-\gamma_1(\theta)-1 & \text{ if}\ \val(\theta)=1\ \text{and conditions (a) and (b) of Theorem~\ref{thm:punctures} hold},\\
		p^k-\gamma_0(\theta)-\gamma_1(\theta) & \text{otherwise}.
	\end{cases} \]
\end{theorem}

\begin{proof}
	This follows from Corollary~\ref{cor:m-pk} and Theorem~\ref{thm:m-val-2}, observing that when $\val(\theta)\leq 1$ then $\gamma_1(\theta)=0$. We remark that if $\val(\theta)=0$ then $\gamma_0(\theta)=\gamma_1(\theta)=0$.
\end{proof}

\bigskip

\section{The Structure of the Set $\Omega(\theta)$: Arbitrary $n$ Case}\label{sec:arbitrary-n}

Again, throughout this section $p$ will be a fixed prime number such that $p\ge 5$.
Using our results from the prime power case, in this section we now allow $n$ to be any positive integer and determine $m(\theta)$ and $M(\theta)$ for any $\theta\in\Irr(P_n)$. Exactly as before, we start with $M(\theta)$. 

\begin{lemma}\label{lem:omega-general-n}
	Let $n\in\N$ with $p$-adic expansion $n=p^{n_1}+\cdots+p^{n_t}$. Let $\theta\in\Irr(P_n)$ and suppose $\theta=\theta_1\times\cdots\times\theta_t$ where $\theta_i\in\Irr(P_{p^{n_i}})$. Then $\Omega(\theta)=\Omega(\theta_1)\star\cdots\star\Omega(\theta_t)$. 
\end{lemma}

\begin{proof}
	Since $P_n=P_{p^{n_1}}\times\cdots\times P_{p^{n_t}}\le S_{p^{n_1}}\times\cdots\times S_{p^{n_t}}\leq S_n$, the statement is an immediate consequence of the basic properties of induction of characters. 
\end{proof}

Before moving forward, we suggest to the reader to briefly recall Definition~\ref{def:trees-arbitrary}. 

\begin{theorem}\label{thm:arbitrary-M}
	For all $n\in\N$ and $\theta\in\Irr(P_n)$, we have that $M(\theta) = n-\gamma_0(\theta)$.
\end{theorem}

\begin{proof}
	Let $n=p^{n_1}+\cdots+p^{n_t}$ be the $p$-adic expansion of $n$. 
	If $t=1$ then the statement holds by Theorem~\ref{thm:M}. Suppose now that $t\geq 2$. For each $i\in [1,t]$, let $\theta_i\in\Irr(P_{p^{n_i}})$ be such that $\theta=\theta_1\times\theta_2\times\cdots\times\theta_t$. 
	Moreover, let $\Sigma:=M(\theta_1)+M(\theta_2)+\cdots + M(\theta_t)$. Since $\gamma_0(\theta_i)\leq p^{n_i-1}<\frac{p^{n_i}}{2}$ we can use Proposition~\ref{prop:smooth} together with Lemma~\ref{lem:omega-general-n} to deduce that 
	\[ \Omega(\theta)=\Omega(\theta_1)\star\cdots\star\Omega(\theta_t)\subseteq \cB_{p^{n_1}}(M(\theta_1))\star\cdots\star \cB_{p^{n_t}}(M(\theta_t))=\cB_{n}(\Sigma). \]
	This implies that $M(\theta)\leq \Sigma$. 
	
	On the other hand, for every $i\in [1,t]$ let $\lambda^{i}\in\Omega(\theta_i)$ be such that $(\lambda^{i})_1 = M(\theta_i)$.  
	Setting $\lambda=\lambda^{1}+\cdots+\lambda^{t}$, we know from Lemma~\ref{lem:LR-sum} that $\LR(\lambda; \lambda^{1},\dotsc,\lambda^{t})=1$. Hence $\lambda\in\Omega(\theta_1)\star\cdots\star\Omega(\theta_t)=\Omega(\theta)$, and $\lambda_1=\sum_{i=1}^t (\lambda^{i})_1 = \Sigma$. Hence, we deduce that $\Sigma\leq M(\theta)$ and therefore that $\Sigma=M(\theta)$. 
	
	From Definition~\ref{def:trees-arbitrary}, we know that $\gamma_0(\theta)=\gamma_0(\theta_1)+\cdots+\gamma_0(\theta_t)$. Thus $M(\theta)=n-\gamma_0(\theta)$.
\end{proof}

We are now ready to extend Theorem~\ref{thm:value-1} from $n=p^k$ to any arbitrary $n\in\N$, by describing the structure of $\Omega(\theta)$ for any $\theta\in\Irr(P_n)$ with $\val(\theta)\leq 1$. This is done in Theorems~\ref{thm:val-1-arbitrary} and~\ref{thm:puncture-arbitrary} below, and it is a broad generalization of \cite[Theorem 2.9]{GL2}. Recall from Definition~\ref{def:puncture} that $\puncb{n}{t}$ denotes a punctured box, which is $\cB_n(t)$ less two particular partitions and their conjugates.

\begin{theorem}\label{thm:val-1-arbitrary}
	Let $n\in\N$ and $\theta\in\Irr(P_n)$. Suppose $\val(\theta)=1$. Then
	\[ \Omega(\theta)\ \text{equals either}\ \cB_{n}(n-\gamma_0(\theta))\text{ or }\puncb{n}{n-\gamma_0(\theta)}. \]
\end{theorem}

\begin{proof}
	As usual, let $n=p^{n_1}+\cdots+p^{n_t}$ be the $p$-adic expansion of $n$. If $t=1$ then the statement holds by Theorem~\ref{thm:value-1}. Suppose now that $t\geq 2$. For each $i\in [1,t]$, let $\theta_i\in\Irr(P_{p^{n_i}})$ be such that $\theta=\theta_1\times\theta_2\times\cdots\times\theta_t$. Since $\val(\theta)=1$ we have that $\val(\theta_i)\leq 1$ for each $i\in [1,t]$. 
	The statement is now a direct consequence of Lemma~\ref{lem:omega-general-n}, Theorem~\ref{thm:value-1} and Proposition~\ref{prop:smooth}. 
\end{proof}

As done before in the prime power case, we are able to determine exactly when $\Omega(\theta)$ is a punctured box. Let $n\in\N$ and let $\theta\in\Irr(P_n)$ with $\val(\theta)=1$. We say that $\Omega(\theta)$ is \emph{punctured} if $\Omega(\theta)= \puncb{n}{n-\gamma_0(\theta)}$.

\begin{theorem}\label{thm:puncture-arbitrary}
	Let $n\in\N$ and let $n=p^{n_1}+\cdots+p^{n_t}$ be its $p$-adic expansion. Let $\theta\in\Irr(P_n)$ be such that $\val(\theta)=1$. Suppose $\theta=\theta_1\times\theta_2\times\cdots\times\theta_t$ where $\theta_i\in\Irr(P_{p^{n_i}})$ for each $i\in[1,t]$. Then $\Omega(\theta)$ is punctured if and only if there exists a unique $j\in [1,t]$ such that $\val(\theta_i)=0$ for all $i\in [1,t]\setminus\{j\}$ and such that $\theta_j$ satisfies the following conditions: 
	\begin{itemize}
		\item[(a)] $\eta(\theta_j)\ge 2$, and
		\item[(b)] for any $T\in\cT(\theta_j)$, the closest common graph-theoretic ancestor of any two people in $T$ has label 0.
	\end{itemize}
\end{theorem}

\begin{proof}
	Let us start by assuming the existence of a unique $j\in [1,t]$ such that $\val(\theta_i)=0$ for all $i\in [1,t]\setminus\{j\}$, and such that $\theta_j$ satisfies conditions (a) and (b). Then $\theta_i=\triv_{p^{n_i}}$ for all $i\in [1,t]\setminus\{j\}$, and hence $\Omega(\theta_i)=\puncp{p^{n_i}}$ for all such $i$ by Theorem~\ref{thm:GL1}. 
	Moreover, since $\val(\theta)=1$ we must have that $\val(\theta_j)=1$. Hence $\Omega(\theta_j)$ is punctured by Theorem~\ref{thm:punctures}; that is, $\Omega(\theta_j)=\puncb{n}{p^{n_j}-\gamma_0(\theta_j)}$ by Theorem~\ref{thm:value-1}. Now, using Lemma~\ref{lem:omega-general-n} and Proposition~\ref{prop:smooth} we obtain that 
	\[ \Omega(\theta)=\Omega(\theta_1)\star\cdots\star\Omega(\theta_t) 
	=\cP(n-p^{n_t})\star \puncb{p^{n_j}}{p^{n_j}-\gamma_0(\theta_j)} = \puncb{n}{n-\gamma_0(\theta)}. \]
	For the converse implication, suppose that $\Omega(\theta)$ is punctured; that is, $\Omega(\theta)=\puncb{n}{n-\gamma_0(\theta)}$.
	Since $\val(\theta)=1$, then $\val(\theta_i)\le 1$ for all $i\in[1,t]$. Thus Theorems~\ref{thm:GL1}, \ref{thm:value-1} and~\ref{thm:punctures} give a complete description of the possible forms of $\Omega(\theta_i)$.
	Now, by Lemma~\ref{lem:omega-general-n} we know that
	\[ \puncb{n}{n-\gamma_0(\theta)} = \Omega(\theta_1)\star\cdots\star\Omega(\theta_t). \]
	So by Proposition~\ref{prop:smooth} (it is helpful to refer to Table~\ref{tab:smooth} here), there exists a unique $j\in[1,t]$ such that $\Omega(\theta_j)=\puncb{p^{n_j}}{x}$ for some $x$, and $\Omega(\theta_i)=\puncp{p^{n_i}}$ for all $i\in[1,t]\setminus\{j\}$. 
	This implies that $\theta_i=\triv_{p^{n_i}}$ for all such $i$, and $\theta_j$ satisfies conditions (a) and (b), by Theorem~\ref{thm:punctures}.
\end{proof}

We remark that Theorems~\ref{thm:val-1-arbitrary} and~\ref{thm:puncture-arbitrary} provide a complete description of $\Omega(\theta)$ for all irreducible characters $\theta$ of value $\val(\theta)=1$. This generalises \cite[Theorem 2.9]{GL2} from linear to arbitrary irreducible characters of Sylow subgroups of symmetric groups, since a \textit{quasi-trivial} character of $P_n$ (in the sense of \cite[Definition 2.8]{GL2}) is precisely a linear character $\theta$ of $P_n$ such that $\val(\theta)=1$.

We now turn to the description of $\Omega(\theta)$ for all other irreducible characters $\theta$ of $P_n$. 

\begin{theorem}\label{thm:arbitrary-val-2}
	Let $n\in\N$ and $\theta\in\Irr(P_n)$ with $\val(\theta)\geq 2$. Then $\cB_n(n-\gamma_0(\theta)-\gamma_1(\theta))\subseteq \Omega(\theta)$ and 
	$\Omega(\theta)\setminus\cB(n-\gamma_0(\theta)-\gamma_1(\theta))$ contains no thin partitions. 
\end{theorem}

\begin{proof}
	Let $n=p^{n_1}+p^{n_2}+\cdots+p^{n_t}$ be the $p$-adic expansion of $n$. If $t=1$ then the statement holds by Theorem~\ref{thm:value-2}. Suppose now that $t\geq 2$ and that $\theta=\theta_1\times\theta_2\times\cdots\times\theta_t$, where $\theta_i\in\Irr(P_{p^{n_i}})$ for all $i\in [1,t]$. 
	From Definition~\ref{def:trees-arbitrary} we know that there exists $j\in [1,t]$ such that $\val(\theta_j)=\val(\theta)\geq 2$. Hence, by Theorem~\ref{thm:value-2} we have that $\cB_{p^{n_j}}(p^{n_j}-\gamma_0(\theta_j)-\gamma_1(\theta_j))\subseteq \Omega(\theta_j)$ and 
	$\Omega(\theta_j)\setminus\cB_{p^{n_j}}(p^{n_j}-\gamma_0(\theta_j)-\gamma_1(\theta_j))$ contains no thin partitions. 
	
	Let us now analyse the possible structures of $\Omega(\theta_i)$ for all the remaining $i\in [1,t]$. 
	If $\val(\theta_i)=0$, then $\theta_i=\triv_{p^{n_i}}$. In this case $\gamma_0(\theta_i)=\gamma_1(\theta_i)=0$ and by Theorem~\ref{thm:GL1} we have that $\Omega(\theta_i)=\puncp{p^{n_i}}=\puncp{p^{n_i}-\gamma_0(\theta_i)-\gamma_1(\theta_i)}$.
	If $\val(\theta_i)=1$ then $\gamma_1(\theta_i)=0$ and by Theorem~\ref{thm:value-1} we have that $\Omega(\theta_i)$ is equal to $\cB_{p^{n_i}}(p^{n_i}-\gamma_0(\theta_i)-\gamma_1(\theta_i))$ or to $\puncb{p^{n_i}}{p^{n_i}-\gamma_0(\theta_i)-\gamma_1(\theta_i)}$. 
	Finally, if $\val(\theta_i)\geq 2$ then we have that $\cB(p^{n_i}-\gamma_0(\theta_i)-\gamma_1(\theta_i))\subseteq \Omega(\theta_i)$ by Theorem~\ref{thm:value-2}. In particular, there is at least one such $i\in[1,t]$ with $\val(\theta_i)\ge 2$, namely $j$ above.
	Since $\gamma_0(\theta)+\gamma_1(\theta)=\sum_{y=1}^t(\gamma_0(\theta_y)+\gamma_1(\theta_y))$, using Proposition~\ref{prop:smooth} and Lemma~\ref{lem:omega-general-n} we deduce that 
	\[ \cB_n(n-\gamma_0(\theta)-\gamma_1(\theta))\subseteq \Omega(\theta_1)\star\Omega(\theta_2)\star\cdots\star\Omega(\theta_t)=\Omega(\theta). \]
	
	Let us now assume that $\lambda\in\Omega(\theta)$ is a thin partition. Without loss of generality we also assume that $\lambda_1\geq \ell(\lambda)$. 
	Since $\theta=\theta_1\times\cdots\times\theta_t$ is an irreducible constituent of $\lambda\down_{P_n}$, considering restriction to the intermediate subgroup $X:=S_{p^{n_1}}\times S_{p^{n_2}}\times\cdots\times S_{p^{n_t}}$, we see that there exist $\lambda^1,\lambda^2,\ldots, \lambda^t$ such that $\lambda^i\in\Omega(\theta_i)$ for every $i\in [1,t]$ and such that $\LR(\lambda; \lambda^{1}, \ldots, \lambda^{t})\neq 0$. By Lemma~\ref{lem:LR-first-part} this implies that $\lambda^i$ is a thin partition for every $i\in [1,t]$, and that $\lambda_1\leq (\lambda^1)_1+(\lambda^2)_1+\cdots+(\lambda^t)_1$. 
	Since each $\lambda^i$ is thin and contained in $\Omega(\theta_i)$, we have that $(\lambda^i)_1\leq p^{n_i}-\gamma_0(\theta_i)-\gamma_1(\theta_i)$ by Theorems~\ref{thm:GL1}, \ref{thm:value-1} and~\ref{thm:value-2}. We conclude that 
	\[ \lambda_1\leq \sum_{i=1}^t(\lambda^i)_1\leq \sum_{i=1}^t \big(p^{n_i}-\gamma_0(\theta_i)-\gamma_1(\theta_i)\big) = n-\gamma_0(\theta)-\gamma_1(\theta), \]
	as desired. 
\end{proof}

We can now put together our results to determine $m(\theta)$ for every $\theta\in\Irr(P_n)$. In order to state this compactly, we will say that $\theta\in\Irr(P_n)$ with $\val(\theta)=1$ is \emph{punctured} to mean that $\Omega(\theta)$ is punctured. More precisely, by Theorem~\ref{thm:puncture-arbitrary} this means the following.

\begin{definition}
	Let $n=p^{n_1}+p^{n_2}+\cdots+p^{n_t}$ be the $p$-adic expansion of $n$. Let $\theta=\theta_1\times\cdots\times\theta_t\in\Irr(P_n)$ with $\val(\theta)=1$. We say that $\theta$ is \emph{punctured} if there exists a unique $j\in [1,t]$ such that $\val(\theta_i)=0$ for all $i\in [1,t]\setminus\{j\}$ and such that $\theta_j$ satisfies the following conditions: 
	\begin{itemize}
		\item[(a)] $\eta(\theta_j)\ge 2$, and
		\item[(b)] for any $T\in\cT(\theta_j)$, the closest common graph-theoretic ancestor of any two people in $T$ has label 0.
	\end{itemize}
\end{definition}

The following theorem describes $m(\theta)$ for every irreducible character $\theta$ of $P_n$. In particular, it shows that in the great majority of the cases, $m(\theta)$ equals $n-\gamma_0(\theta)-\gamma_1(\theta)$.

\begin{theorem}\label{thm:final-m(theta)}
	Let $n\in\N$ and $\theta\in\Irr(P_{n})$. Then 
	\[ m(\theta) = \begin{cases}
			n-\gamma_0(\theta)-\gamma_1(\theta)-2 & \text{ if}\ n=p^k \text{ for some } k\in\N \text{ and } \val(\theta)=0,\\
			n-\gamma_0(\theta)-\gamma_1(\theta)-1 & \text{ if}\ \val(\theta)=1 \text{ and } \theta \text{ is punctured},\\
			n-\gamma_0(\theta)-\gamma_1(\theta), & \text{otherwise}.
		\end{cases} \]
\end{theorem}

\begin{proof}
	If $\val(\theta)=0$ then $\theta=\triv_n$ and $\gamma_0(\theta)=\gamma_1(\theta)=0$, and the result follows from Theorem~\ref{thm:GL1}. 
	If $\val(\theta)=1$ then $\gamma_1(\theta)=0$, and the result follows from Theorems~\ref{thm:val-1-arbitrary} and~\ref{thm:puncture-arbitrary}. 
	Finally, if $\val(\theta)\geq 2$ then the statement is an immediate consequence of Theorem~\ref{thm:arbitrary-val-2}. 
\end{proof}

\bigskip

\section{Applications \& future directions}\label{sec:future}

In this section we collect a few interesting and immediate consequences of the main results of this article. Fix a prime number $p\ge 5$.

We have now proven tight upper and lower bounds $M(\theta)$ and $m(\theta)$ respectively such that 
\[ \cB_n(m(\theta))\subseteq\Omega(\theta)\subseteq\cB_n(M(\theta)) \]
for all $n\in\N$ and $\theta\in\Irr(P_n)$, in Theorems~\ref{thm:arbitrary-M} and~\ref{thm:final-m(theta)}. As mentioned in the introduction, it is still unknown in general when a partition $\lambda$ belongs to $\Omega(\theta)$ if $\lambda\in \cB_n(M(\theta))\setminus\cB_n(m(\theta))$. However, it turns out that there are very few such $\lambda$: we describe the difference $M(\theta)-m(\theta)$ in Corollary~\ref{cor:M-m} below, from which we can deduce a crude but simple bound on this difference which does not depend on $\theta$. 

Then, in Corollary~\ref{cor:limit}, we show for all $\theta$ that, in fact, the set $\Omega(\theta)$ is a very large subset of the set of partitions of $n$. 
Moreover, we are able to determine exactly when $\lambda$ belongs to $\Omega(\theta)$ when $\lambda$ is a thin partition. That is, we completely determine the positivity of the Sylow branching coefficient $[\chi^\lambda\down_{P_n},\theta]$ in the case when $\lambda$ has a Durfee square of size at most two. This is done in Corollary~\ref{cor:thin} below.

\begin{corollary}\label{cor:M-m}
	Let $n\in\N$ and $\theta\in\Irr(P_n)$. Then $M(\theta)-m(\theta)=\gamma_1(\theta)+c$, where $c\in\{0,1,2\}$. In particular, $M(\theta)-m(\theta) \le \tfrac{n}{p^2}+2$ for all $\theta$.
\end{corollary}

\begin{proof}
	The first statement follows immediately from Theorems~\ref{thm:arbitrary-M} and~\ref{thm:final-m(theta)}.
	Since $\gamma_1(\theta)\le p^{k-2}=\tfrac{1}{p^2}\cdot n$ when $n=p^k$ by Lemma~\ref{lem:tree-bounds}, it follows that $\gamma_1(\theta)\le\tfrac{1}{p^2}\cdot n$ when $n\in\N$ is arbitrary, by Definition~\ref{def:trees-arbitrary}.
\end{proof}

Now, the quantity $\gamma_1(\theta)$ is in general very small compared with $n$ in an admissible (tuple of) trees. Although this may, in the worst case, seemingly be as large as $\tfrac{n}{p^2}$, in fact our results on $\Omega(\theta)$ show the following statement: that the restriction to $P_n$ of \emph{almost all} irreducible characters of $S_n$ have \emph{every} irreducible character of $P_n$ as a constituent. This will follow from the fact that $m(\theta)$ turns out to be extremely large compared with $n$. More precisely, let $\Omega(n)$ be the subset of $\cP(n)$ defined by 
\[ \Omega(n) := \{\lambda\in\cP(n) \mid [\lambda\down_{P_n}, \theta]\neq 0\ \text{for all}\ \theta\in\Irr(P_n)\}. \]
Of course, $\Omega(n)=\bigcap_{\theta\in\Irr(P_n)} \Omega(\theta)$. The following result shows that $\Omega(n)$ is almost the entirety of $\cP(n)$. 

\begin{corollary}\label{cor:limit}
	For all $n\in\N$ we have that $\cB_n\big(\tfrac{p-2}{p}\cdot n\big)\subseteq\Omega(n)$. In particular, $\lim_{n\to\infty} \tfrac{|\Omega(n)|}{|\cP(n)|}=1$.
\end{corollary}

\begin{proof}
	We first observe that if $n=p^0$ then $\Irr(P_n)=\{\triv\}$ and $m(\triv)=1=n$.
	If $n=p^1$ then $m(\theta)\in\{p-1,p-2\}$ for all $\theta\in\Irr(P_n)$ by Lemma~\ref{lem:small-k}.
	If $n=p^k$ where $k\ge 2$, then Theorem~\ref{thm:final-m(theta)} and Lemma~\ref{lem:tree-bounds}(a) give us 
	\[ m(\theta)\ge p^k-\gamma_0(\theta)-\gamma_1(\theta)-2 \ge p^k-p^{k-1}-p^{k-2}-2 \]
	for all $\theta\in\Irr(P_n)$.
	Therefore, whenever $n$ is a power of $p$, we have that $m(\theta)\ge\tfrac{p-2}{p}\cdot n$ for all $\theta\in\Irr(P_n)$. In particular, this means that $\cB_n\big(\tfrac{p-2}{p}\cdot n\big)\subseteq\Omega(n)$ whenever $n=p^k$ for some $k\in\N_0$.
	
	Now suppose $n\in\N$ with $p$-adic expansion $n=\sum_{i=1}^t p^{n_i}$, and let $\theta=\theta_1\times\cdots\times\theta_t\in\Irr(P_n)$. Then $\cB_{p^{n_i}}\big(\tfrac{p-2}{p}\cdot p^{n_i}\big)\subseteq\Omega(n) \subseteq\Omega(\theta_i)$ for each $i\in[1,t]$. 
	But $\tfrac{p-2}{p}>\tfrac{1}{2}$ since $p\ge 5$, and so by Proposition~\ref{prop:SL5.7} and Lemma~\ref{lem:omega-general-n} we obtain 
	\[ \cB_n\big(\tfrac{p-2}{p}\cdot n\big)\subseteq \Omega(\theta_1)\star\cdots\star\Omega(\theta_t)=\Omega(\theta). \]
	Since $\theta\in\Irr(P_n)$ was arbitrary, we deduce that $\cB_n\big(\tfrac{p-2}{p}\cdot n\big)\subseteq\Omega(n)$. 
	
	The final limit then follows using the classical result \cite[(1.4)]{EL} that all but $o(|\cP(n)|)$ many partitions of $n$ lie in $\cB_n\big( \sqrt{n}\cdot( \tfrac{\log n}{c}+f(n) ) \big)$, for any function $f$ such that $f(n)\to\infty$ as $n\to\infty$ and where $c$ is a constant.
\end{proof}

The analysis of thin partitions has been a fundamental tool in establishing the value of $m(\theta)$ for every $\theta\in\Irr(P_n)$. We highlight in the following corollary that if $\lambda$ is a thin partition then one can easily read from the length of its first part (or first column) the irreducible constituents of its restriction to $P_n$. Since all the sets $\Omega(\theta)$ are closed under conjugation of partitions, we state for simplicity, and without loss of generality, the result for those thin partitions $\lambda$ with $\lambda_1\geq \ell(\lambda)$. 

\begin{corollary}\label{cor:thin}
	Let $\lambda$ be a thin partition with $\lambda_1\geq\ell(\lambda)$ and let $\theta\in\Irr(P_n)\setminus\{\triv_{P_n}\}$. Then $\lambda\in\Omega(\theta)$ if and only if $\lambda_1\le n-\gamma_0(\theta)-\gamma_1(\theta)$.
\end{corollary}
\begin{proof}
	This follows directly from Theorems~\ref{thm:val-1-arbitrary} and~\ref{thm:arbitrary-val-2}.
\end{proof}

All of the results proved in this article hold for any prime $p\ge 5$. The structure of the sets $\Omega(\theta)$, as we consider all $n\in\N$ and all $\theta\in\Irr(P_n)$, presents a certain uniform behaviour for all primes $p$ greater than or equal to 5, as shown by Theorems~\ref{thm:val-1-arbitrary}, \ref{thm:puncture-arbitrary} and \ref{thm:final-m(theta)}. On the contrary, this does not seem to be the case for the primes $2$ and $3$.

More specifically, when $p=2$, even the structure of $\Omega(\triv_n)$ is not yet completely understood. Nevertheless, recent advances on Sylow branching coefficients for $p=2$ in \cite{GV,LO}, for example, show that it is dissimilar to the structure of $\Omega(\triv_n)$ in the case of odd primes, proven in \cite{GL1} (Theorem~\ref{thm:GL1}).

On the contrary, we believe that the case $p=3$ should be more tractable, but will nevertheless require new ideas and ad hoc techniques, in particular to deal with some intricate combinatorics. We observed a similar situation when trying to describe $\Omega(\theta)$ for any linear character $\theta\in\Irr(P_n)$: the case of $p\ge 5$ was first resolved in \cite{GL2}, while the analysis of the $p=3$ case was completed some time later in \cite{L}. Thus, the $p=3$ analogue of the present article will similarly be the subject of future investigation.

\bigskip

\subsection*{Acknowledgements}
We are grateful to the referees whose comments and suggestions were very helpful and greatly improved our article.
The first author is funded by the European Union – Next Generation EU, M4C1, CUP B53D23009410006, and PRIN 2022 - 2022PSTWLB \textit{Group Theory and Applications}. As a member of the GNSAGA, he is grateful for the support of the {\em Istituto Nazionale di Alta Matematica}. The second author was supported by Emmanuel College, Cambridge, and also thanks the Algebra Research Group at Universit\`a degli Studi di Firenze for their hospitality during research visits where part of this work was done.
This material is based upon work supported by the National Science Foundation under Grant No.~DMS-1928930, while the first author was in residence at the Mathematical Sciences Research Institute in Berkeley, California, during Summer 2023.

\bigskip


\end{document}